\documentclass[12pt]{amsart}
\usepackage{color}
\usepackage[all]{xy}
\usepackage{tikz}
\usetikzlibrary{positioning,shapes,shadows,arrows,snakes}

\usepackage[colorlinks=true,
pagebackref,hyperindex]{hyperref}

\advance \topmargin by -\headheight
\advance \topmargin by -\headsep
\evensidemargin \oddsidemargin
\numberwithin{equation}{section}

\newtheorem{theorem}{Theorem}[section]

\newtheorem{corollary}[theorem]{Corollary}
\newtheorem{lemma}[theorem]{Lemma}
\newtheorem{lemma-definition}[theorem]{Lemma-Definition}

\theoremstyle{definition}
\newtheorem{remark}[theorem]{Remark}
\newtheorem{example}[theorem]{Example}
\newtheorem{definition}[theorem]{Definition}

\newcommand{\vv}{{\bf v}}
\newcommand{\minvecr}{\overrightarrow{\textup{min}}(P_1,\ldots,P_r)}
\newcommand{\minvec}{\overrightarrow{\textup{min}}}
\newcommand{\overunder}[2]{ 
\!\begin{array}{c} 
\scriptstyle{#1}\\[-.1in] 
-\!\!\!-\!\!\!-\\[-.1in] 
\scriptstyle{#2} 
\end{array} 
\! 
} 

\title{Newton polytopes of rank 3 cluster variables}
\author{Kyungyong Lee}
\address{Department of Mathematics, 
University of Alabama, Tuscaloosa, AL 35487,
USA, and School of Mathematics, Korea Institute for Advanced Study, Seoul 02455 Republic of Korea}
\email{kyungyong.lee@ua.edu; klee1@kias.re.kr}
\author{Li Li}
\address{Department of Mathematics
and Statistics,
Oakland University, 
Rochester, MI 48309-4479, USA 
}
\email{li2345@oakland.edu}
\author{Ralf Schiffler}
\address{Department of Mathematics, University of Connecticut, 
Storrs, CT 06269-3009, USA}
\email{schiffler@math.uconn.edu}

\thanks{The first author was supported by the University of Alabama, University of Nebraska--Lincoln, Korea Institute for Advanced Study, and the NSF grant DMS 1800207. The second author was supported by Oakland University, and by the NSA grant H98230-16-1-0303. The third author was supported by the NSF-CAREER grant  DMS-1254567, the NSF grant DMS-1800860 and by the University of Connecticut.}

  

\begin{document}
\begin{abstract}
We characterize the cluster variables of skew-symmetrizable cluster algebras of rank 3 by their Newton polytopes. The Newton polytope of the cluster variable $z$ is the convex hull of the set of all $\mathbf{p}\in\mathbb{Z}^3$ such that the Laurent monomial ${\bf x}^{\mathbf{p}}$ appears with nonzero coefficient in the Laurent expansion of $z$ in the cluster ${\bf x}$. 
We give an explicit construction of the Newton polytope in terms of the exchange matrix and the denominator vector of the cluster variable.

Along the way, we give a new proof  of the fact that denominator vectors of non-initial cluster variables are non-negative in a cluster algebra of arbitrary rank. 
\end{abstract}

\maketitle

\section{Introduction}
Cluster algebras were discovered by Fomin and Zelevinsky in 2001. Since then, it has been shown that they are  related to diverse areas of mathematics such as algebraic geometry, total positivity, quiver representations, string theory, statistical physics models, non-commutative geometry, Teichm\"uller theory,  tropical geometry, KP solitons, discrete integrable systems, quantum mechanics, Lie theory, algebraic combinatorics, WKB analysis, knot theory, number theory, symplectic geometry, and Poisson geometry. 

A cluster algebra is equipped with a set of distinguished generators called cluster variables. These generators are very far from being fully understood. 
Explicit combinatorial formulas that are manifestly positive are known for cluster variables in cluster algebras from surfaces \cite{MSW} and for cluster algebras of rank 2 \cite{LS-rank2}. For skew-symmetric cluster algebras, there is the cluster character formula for the cluster variables \cite{P} as well as the $F$-polynomial formula \cite{DWZ2} and for skew-symmetrizable cluster algebras there is the scattering diagram approach \cite{GHKK}, but  none of these  provide computable formulas.

For a general  cluster algebra, we do know that cluster variables satisfy the Laurent phenomenon \cite{fz-ClusterI} and positivity \cite{LS, GHKK}, namely, every cluster variable $z$ can be written as 
$$
z=\sum_{{\bf p}\in \mathbb{Z}^n} e({\bf p})x^{\bf p},
$$where $e({\bf p})\ge 0$ for all $\mathbf{p}$, and $e({\bf p})> 0$ for finitely many $\mathbf{p}$.

A natural questions is how to describe the set $S(z):=\{\mathbf{p} \ : \  e({\bf p})> 0 \}$. However this can be very hard in general (see Remark~\ref{remark of cor}). More feasible questions would be the following.

\bigskip

\noindent (a) Describe the Newton polytope of $z$ (which is the convex hull of $S(z)$ by definition).

\bigskip

\noindent (b) Find a subset $U(z)\subset \mathbb{R}^n$ such that the condition $S(z)\subset U(z)$ uniquely detects the cluster variable $z$ (up to a scalar) among all elements in the cluster algebra.

\bigskip

These problems have been solved in  \cite{LLZ} for the rank 2 case. In fact, it turns out that the Newton polytope is a solution to (b) . The paper \cite{LLZ} also introduced a so-called greedy basis, which includes all cluster variables, and found a certain support condition that uniquely detects each greedy basis element (up to a scalar) among all elements in the cluster algebra. 
An alternative characterization of greedy elements using a support condition (SC) plays an essential role in the construction of quantum greedy bases of rank 2 cluster algebras \cite{LLRZ}.  
This support condition  was a key ingredient in \cite{CGMMRSW}, where it was shown that, in rank 2,  the greedy basis coincides with the theta basis defined in \cite{GHKK}.

In this paper, we consider the rank 3 case.  We solve problems (a) and (b), and  prove that the Newton polytope of $z$ is a solution to (b). We also generalize the result to quantum cluster variables.
The step from rank 2 to rank 3 is known to be difficult, since one has to  add the dynamics of the exchange matrix to the problem. In rank 2, the mutation is trivial on the level of the exchange matrix. In rank 3 however, except for a few small cases, 
the mutation class of the matrix is infinite and the representation theory of the quiver is wild.

Along the way, we give a new elementary proof  of the fact that denominator vectors of non-initial cluster variables are non-negative in a cluster algebra of arbitrary rank. This was conjectured by Fomin and Zelevinsky in \cite{ClusteralgebraIV} and recently proved by Cao and Li in \cite{CL} using the positivity theorem.

Recently, Fei has studied combinatorics of $F$-polynomials using a representation-theoretic approach \cite{Fei}. In that paper it was shown that the $F$-polynomial of every cluster variable of an acyclic skew-symmetric cluster algebra has saturated support, which means that all lattice points in the Newton polytope of the $F$-polynomial are in the support of the $F$-polynomial. Briefly speaking, in the case of skew-symmetric rank 3 cluster algebras, Fei's result is related to our work in the following sense.  The Newton polytope of a cluster variable (which lies in a plane inside $\mathbb{R}^3$) is a projection of the Newton polytope of the corresponding $F$-polynomial (which is usually $3$-dimensional) under a linear map. The supports of $F$-polynomials are expected to be saturated but cluster variables are not saturated in general.  On the other hand, the Newton polytopes of the $F$-polynomials are difficult to determine but the Newton polytopes of the cluster variables can be explicitly determined. Please see Corollary \ref{cor of main thm}, Remark \ref{remark of cor} and Remark \ref{remark:F-poly} for more details.

The paper is organized as follows. In Section~\ref{sect 1+}, we briefly review the solutions for (a) and (b) for  rank 2 cluster algebras and in Section \ref{sect 2}, we explain our notation and recall several basic facts about cluster algebras. We prove the non-negativity of denominator vectors  in Section \ref{sect den}. Our main theorem is presented in Section \ref{sect 3} and proved in Section \ref{sect 3+}. We then give an example in Section~\ref{sect 4}. Finally, in Section~\ref{sect 5} we prove a quantum analogue of the main theorem.

\section{Rank 2}\label{sect 1+} In this section, we let $B$ be a $2\times 2$ skew-symmetrizable matrix and $\mathcal{A}(B)$ the corresponding cluster algebra with principal coefficients.
\subsection{Greedy basis}
It is proved in \cite{LLZ} that for each rank 2 cluster algebra there exists a so-called greedy basis defined as follows. Let $B=\left[\begin{smallmatrix} 0&b\\c&0\end{smallmatrix}\right]$ denote the exchange matrix. Then for $(a_1,a_2)\in\mathbb{Z}^2$, define $c(p,q)$ for $(p,q) \in \mathbb{Z}_{\geq 0}^2$ recursively by
$c(0,0)=1$,
{\footnotesize
$$
\aligned
&c(p,q)= \max&\left( {\sum_{k=1}^p (-1)^{k-1}
c(p-k,q) \binom{a_2\!-\!cq\!+\!k\!-\!1}{k}}, \right.
&\left. \; {\sum_{k=1}^q
(-1)^{k-1} c(p,q-k) \binom{a_1\!-\!bp\!+\!k\!-\!1}{k}}\right)\\
\endaligned
$$
}
and define the greedy element at $(a_1,a_2)$ as
$$x[a_1,a_2]=\displaystyle\frac{\sum c(p,q)x_1^{bp}x_2^{cq}}{x_1^{a_1}x_2^{a_2}}.$$

Recall that an element of $\mathcal{A}(B)$ is called {\em positive} if its Laurent expansion is positive in every seed. A positive element is \emph{indecomposable} if it cannot be written as a sum of two positive elements. Finally, a basis $\mathcal{B} $ is called \emph{strongly positive} if any product of elements from $\mathcal{B} $ can be expanded as a positive linear combination of elements of $\mathcal{B} $.
\begin{theorem}
 \cite{LLZ}
 The set  $\mathcal{B}=\{x[a_1,a_2]\mid (a_1,a_2)\in \mathbb{Z}^2\}$ is a strongly positive basis for the cluster algebra $\mathcal{A}(B)$. Moreover $\mathcal{B} $ contains all cluster monomials and all elements of $\mathcal{B}$ are indecomposable positive elements of $\mathcal{A}(B)$. $\mathcal{B}$ is called the \emph{greedy basis}. 
\end{theorem}
Here the fact that $\mathcal{B}$ is strongly positive follows from \cite{CGMMRSW}, where it is shown that the greedy basis coincides with the theta function basis defined in \cite{GHKK}.

\subsection{Characterization using support conditions} The following alternative characterization of greedy elements using a support condition (SC) plays an essential role in the construction of  greedy bases of rank 2 quantum cluster algebras \cite{LLRZ}.  

\begin{theorem}
 The coefficients $c(p,q)$ of $x[a_1,a_2]$ are determined by:

\begin{itemize}
\item[(NC)] (Normalization condition) $c(0,0)=1$.
\smallskip

\item[(DC)] (Divisibility condition) 
\smallskip

if $a_2>cq$, then $(1+x)^{a_2-cq}|\sum_i c(i,q)x^i$.
\smallskip

if $a_1>bp$, then $(1+x)^{a_1-bp}|\sum_i c(p,i)x^i$.
\medskip

\item[(SC)] (Support condition) {$c(p,q)=0$ outside the  region given in \cite[Figure 1]{LLRZ}}.
\end{itemize}

Moreover, if $x[a_1,a_2]$ is a cluster variable then condition (SC) becomes $c(p,q)=0$ outside the closed triangle with vertices $(0,0), (a_2,0),(0,a_1)$, as shown in Figure \ref{fig:rank2}. 

\end{theorem}

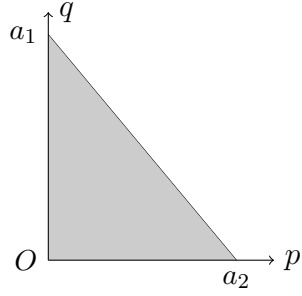
\begin{figure}
\begin{tikzpicture}
\usetikzlibrary{patterns}
\draw (0,3)--(2.5,0);
\fill [black!20]  (0,3)--(2.5,0)--(0,0)--(0,3);
\draw (0,0) node[anchor=east] {\small$O$};
\draw (2.5,0) node[anchor=north] {\small$a_2$};
\draw (0,3) node[anchor=east] {\small$a_1$};
\draw[->] (0,0) -- (3,0)
node[right] {$p$};
\draw[->] (0,0) -- (0,3.3)
node[right] {$q$};
\end{tikzpicture}
\caption{Support of a cluster variable $x[a_1,a_2]$ of a rank 2 cluster algebra}\label{fig:rank2}
\end{figure}

The main theorem of this paper, Theorem \ref{main theorem}, gives a similar characterization for cluster variables for every rank 3 cluster algebra.

\section{Preparation}\label{sect 2}

\subsection{Definition, notations, and facts in cluster algebras}
We recall the definition of skew-symmetrizable cluster algebras with principal coefficients.  

A square matrix $B$ is called skew-symmetrizable if there exists a positive integer diagonal matrix $D$ such that $DB$ is skew-symmetric.  

Let $n$ be a positive integer. 
Let $\mathbb{T}$ denote the $n$-regular tree whose edges are labeled by integers in $\{1,\dots,n\}$ so that each vertex is incident on $n$ edges with distinct labels. The notation $t \overunder{k}{} t'$ means that the edge joining $t$ and $t'$ is labeled by $k$.

Denote by $\mathcal{F}$ the field of rational functions $\mathbb{Q}(x_1,\dots,x_{2n})$. 
To distinguish between mutable variables $x_1,\dots,x_n$ and coefficient variables $x_{n+1},\dots,x_{2n}$, we also use the notation $y_i=x_{n+i}$, for $i=1,\dots,n$. For ${\bf p}=(p_1,\dots,p_n)\in\mathbb{Z}^n$, let $x^{\bf p}=x_1^{p_1}\cdots x_n^{p_n}$, $y^{\bf p}=y_1^{p_1}\cdots y_n^{p_n}$. For $\tilde{\bf p}=(p_1,\dots,p_{2n})\in\mathbb{Z}^{2n}$, let $x^{\tilde{\bf p}}=x_1^{p_1}\cdots x_{2n}^{p_{2n}}$.

Each vertex $t\in \mathbb{T}$ is decorated with a seed 
$\Sigma_t=(\mathbf{x}(t),\tilde{B}(t))$ 
where:
\begin{itemize}
\item $\tilde{B}(t)=[b_{ij}^{(t)}]$ is a $2n\times n$ integer matrix such that the submatrix $B(t)$ formed by the top $n$ rows  of $\tilde{B}(t)$ is skew-symmetrizable. 
\item $\mathbf{x}(t)=\{x_1(t),\dots,x_n(t)\}$ is an $n$-tuple of elements of $\mathcal{F}$.
\end{itemize}

The seeds are defined recursively by mutation as follows. Fix an initial vertex $t_0\in\mathbb{T}$ and define the initial seed as $\mathbf{x}(t_0)=\{x_1,\dots,x_n\}$, 
$\tilde{B}=\begin{bmatrix}B\\ I\end{bmatrix}$ where $B$ is a $n\times n$ skew-symmetrizable matrix and $I$ is the $n\times n$ identify matrix. 

For any real number $a$, let $[a]_+:=\max(a,0)$. Given a seed 
$\Sigma_t=(\mathbf{x}(t),\tilde{B}(t))$ 
and an edge  $t \overunder{k}{} t'$, we define the mutation of $\Sigma_t$ to be 
$\mu_k(\Sigma_{t})=\Sigma_{t'}=(\mathbf{x}(t'),\tilde{B}(t'))$, 
where
$$\begin{array}{rcl}
b^{(t')}_{ij}&=&\left\{\begin{array}{ll}
-b^{(t)}_{ij} &\textrm{ if $i=k$ or $j=k$,}\\[5pt]
b^{(t)}_{ij}+{\rm sgn}(b^{(t)}_{ik})[b^{(t)}_{ik}b^{(t)}_{kj}]_+ &\textrm{ otherwise.}\\
\end{array}\right.\\ \\
x_i(t')&=&\left\{\begin{array}{ll}
x_k(t)^{-1}\Big(\displaystyle\prod_{j=1}^{2n}x_j(t)^{[b^{(t)}_{jk}]_+}+\prod_{j=1}^{2n}x_j(t)^{[-b^{(t)}_{jk}]_+}\Big) &\textrm{ if $i=k$ ,}\\
x_i(t) &\textrm{ otherwise.}\\ 
\end{array}\right.\\\\
\end{array}
$$
Each $x_i(t)$ is called a cluster variable. The cluster algebra $\mathcal{A}$ is the $\mathbb{Q}[x_{n+1}^{\pm},\dots,x_{2n}^{\pm}]$-subalgebra of $\mathcal{F}$ generated by all cluster variables.

For each seed $\Sigma_t$, let $C(t)$ be the $n\times n$ submatrix of $\tilde{B}(t)$ formed by the bottom $n$ rows of $\tilde{B}(t)$. Its columns, ${\bf c}_1(t),\dots,{\bf c}_n(t)$, are called $c$-vectors. 
We need the following theorem proved by Gross-Hacking-Keel-Kontsevich in \cite[Corollary 5.5]{GHKK}.

\begin{theorem}(Sign-coherence of $c$-vectors)\cite{GHKK}\label{Sign-coherence of $c$-vectors}
In a skew-symmetrizable cluster algebra, every $c$-vector ${\bf c}^{(t)}_k=[b^{(t)}_{ik}]_{i=n+1}^{2n}$ is in $\mathbb{Z}_{\ge0}^n\cup\mathbb{Z}_{\le0}^n$. 
\end{theorem}

We will need the following lemma. (As suggested by the referee, it follows immediately from a result of Nakanishi and Zelevinsky.)

\begin{lemma}\label{C-matrix det}
The determinant of $C(t)$ is $1$ or $-1$. As a consequence, the $c$-vectors are linearly independent, and all $c$-vectors are nonzero.
\end{lemma}
\begin{proof}
\iffalse
We prove the statement by induction. The initial $C$-matrix is the identity matrix, so its determinant is $1$. Now assume the lemma is true for $t$. Consider an edge $t \overunder{k}{} t'$ in $\mathbb{T}$. We need to show the lemma holds for $t'$. 
As observed in \cite[(3,2)]{ClusteralgebraIII}, for a fixed index $1\le k\le n$ and a fixed sign $\varepsilon \in\{\pm1\}$, the mutation relation for $\tilde{B}$ is can be rewritten as
$$\tilde{B}(t')=(J_{2n,k}+E_k) \tilde{B}(t) (J_{n,k}+F_k)$$
where $J_{2n,k}$ (resp. $J_{n,k}$) is the diagonal $(2n\times 2n)$-  (resp. $(n\times n)$-) matrix whose diagonal entries are all 1's except for $-1$ in the $k$-th position, 
 $E_k$ is the $(2n\times 2n)$-matrix whose only nonzero entries are $\max(0,-\varepsilon b^{(t)}_{ik})$ at $(i,k)$-entries for all $i$, $F_k$ is the $(n\times n)$-matrix whose only nonzero entries are $\max(0,\varepsilon b^{(t)}_{kj})$ at $(k,j)$-entries for all $j$. 

By Theorem \ref{Sign-coherence of $c$-vectors}, the $c$-vector ${\bf c}_k=[b^{(t)}_{ik}]_{i=n+1}^{2n}$ is either in $\mathbb{Z}_{\ge0}^n\cup\mathbb{Z}_{\le0}^n$.  We choose $\varepsilon$ to be 1 (resp. $-1$) if the $k$-th c-vector ${\bf c}_k$ is in $\mathbb{Z}_{\ge0}^n$ (resp. $\mathbb{Z}_{\le0}^n$). Then $E_k$ is the zero matrix, implying that $C(t')=C(t)(J_{n,k}+F_k)$. Since both $C(t)$ and $(J_{n,k}+F_k)$ have determinant $\pm1$, $C(t)$ also has determinant $\pm1$.
\fi
By \cite[Theorem 1.2]{NZ}, the integer matrix $C(t)$ is the inverse of another integer matrix. So the determinant of $C(t)$ is $1$ or $-1$.
\end{proof}

We also need the following fact, which is shown for skew-symmetric cluster algebras in \cite{DWZ2}, but we could not find a  reference in the skew-symmetrizable setting. A close reference is \cite[Lemma 5.1]{Gupta} which describes a similar idea in the skew-symmetric case.
\begin{lemma}\label{F-polynomial 1}
The $F$-polynomial of every cluster variable of a skew-symmetrizable cluster algebra has constant term 1.
\end{lemma}
\begin{proof}
We prove that the conclusion holds for cluster variables in every seed, by induction on the distance from the current seed to the initial seed in the $n$-regular tree $\mathbb{T}$. 

The statement is true for the initial seed since the $F$-polynomials are 1.  

Assume the conclusion is true for the seed $t$ and unknown for $t'=\mu_k(t)$. The rule of change of $F$-polynomials under mutation is given in 
\cite[Proposition 5.1]{fz-ClusterI}, where the only $F$-polynomial that changes under mutation $\mu_k$ is:
$$F_{k}^{(t')}=\frac{
 {y}^{[{\bf c}^{(t)}_{k}]_+}\prod_{i=1}^n (F_i^{(t)})^{[b^{(t)}_{ik}]_+}
+
{y}^{[-{\bf c}^{(t)}_{k}]_+}\prod_{i=1}^n (F_i^{(t)})^{[-b^{(t)}_{ik}]_+}
}{F_{k}^{(t)}}$$
By sign-coherence of $c$-vectors (Theorem \ref{Sign-coherence of $c$-vectors}), $[{\bf c}^{(t)}_{k}]_+\in\mathbb{Z}_{\ge0}^n\cup\mathbb{Z}_{\le0}^n$. Assuming ${\bf c}^{(t)}_{k}\in\mathbb{Z}_{\ge0}$ (the other case can be proved similarly), we have an equality in $\mathbb{Z}[y_1,\dots,y_n]$:
$$F_{k}^{(t')}F_{k}^{(t)}=
 {y}^{[{\bf c}^{(t)}_{k}]_+}\prod_{i=1}^n (F_i^{(t)})^{[b^{(t)}_{ik}]_+}
+
\prod_{i=1}^n (F_i^{(t)})^{[-b^{(t)}_{ik}]_+}
$$
Moreover  ${y}^{[{\bf c}^{(t)}_{k}]_+}\neq1$ since ${\bf c}^{(t)}_{k}\neq{\bf 0}$. Letting $y_1=\cdots=y_n=0$, we immediately conclude that the constant term of $F_k^{(t')}$ is 1. So the conclusion is true for $t'$.
\end{proof}

For convenience, we introduce simpler notations for rank 3 cluster algebras. Let
$$
\quad B=[b_{ij}]=\begin{bmatrix}0&a&-c'\\-a'&0&b\\c&-b'&0\end{bmatrix},
\quad \tilde{B}=[b_{ij}]=\begin{bmatrix}0&a&-c'\\-a'&0&b\\c&-b'&0\\1&0&0\\0&1&0\\0&0&1\end{bmatrix}$$
The assumption that $\tilde{B}$ be skew-symmetrizable implies the existence of positive integers $\delta_1,\delta_2,\delta_3$ such that $\delta_ib_{ij}=-\delta_jb_{ji}$ for all $i,j$. So we can define
$$\bar{a}:=\delta_1a=\delta_2a',\quad \bar{b}:=\delta_2b=\delta_3b',\quad \bar{c}:=\delta_3c=\delta_1c',\quad \textrm{ thus } DB=\begin{bmatrix} 0&\bar {a}&-\bar{c}\\-\bar{a}&0&\bar{b}\\\bar{c}&-\bar{b}&0\end{bmatrix}$$

We say $B$ is \emph{cyclic} if $a,b,c$ are either all strictly positive or all strictly negative, otherwise $B$ is \emph{acyclic}.

 Note that $aa', bb', cc'\ge0$. Denote the $i$-th column of $B$ by $B_i$.  Then $$\bar{b}B_1+\bar{c}B_2+\bar{a}B_3={\bf 0}.$$
In particular, the vectors $B_1,B_2,B_3\in\mathbb{R}^3$ are coplanar, which is a fact essential for this paper.
\smallskip

{\bf In this paper, we assume the \emph{non-degeneracy} condition that at most one of $a,b,c$ is zero.} 
\begin{remark}\label{remark:nondegenerate}
In the degenerate case when at least two of $a,b,c$ are zero, some of the proofs in this paper may not longer work. On the other hand, the degenerate case is essentially a rank 2 cluster algebra and our main theorems follows directly from \cite{LLRZ}. Note that if the initial $B$-matrix satisfies the non-degeneracy condition, then all the $B$-matrices obtained by mutations satisfy the non-degeneracy condition.
\end{remark}

\subsection{Circular order}\label{subsection: Circular order}
\begin{definition}
We say a sequence of coplanar vectors $v_1,\dots,v_n$ is in \emph{circular order}
 if there is an $\mathbb{R}$-linear isomorphism $\phi$ from a plane containing these vectors to the complex plane $\mathbb{C}$ such that  
$$\phi(v_k)=r_ke^{\sqrt{-1}\theta_k}, \; r_k\ge 0 \quad (1\le k\le n),\quad \textup{and} \quad \theta_1\le \theta_2\le\cdots\le \theta_n\le\theta_1+2\pi.$$

\end{definition}

We introduce the following notation. Given $(d_1,d_2,d_3)\in\mathbb{Z}^3$, define
\begin{equation}\label{v1234} 
\vv_i=d_iB_i\; \textrm{ (for $i=1, 2, 3$)}\quad \textup{and} \quad  \vv_4=-\vv_1-\vv_2-\vv_3.
\end{equation}
The following easy observation is  very useful for proving the circular order condition.

\begin{lemma}\label{criterion of circular order}
Assume $(d_1,d_2,d_3)\in\mathbb{Z}_{\ge0}^3\setminus\{(0,0,0)\}$. Then $B_i, B_j, B_k, \vv_4$ is in circular order if and only if the following conditions hold:

(1) If $B_1, B_2, B_3$ are not in the same half plane, then  $\vv_4=\lambda_1B_i+\lambda_2 B_k$ for some $\lambda_1, \lambda_2\ge0$.  

(2) If $B_1, B_2, B_3$ are strictly in the same half plane (so no two are in opposite directions), then $B_j=\eta_1 B_i+\eta_2 B_k$ for some $\eta_1, \eta_2\ge0$. In particular, if two of $B_1,B_2,B_3$ are in the same direction, then one of them is $B_j$.

(3) If two of $B_1, B_2, B_3$ are in opposite directions, then either $B_i, B_k$ are in opposite directions, or ``$B_i, B_j$ are in opposite directions, $d_k=0$, and $\vv_4, B_i$ are in the same direction'', or ``$B_j, B_k$ are in opposite directions, $d_i=0$,  and $\vv_4, B_k$ are in the same direction''. 
\end{lemma}
\begin{proof} This is an easy observation using Figure \ref{fig:3cases} as reference.
\begin{figure}[ht]
\begin{tikzpicture}
\node at (0,-2) {(1)};
\node[left] at (0,0) {$O$};
\draw[-latex](0,0)--(1,1); \node[right] at (1,1){$B_i$};
\draw[-latex](0,0)--(-.5,1.5); \node[right] at (-.5,1.5){$B_j$};
\draw[-latex](0,0)--(-.5,-1.5); \node[right] at (-.5,-1.5){$B_k$};
\draw[-latex](0,0)--(1.5,-1); \node[right] at (1.5,-1){$\vv_4$};
\end{tikzpicture}
\hspace{2cm}
\begin{tikzpicture}
\node at (0,-2) {(2)};
\node[left] at (0,0) {$O$};
\draw[-latex](0,0)--(1,1); \node[right] at (1,1){$B_i$};
\draw[-latex](0,0)--(-.5,1.5); \node[right] at (-.5,1.5){$B_j$};
\draw[-latex](0,0)--(-1,-.3); \node[left] at (-1,-.3){$B_k$};
\draw[-latex](0,0)--(1.5,-1); \node[right] at (1.5,-1){$\vv_4$};
\end{tikzpicture}

\begin{tikzpicture}
\node at (0,-2) {(3) $B_i, B_k$ opposite};
\node[left] at (0,0) {$O$};
\draw[-latex](0,0)--(1,1); \node[right] at (1,1){$B_i$};
\draw[-latex](0,0)--(-.5,1.5); \node[right] at (-.5,1.5){$B_j$};
\draw[-latex](0,0)--(-.7,-.7); \node[left] at (-.7,-.7){$B_k$};
\draw[-latex](0,0)--(1.5,-1); \node[right] at (1.5,-1){$\vv_4$};
\end{tikzpicture}
\begin{tikzpicture}
\node at (0,-2) {(3) $B_i, B_j$ opposite, $d_k=0$};
\node[left] at (0,0) {$O$};
\draw[-latex](0,0)--(1,1); \node[right] at (1,1){$B_i$};
\draw[-latex](0,0)--(-.7,-.7); \node[right] at (-.7,-.7){$B_j$};
\draw[-latex](0,0)--(1.5,-1); \node[above right] at (1.5,-1){$B_k$};
\draw[-latex](0,0)--(.5,.5); \node[above left] at (.5,.5){$\vv_4$};
\end{tikzpicture}
\begin{tikzpicture}
\node at (0,-2) {(3) $B_j, B_k$ opposite, $d_i=0$};
\node[left] at (0,0) {$O$};
\draw[-latex](0,0)--(.7,.7); \node[right] at (.7,.7){$B_j$};
\draw[-latex](0,0)--(-1,-1); \node[right] at (-1,-1){$B_k$};
\draw[-latex](0,0)--(1.5,-1); \node[above right] at (1.5,-1){$B_i$};
\draw[-latex](0,0)--(-.5,-.5); \node[below right] at (-.5,-.5){$\vv_4$};
\end{tikzpicture}
\caption{The three cases of Lemma \ref{criterion of circular order}}\label{fig:3cases}
\end{figure}
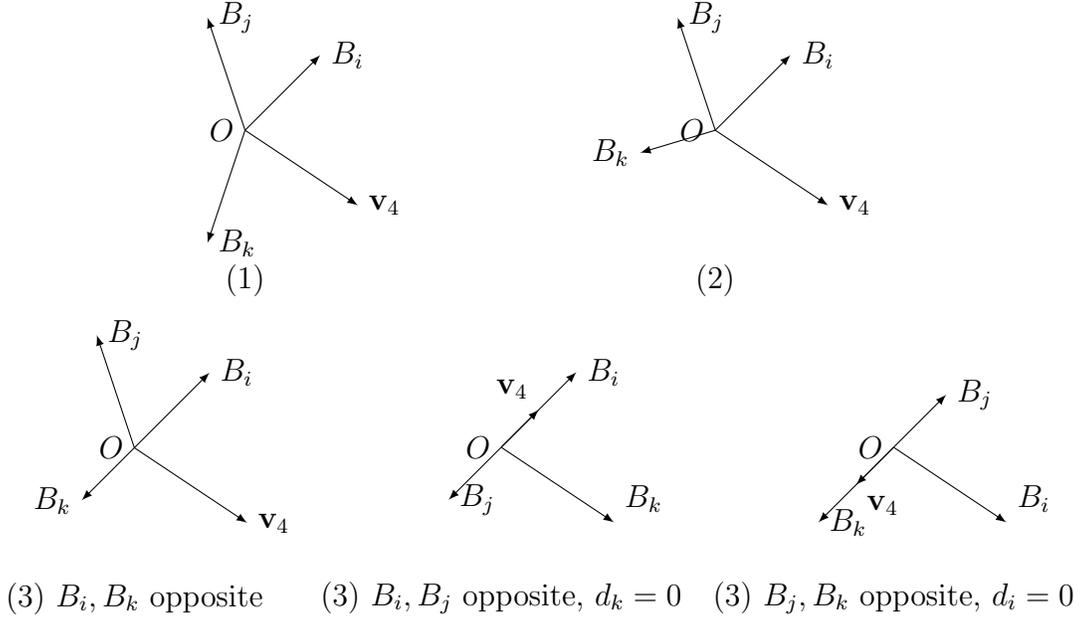
\end{proof}

\subsection{Weakly convex quadrilaterals}
\begin{definition}
Assume four points $P_1,P_2,P_3,P_4\in \mathbb{R}^3$, not necessarily distinct,  are coplanar. 
We call the polygon ${\bf P}=P_1P_2P_3P_4$ a \emph{weakly convex quadrilateral} if the four vectors $\overrightarrow{P_1P_2}$, $\overrightarrow{P_2P_3}$, $\overrightarrow{P_3P_4}$, $\overrightarrow{P_4P_1}$ are in circular order. 

We use convention that $P_{i+4k}=P_i$ for $1\le i\le 4$ and $k\in\mathbb{Z}$.  

If ${\bf P}=P_1P_2P_3P_4$ is a weakly convex quadrilateral,  we denote by $|{\bf P}|\subset\mathbb{R}^3$ the convex hull of $\{P_1, P_2, P_3, P_4\}$.     
\end{definition}
\begin{remark}
By definition, $|{\bf P}|$ is just a bounded convex subset of a real plane inside $\mathbb{R}^3$, while ${\bf P}$ ``remembers" four points in the polygon which are not necessarily distinct and not necessarily the vertices of the polygon. Nevertheless, the set of vertices of $|{\bf P}|$ is a subset of $\{P_1,\dots,P_4\}$. See Figure \ref{fig:limit of convex quadrilaterals} for some examples of ${\bf P}$. 
When we talk about the physical features of a weakly convex quadrilateral ${\bf P}$, where we do not have to pay attention to the four special points $P_1,\dots,P_4$ of ${\bf P}$, we would for simplicity identify ${\bf P}$ with $|{\bf P}|$, the underlying convex set. This applies to phrases like ``a point is contained
in ${\bf P}$",  ``the Newton polytope of a cluster variable is ${\bf P}$",  ``${\bf P}$ is a line segment",  ``${\bf P}$ is a triangle", or ``$\dim{\bf P}=1$''. If we need to use the actual coordinates of the points $P_1,\dots,P_4$, then we distinguish ${\bf P}$ from $|{\bf P}|$. This includes Lemma \ref{shape weak convex}, Lemma \ref{quadrilateral}, and Section \ref{subsection:Proof of (SC')}.

\end{remark}

In the following, we shall give a more explicit description of weakly convex quadrilaterals. Recall that ${\bf P}$ is a (usual) convex quadrilateral if the four vectors $\overrightarrow{P_1P_2}$, $\overrightarrow{P_2P_3}$, $\overrightarrow{P_3P_4}$, $\overrightarrow{P_4P_1}$ are in circular order, all nonzero,  and that no two are in the same direction. In terms of complex numbers, that is to say: if there is an $\mathbb{R}$-linear isomorphism $\phi$ such that
$$\phi(\overrightarrow{P_kP_{k+1}})=r_ke^{\sqrt{-1}\theta_k}, \; r_k> 0 \quad (1\le k\le 4),\quad \textup{and} \quad \theta_1< \theta_2<\theta_3< \theta_4<\theta_1+2\pi.$$

\begin{lemma}\label{shape weak convex}
Let ${\bf P}=P_1P_2P_3P_4$. The following are equivalent:

(1) ${\bf P}$ is weakly convex.

(2) ${\bf P}$ is the limit of a sequence of convex quadrilaterals.

(3) ${\bf P}$ is one of the following:

\begin{itemize}
\item $\dim{\bf P}=0$ (that is, $P_1=P_2=P_3=P_4$, so ${\bf P}$ degenerates to a point);

\item $\dim{\bf P}=1$, ${\bf P}$ is a line segment $P_iP_{i+3}$ and $P_i, P_{i+1}, P_{i+2}, P_{i+3}$ are arranged in order, for some $1\le i\le 4$; 

\item $\dim{\bf P}=1$,  ${\bf P}$ is a line segment $P_iP_{i+2}$ which contains $P_{i+1}$ and $P_{i+3}$, for some $1\le i\le 4$;

\item $\dim{\bf P}=2$, ${\bf P}$ is a triangle $P_iP_{i+1}P_{i+2}$ whose side $P_{i+2}P_i$ contains the point $P_{i+3}$, for some $1\le i\le 4$;

\item $\dim{\bf P}=2$, $P_1P_2P_3P_4$ is a (usual) convex quadrilateral.
\end{itemize}

\end{lemma}

\begin{proof}
(1)$\Leftarrow$(2): Assume the sequence of ${\bf P}^{(j)}=P_1^{(j)}P_2^{(j)}P_3^{(j)}P_4^{(j)}$ has limit ${\bf P}$ and  
$$\phi(P_k^{(j)}P_{k+1}^{(j)})=r_k^{(j)}e^{\sqrt{-1}\theta_k^{(j)}}, \; r_k^{(j)}> 0 \quad (1\le k\le 4),$$ 
and 
$$\theta_1^{(j)}< \theta_2^{(j)}<\theta_3^{(j)}< \theta_4^{(j)}<\theta_1^{(j)}+2\pi.$$
By choosing appropriate angles $\theta_k^{(j)}$ and replace the sequence by a subsequence if necessary, we can assume 
$$\lim_{j\to\infty} r_k^{(j)}=r_k, \quad \lim_{j\to\infty} \theta_k^{(j)}=\theta_k.$$
Then by the property of limits we conclude that $r_k\ge0$ and $\theta_1\le\theta_2\le\theta_3\le\theta_4\le\theta_1+2\pi$. So $\overrightarrow{P_1P_2},\dots,\overrightarrow{P_4P_1}$ are in circular order.

(2)$\Leftarrow$(3): Obvious from the Figure \ref{fig:limit of convex quadrilaterals}.
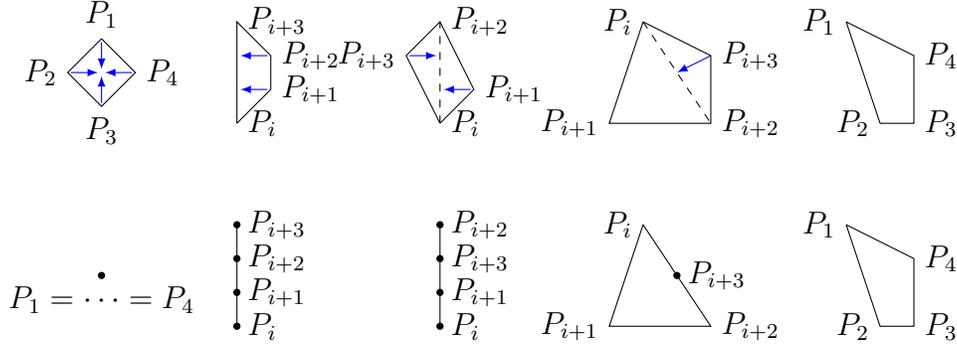
\begin{figure}[ht]
\begin{tikzpicture}[scale=0.45]
\draw (0,1)--(-1,0)--(0,-1)--(1,0)--(0,1);
\draw[-latex,blue] (0,.9)--(0,0.1);
\draw[-latex,blue] (0,-.9)--(0,-0.1);
\draw[-latex,blue] (-.9,0)--(-.1,0);
\draw[-latex,blue] (.9,0)--(.1,0);
\node[above] at (0,1) {$P_1$};
\node[left] at (-1,0) {$P_2$};
\node[below] at (0,-1) {$P_3$};
\node[right] at (1,0) {$P_4$};
\begin{scope}[shift={(0,-6)}]
 \draw[black,fill] (0,0) circle (.5ex);
\node[below] at (0,0) {$P_1=\cdots=P_4$};
\end{scope}
\begin{scope}[shift={(4,-1.5)}]
\draw (0,0)--(1,1)--(1,2)--(0,3)--(0,0);
\draw[-latex,blue] (.9,1)--(.1,1);
\draw[-latex,blue] (.9,2)--(.1,2);
\node[right] at (0,0) {$P_i$};
\node[right] at (1,1) {$P_{i+1}$};
\node[right] at (1,2) {$P_{i+2}$};
\node[right] at (0,3) {$P_{i+3}$};
\end{scope}
\begin{scope}[shift={(4,-7.5)}]
 \draw[black,fill] (0,3) circle (.5ex);\draw (0,0)--(0,3);\node[right] at (0,3) {$P_{i+3}$};
 \draw[black,fill] (0,2) circle (.5ex);\node[right] at (0,2) {$P_{i+2}$};
 \draw[black,fill] (0,1) circle (.5ex); \node[right] at (0,1) {$P_{i+1}$};
 \draw[black,fill] (0,0) circle (.5ex);\node[right] at (0,0) {$P_i$};
\end{scope}
\begin{scope}[shift={(10,-1.5)}]
\draw (0,0)--(1,1)--(0,3)--(-1,2)--(0,0);
\draw[dashed] (0,0)--(0,3);
\draw[-latex,blue] (.9,1)--(.1,1);
\draw[-latex,blue] (-.9,2)--(-.1,2);
\node[right] at (0,0) {$P_i$};
\node[right] at (1,1) {$P_{i+1}$};
\node[left] at (-1,2) {$P_{i+3}$};
\node[right] at (0,3) {$P_{i+2}$};
\end{scope}
\begin{scope}[shift={(10,-7.5)}]
 \draw[black,fill] (0,3) circle (.5ex);\draw (0,0)--(0,3);\node[right] at (0,3) {$P_{i+2}$};
 \draw[black,fill] (0,2) circle (.5ex);\node[right] at (0,2) {$P_{i+3}$};
 \draw[black,fill] (0,1) circle (.5ex); \node[right] at (0,1) {$P_{i+1}$};
 \draw[black,fill] (0,0) circle (.5ex);\node[right] at (0,0) {$P_i$};
\end{scope}
\begin{scope}[shift={(14,-1.5)}]
\draw (2,3)--(1,0)--(4,0)--(4,2)--(2,3);
\draw[dashed] (2,3)--(4,0);
\draw[-latex,blue] (4,2)--(3,1.5);
\node[left] at (2,3) {$P_i$};
\node[left] at (1,0) {$P_{i+1}$};
\node[right] at (4,0) {$P_{i+2}$};
\node[right] at (4,2) {$P_{i+3}$};
\end{scope}
\begin{scope}[shift={(14,-7.5)}]
\draw (2,3)--(1,0)--(4,0)--(2,3);
\node[left] at (2,3) {$P_i$};
\node[left] at (1,0) {$P_{i+1}$};
\node[right] at (4,0) {$P_{i+2}$};
 \draw[black,fill] (3,1.5) circle (.5ex);\node[right] at (3,1.5) {$P_{i+3}$};
\end{scope}
\begin{scope}[shift={(20,-1.5)}]
\draw (2,3)--(3,0)--(4,0)--(4,2)--(2,3);
\node[left] at (2,3) {$P_1$};
\node[left] at (3,0) {$P_2$};
\node[right] at (4,0) {$P_3$};
\node[right] at (4,2) {$P_4$};
\end{scope}
\begin{scope}[shift={(20,-7.5)}]
\draw (2,3)--(3,0)--(4,0)--(4,2)--(2,3);
\node[left] at (2,3) {$P_1$};
\node[left] at (3,0) {$P_2$};
\node[right] at (4,0) {$P_3$};
\node[right] at (4,2) {$P_4$};
\end{scope}
\end{tikzpicture}
\caption{Bottom figures are the limit of the corresponding top convex quadrilaterals for the five cases in Lemma \ref{shape weak convex} (3)}\label{fig:limit of convex quadrilaterals}
\end{figure}

(1)$\Rightarrow$(3): we show the contrapositive. If ${\bf P}$ is not listed in (3), then it must be one of those listed below,  all of which are obviously not weakly convex (see Figure \ref{fig:not weakly convex}).

\begin{itemize}

\item $\dim{\bf P}=1$, ${\bf P}$ is a line segment $P_iP_{i+3}$ and $P_i, P_{i+2}, P_{i+1}, P_{i+3}$ are arranged in order, for some $1\le i\le 4$; 

\item $\dim{\bf P}=2$, ${\bf P}$ is a triangle $P_iP_{i+1}P_{i+2}$ for some $1\le i\le 4$, and the line segment $P_{i+2}P_i$ does not contain the point $P_{i+3}$.
\end{itemize}

\begin{figure}[ht]
\begin{tikzpicture}[scale=0.52]
 \draw[black,fill] (0,3) circle (.5ex);\draw (0,0)--(0,3);\node[right] at (0,3) {$P_{i+3}$};
 \draw[black,fill] (0,2) circle (.5ex);\node[right] at (0,2) {$P_{i+1}$};
 \draw[black,fill] (0,1) circle (.5ex); \node[right] at (0,1) {$P_{i+2}$};
 \draw[black,fill] (0,0) circle (.5ex);\node[right] at (0,0) {$P_i$};
\begin{scope}[shift={(2,0)}]
\draw (2,3)--(1,0)--(4,0)--(2.5,1)--(2,3);
\node[left] at (2,3) {$P_i$};
\node[below] at (1,0) {$P_{i+1}$};
\node[below] at (4,0) {$P_{i+2}$};
 \draw[black,fill] (2.5,1) circle (.5ex);\node[right] at (2.5,1.3) {$P_{i+3}$};
 \end{scope}
\begin{scope}[shift={(8,0)}]
\draw (2,3)--(1,0)--(4,0)--(2.5,0)--(2,3);
\node[left] at (2,3) {$P_i$};
\node[below] at (1,0) {$P_{i+1}$};
\node[below] at (4,0) {$P_{i+2}$};
 \draw[black,fill] (2.5,0) circle (.5ex);\node[above right] at (2.3,0) {$P_{i+3}$};
\end{scope}
\begin{scope}[shift={(14,0)}]
\draw (2,3)--(1,0)--(4,0)--(1.5,1.5)--(2,3);
\node[left] at (2,3) {$P_i$};
\node[below] at (1,0) {$P_{i+1}$};
\node[below] at (4,0) {$P_{i+2}$};
 \draw[black,fill] (1.5,1.5) circle (.5ex);\node[right] at (1.5,1.5) {$P_{i+3}$};
\end{scope}
\begin{scope}[shift={(20,0)}]
\draw (2,3)--(1,0)--(4,0);
\node[left] at (2,3) {$P_i$};
\node[below] at (1,0) {$P_{i+1}=P_{i+3}$};
\node[below] at (4,0) {$P_{i+2}$};
 \draw[black,fill] (1,0) circle (.5ex);
\end{scope}
\end{tikzpicture}
\caption{Quadrilaterals that are not weakly convex}\label{fig:not weakly convex} 
\end{figure}
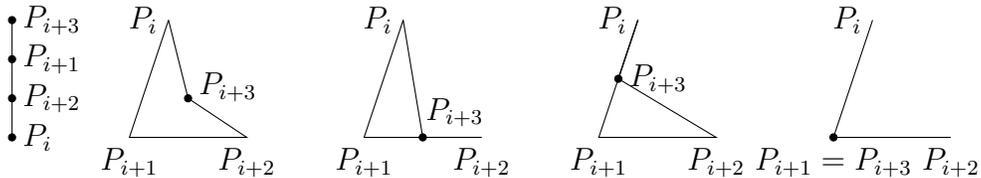
\end{proof}

\iffalse
\begin{lemma}\label{unique2convex}
Given vectors $v_1,\dots,v_n$ in a plane with zero sum, there are at most two distinct $|{\bf P}|$ up to a shift satisfying the following 
condition (where $P_{n+1}=P_1$):

``There exists a permutation $\sigma\in S_n$ such that the polygon ${\bf P}=P_1P_2\cdots P_n$,  with $P_iP_{i+1}=v_{\sigma(i)}$, is weakly convex. ''

Moreover, if there are two distinct $|{\bf P}|$ up to a shift, then one is a shift of the mirror image of the other.
\end{lemma}
\begin{proof}
Without loss of generality we assume $v_1,\dots,v_n$ are all nonzero, in circular order, and $v_1$ is not in the same direction as $v_n$. (This can always be done except the trivial situation when all vectors are 0; in this exceptional case, $|{\bf P}|$ is a point, hence the lemma is obviously true.)

The lemma reduces to the following simple fact: given $n$ numbers $\theta_1\le \theta_2\le \dots\le \theta_n<\theta_1+2\pi$, there is a unique way to arrange them in weakly increasing order, and a unique way to arrange them in weakly decreasing order.
\end{proof}
\fi

\begin{definition}
Given $r$ vectors $P_1,P_2,\ldots,P_r$ in $\mathbb{R}^3$ with coordinates $P_i=\left[\begin{smallmatrix}p_{i1}\\p_{i2}\\p_{i3}\end{smallmatrix}\right]$, define their \emph{minimum vector} 
by
\[\minvecr = \left[\begin{matrix} m_1\\m_2\\m_3\end{matrix}\right],\quad \textup{with } m_i=\textup{min}(p_{1,i},p_{2,i},\ldots,p_{r,i})\]
\end{definition}
For example, if $P_1=\left[\begin{smallmatrix} 1\\2\\3\end{smallmatrix}\right]$ and $P_2=\left[\begin{smallmatrix} 3\\4\\0\end{smallmatrix}\right]$, then $\minvec(P_1,P_2)=\left[\begin{smallmatrix} 1\\2\\0\end{smallmatrix}\right]$.

\subsection{Weakly convex quadrilaterals in rank 3 cluster algebras}\label{Weakly convex quadrilaterals in rank 3 cluster algebras}

We are going to construct a weakly convex quadrilateral ${\bf P_d}$ for all positive integer vectors ${\bf d}\in \mathbb{Z}^3$. Later the vector ${\bf d}$ will be the  denominator vector of a cluster variable. 
For the initial denominator vectors ${\bf d}\in\left\{\left[\begin{smallmatrix} -1\\0\\0\end{smallmatrix}\right],
\left[\begin{smallmatrix} 0\\-1\\0\end{smallmatrix}\right],
\left[\begin{smallmatrix} 0\\0\\-1\end{smallmatrix}\right]
\right\}$, we let ${\bf P_d}$ be the degenerate quadrilateral consisting of the point ${\bf d}$. For all other ${\bf d}$ we have the following lemma. 
Recall that $\vv_i={\bf d}_i B_i$.
\begin{lemma}\label{quadrilateral}
 For all ${\bf d}\in \mathbb{Z}_{\ge 0}^3 \setminus{(0,0,0)}$ there exists a weakly convex quadrilateral 
${\bf P_d}={\bf P}^B_{\bf d}=P_1P_2P_3P_4$ 
 (We use ${\bf P_d}$ when there is no risk of confusion of which matrix $B$ we refer to.)

such that
 
\begin{itemize}
\item[(1)] There is a permutation $(i,j,k)$ of $(1,2,3)$ such that

\begin{itemize}
\item[(1a)] $b_{ij}\ge0, b_{jk}\ge0$. 
\item[(1b)]  $\overrightarrow{P_1P_2}={\bf v}_i$, $\overrightarrow{P_2P_3}={\bf v}_j$, $\overrightarrow{P_3P_4}={\bf v}_k$. 
\item[(1c)]  the four vectors
$B_i,B_j,B_k,\vv_4=-\vv_i-\vv_j-\vv_k$ are in circular order.
\end{itemize}
\smallskip

\item[(2)] $\minvec(P_1,P_2,P_3,P_4)=-{\bf d}$.
\end{itemize}
Moreover, even though  ${\bf P_d}$ may not be uniquely determined by the above condition  (because  $P_1,\dots, P_4$ may vary), the convex hull $|{\bf P_d}|$ is unique.
\end{lemma}

\begin{proof}
We define ${\bf P_d}=P_1P_2P_3P_4$ as follows.
 For an illustration see Example \ref{eg:x621}.
First define a quadrilateral ${\bf\tilde  P}$ by the vertices $\tilde P_1=(0,0,0)$, $\tilde P_2=\tilde P_1+\vv_i 
$, $\tilde P_3=\tilde P_2+\vv_j
$, and  $\tilde P_4=\tilde P_3+\vv_k$.  
  Clearly this quadrilateral satisfies condition (1b), but it does not necessarily satisfy condition (2).
      Let $\tilde{\bf d'}=-\minvec(\tilde P_1,\tilde P_2,\tilde P_3,\tilde P_4)$.
Then define the quadrilateral ${\bf P_d}$ as the translation of the quadrilateral ${\bf \tilde P}$ by ${\bf d}'-{\bf d}$. 
Then ${\bf P_d}$ satisfies conditions (1b) and (2).

Also note  that condition (1c) implies  the  weaker condition that the vectors 
$\overrightarrow{P_1P_2}=\vv_i$, $\overrightarrow{P_2P_3}=\vv_j$, $\overrightarrow{P_3P_4}=\vv_k$, $\overrightarrow{P_4P_1}=\vv_4$ are in circular order. Thus (1c) implies that the quadrilateral is weakly convex.

Thus it remains to show conditions (1a) and (1c). 
We prove these in five separate cases. See Figures \ref{fig:case1}--\ref{fig:case5}.

(Case 1) Suppose $Q$ is acyclic and $abc\neq0$. We may assume without loss of generality that $a,b>0$, and $c<0$. Let $(i,j,k) =(1,2,3)$. Thus condition (1a) holds. Furthermore,
in this case, 
 $B_2=(-\bar{b}/\bar{c})B_1+(-\bar{a}/\bar{c})B_3$  is a positive linear combination of $B_1$ and $B_3$. Hence all three vectors $B_1,B_2,B_3$ lie in the same half plane. Thus Lemma \ref{criterion of circular order}(b) implies that $B_1,B_2, B_3, \vv_4 $ are in circular order. This proves (1c).

 \begin{figure}[ht]
\begin{tikzpicture}
\node[left] at (0,0) {$O$};
\draw[-latex](0,0)--(0,-1); \node[right] at (0,-1){$B_1$};
\draw[-latex](0,0)--(1.5,0); \node[right] at (1.5,0){$B_2$};
\draw[-latex](0,0)--(0.5,1); \node[right] at (0.5,1){$B_3$};
\draw[-latex](0,0)--(-1.5,1); \node[right] at (-1.5,1){$\vv_4$};
\begin{scope}[shift={(4,0)}]
\node[left] at (0,1) {$P_1$};
\draw[-latex](0,1)--(0,-1); \node[left] at (0,-1){$P_2$};
\draw[-latex](0,-1)--(1.5,-1); \node[right] at (1.5,-1){$P_3$};
\draw[-latex](1.5,-1)--(2,0); \node[right] at (2,0){$P_4$};
\draw[-latex](2,0)--(0,1);
\end{scope}
\end{tikzpicture}
\caption{(Case 1)}\label{fig:case1}
\end{figure}

(Case 2) Suppose $Q$ is acyclic, one of $a, b, c$ is zero and the other two have the same sign. That is, $Q$ has two arrows forming a length-2 directed path. Without loss of generality, we may assume that $Q$ is $1\to 2\to 3$, that is, $a,b>0$ and $c=0$.   The vectors $B_1$ and $B_3$ are in opposite directions.

If $d_2>0$, we let  $(i,j,k)=(1,2,3)$. Then condition (1a) holds and condition (1c) holds by Lemma \ref{criterion of circular order}(c). Note that condition (1a) would also hold for $(i,j,k)=(3,1,2)$, or $(2,3,1)$; however, condition (1c) would fail for both. Thus the permutation of $(i,j,k)$, and hence the quadrilateral ${\bf P_d}$, are unique in this case. 

If $d_2=0$, conditions (1a) and (1c) hold for $(i,j,k)=(1,2,3)$ and $(3,1,2)$. In both cases, the quadrilateral degenerates to a line segment in the direction of $B_1$. of length $\max(|\vv_1|,|\vv_3|)$.  The  case (1,2,3) is illustrated in the second picture of Figure \ref{fig:case2}. 
In particular, even though ${\bf P_d}$ may not be uniquely determined by the conditions of the lemma, the convex hull $|{\bf P_d}|$ is unique.

\begin{figure}[ht]
\begin{tikzpicture}
\node at (-3,-1.5) {$d_2>0$};
\node[left] at (0,0) {$O$};
\draw[-latex](0,0)--(0,-1); \node[right] at (0,-1){$B_1$};
\draw[-latex](0,0)--(1.5,0); \node[right] at (1.5,0){$B_2$};
\draw[-latex](0,0)--(0,1); \node[right] at (0,1){$B_3$};
\draw[-latex](0,0)--(-1,-.5); \node[left] at (-1,-.5){$\vv_4$};
\begin{scope}[shift={(5,0)}]
\node[left] at (0,0) {$P_1$};
\draw[-latex](0,0)--(0,-1); \node[left] at (0,-1){$P_2$};
\draw[-latex](0,-1)--(1.5,-1); \node[right] at (1.5,-1){$P_3$};
\draw[-latex](1.5,-1)--(1.5,.75); \node[right] at (1.5,.75){$P_4$};
\draw[-latex](1.5,.75)--(0,0);
\end{scope}
\end{tikzpicture}
\begin{tikzpicture}
\node at (1,-1.5) {$d_2=0$, $\vv_4$ in the opposite direction as $B_1$,  length of $|{\bf P}|$ is $|\vv_1|$};
\node[left] at (0,0) {$O$};
\draw[-latex](0,0)--(0,-1); \node[right] at (0,-1){$B_1$};
\draw[-latex](0,0)--(1.5,0); \node[right] at (1.5,0){$B_2$};
\draw[-latex](0,0)--(0,1); \node[right] at (0,1){$B_3$};
\draw[-latex](0,0)--(0,.6); \node[left] at (0,.6){$\vv_4$};
\begin{scope}[shift={(5,0)}]
\node[left] at (0,1) {$P_1$};
\draw[-latex](0,1)--(0,-1); \node[left] at (0,-1){$P_2=P_3$};
\draw[-latex](0.1,-1)--(0.1,.4); \node[right] at (0,.4){$P_4$};
\draw[-latex](.1,.4)--(0.1,1);
\end{scope}
\end{tikzpicture}

\begin{tikzpicture}
\node at (1,-1.5) {$d_2=0$, $\vv_4$ in the same direction as $B_1$, length of $|{\bf P}|$ is $|\vv_3|$};
\node[left] at (0,0) {$O$};
\draw[-latex](0,0)--(0,-1); \node[right] at (0,-1){$B_1$};
\draw[-latex](0,0)--(1.5,0); \node[right] at (1.5,0){$B_2$};
\draw[-latex](0,0)--(0,1); \node[right] at (0,1){$B_3$};
\draw[-latex](0,0)--(0,-.6); \node[left] at (0,-.6){$\vv_4$};
\begin{scope}[shift={(5,0)}]
\node[left] at (0,1) {$P_4$};
\draw[-latex](0,1)--(0,.4); \node[left] at (0,.4){$P_1$};
\draw[-latex](0,.5)--(0,-1); \node[left] at (0,-1){$P_2=P_3$};
\draw[-latex](.1,-1)--(.1,1);
\end{scope}
\end{tikzpicture}
\caption{(Case 2)}\label{fig:case2}
\end{figure}

(Case 3) Suppose $Q$ is acyclic,  one of $a, b, c$ is zero and the other two have the opposite sign. Then exactly one vertex is adjacent to both the other two, and this vertex is either a sink or a source.

(Case 3a) If this vertex is a sink. Without loss of generality, assume $Q$ is
$1\to 3\leftarrow 2$, that is, $a=0$, $b>0>c$. The vectors  $B_1$ and $B_2$ are in the same direction.
Condition (1a) is satisfied for the two permutations 
 $(i,j,k)=(1,2,3)$ or $(2,1,3)$, and  both satisfy  condition (1c) by Lemma \ref{criterion of circular order}(a). Both cases give the same  $|{\bf P_d}|$, which is triangle with edges $\vv_1+\vv_2,\vv_3,\vv_4$, in that order. 

\begin{figure}[ht]
\begin{tikzpicture}
\node[left] at (0,0) {$O$};
\draw[-latex](0,0)--(0,-1); \node[right] at (0,-1){$B_1$};
\draw[-latex](0,0)--(0,-.6); \node[right] at (0,-.6){$B_2$};
\draw[-latex](0,0)--(1,0); \node[right] at (1,0){$B_3$};
\draw[-latex](0,0)--(-2,1); \node[left] at (-2,1){$\vv_4$};
\begin{scope}[shift={(5,.5)}]
\node[left] at (0,0) {$P_1$};
\draw[-latex](0,0)--(0,-.5); \node[left] at (0,-.5){$P_2$};
\draw[-latex](0,-.5)--(0,-1); \node[left] at (0,-1){$P_3$};
\draw[-latex](0,-1)--(2,-1); \node[right] at (2,-1){$P_4$};
\draw[-latex](2,-1)--(0,0);
\end{scope}
\end{tikzpicture}
\caption{(Case 3a)}\label{fig:case3}
\end{figure}

(Case 3b) If this vertex is a source. Without loss of generality, assume $Q$ is $2 \leftarrow 1\to 3$, that is, $b=0$, $a>0>c$. 
The vectors $B_2$ and $B_3$ are in the same direction.
Condition (1a) is satisfied for the two permutations 
 $(i,j,k)=(1,2,3)$ or $(1,3,2)$, and  both satisfy  condition (1c) by Lemma \ref{criterion of circular order}(a). Both cases give the same  $|{\bf P_d}|$, which is a triangle with edges $\vv_1,\vv_2+\vv_3,\vv_4$, in that order.
 
\begin{figure}[ht]
\begin{tikzpicture}
\node[left] at (0,0) {$O$};
\draw[-latex](0,0)--(0,-1); \node[right] at (0,-1){$B_1$};
\draw[-latex](0,0)--(1,0); \node[right] at (1,0){$B_2$};
\draw[-latex](0,0)--(.5,0); \node[above] at (.5,0){$B_3$};
\draw[-latex](0,0)--(-2,1); \node[left] at (-2,1){$\vv_4$};
\begin{scope}[shift={(5,.5)}]
\node[left] at (0,0) {$P_1$};
\draw[-latex](0,0)--(0,-1); \node[left] at (0,-1){$P_2$};
\draw[-latex](0,-1)--(1,-1); \node[below] at (1,-1){$P_3$};
\draw[-latex](0,-1)--(2,-1); \node[right] at (2,-1){$P_4$};
\draw[-latex](2,-1)--(0,0);
\end{scope}
\end{tikzpicture}
\caption{(Case 3b)}\label{fig:case4}
\end{figure}

(Case 4) Suppose $Q$ is cyclic. Without loss of generality we may assume $a,b,c>0$. Condition (1a)  narrows down the choices of $(i,j,k)$ to $(1,2,3)$, $(2,3,1)$, or $(3,1,2)$. 

\begin{figure}[ht]
\begin{tikzpicture}[scale=0.9]
\node[left] at (0,0) {$O=\vv_4$};
\node[below] at (0,-1.5){$\vv_4=0$};
\draw[-latex](0,0)--(0,-1); \node[right] at (0,-1){$B_1$};
\draw[-latex](0,0)--(1,0); \node[right] at (1,0){$B_2$};
\draw[-latex](0,0)--(-1,1); \node[left] at (-1,1){$B_3$};
\begin{scope}[shift={(4,.5)}]
\node[above] at (0,0) {$P_1=P_4$};
\draw[-latex](0,0)--(0,-1); \node[left] at (0,-1){$P_2$};
\draw[-latex](0,-1)--(1,-1); \node[right] at (1,-1){$P_3$};
\draw[-latex](1,-1)--(0,0);
\node[below] at (0,-2){$(i,j,k)=(1,2,3)$};
\end{scope}
\begin{scope}[shift={(7,.5)}]
\node[above] at (0,0) {$P_3$};
\draw[-latex](0,0)--(0,-1); \node[below] at (0,-1){$P_1=P_4$};
\draw[-latex](0,-1)--(1,-1); \node[right] at (1,-1){$P_2$};
\draw[-latex](1,-1)--(0,0);
\node[below] at (.5,-2){$(i,j,k)=(2,3,1)$};
\end{scope}
\begin{scope}[shift={(10.5,.5)}]
\node[above] at (0,0) {$P_2$};
\draw[-latex](0,0)--(0,-1); \node[left] at (0,-1){$P_3$};
\draw[-latex](0,-1)--(1,-1); \node[below] at (1,-1){$P_1=P_4$};
\draw[-latex](1,-1)--(0,0);
\node[below] at (0.5,-2){$(i,j,k)=(3,1,2)$};
\end{scope}
\end{tikzpicture}
\begin{tikzpicture}[scale=0.8]
\node[left] at (0,0) {$O$};
\node[below] at (0,-1.5){$\vv_4$ between $\vv_1$ and $\vv_2$};
\draw[-latex](0,0)--(0,-1); \node[right] at (0,-1){$B_1$};
\draw[-latex](0,0)--(1,0); \node[right] at (1,0){$B_2$};
\draw[-latex](0,0)--(-1,1); \node[left] at (-1,1){$B_3$};
\draw[-latex](0,0)--(1.5,-.5); \node[left] at (1.5,-1){$\vv_4$};
\begin{scope}[shift={(4,-.5)}]
\node[left] at (0,0) {$P_4$};
\draw[-latex](0,0)--(1.5, -.5); \node[below] at (1.5,-.5){$P_1$};
\draw[-latex](1.5,-.5)--(2.5,-.5); \node[below] at (2.5,-.5){$P_2$};
\draw[-latex](2.5,-.5)--(0,2); \node[left] at (0,2){$P_3$};
\draw[-latex](0,2)--(0,0); 
\node[below] at (1,-1){$(i,j,k)=(2,3,1)$};
\end{scope}
\end{tikzpicture}
\begin{tikzpicture}[scale=0.8]
\node[left] at (0,0) {$O$};
\node[below] at (0,-1.5){$\vv_4$ and $\vv_1$ in the same direction};
\draw[-latex](0,0)--(0,-1); \node[right] at (0,-1){$B_1$};
\draw[-latex](0,0)--(1,0); \node[right] at (1,0){$B_2$};
\draw[-latex](0,0)--(-1,1); \node[left] at (-1,1){$B_3$};
\draw[-latex](0,0)--(0,-.5); \node[left] at (0,-.5){$\vv_4$};
\begin{scope}[shift={(4,-.5)}]
\node[left] at (0,0) {$P_1$};
\draw[-latex](0,0)--(0, -.5); \node[below] at (0,-.5){$P_2$};
\draw[-latex](0,-.5)--(2.5,-.5); \node[below] at (2.5,-.5){$P_3$};
\draw[-latex](2.5,-.5)--(0,2); \node[left] at (0,2){$P_4$};
\draw[-latex](0,2)--(0,0); 
\node[below] at (1,-1){$(i,j,k)=(1,2,3)$};
\end{scope}
\begin{scope}[shift={(8,-.5)}]
\node[left] at (0,1) {$P_4$};
\draw[-latex](0,1)--(0, -.5); \node[below] at (0,-.5){$P_1$};
\draw[-latex](0,-.5)--(2.5,-.5); \node[below] at (2.5,-.5){$P_2$};
\draw[-latex](2.5,-.5)--(0,2); \node[left] at (0,2){$P_3$};
\draw[-latex](0,2)--(0,1); 
\node[below] at (1,-1){$(i,j,k)=(2,3,1)$};
\end{scope}
\end{tikzpicture}
\caption{(Case 4)}\label{fig:case5}
\end{figure}

If $\vv_4=0$, then for all of the above three choices of $(i,j,k)$, we get the same $|{\bf P}|$  which is a triangle with edges $\vv_1,\vv_2,\vv_3$, in that order.

In the following we assume $\vv_4\neq0$. 
If $\vv_4$ is strictly between $\vv_1$ and $\vv_2$ (respectively, $\vv_2$ and $\vv_3$, $\vv_3$ and $\vv_1$), then the circular order condition implies the unique choice $(i,j,k)=(2,3,1)$ (resp. $(3,1,2)$, $(1,2,3)$). If $\vv_4$ is in the same direction as $\vv_1$, then $(i,j,k)=(1,2,3)$ or $(2,3,1)$. But they give the same quadrilateral which degenerates to a (possibly degenerated) triangle with edges $\vv_1+\vv_4,\vv_2,\vv_3$, in that order. Similar argument holds for $\vv_4$ being in the same direction as $\vv_2$ or $\vv_3$.
\end{proof}

\subsection{A substitution lemma}
In the following lemma, we describe the effect of replacing the variable $x_i$ by its mutation $x_i'$.  We use the following notation.
\[\begin{array}{ccc}
\alpha_1= \left[\begin{matrix} -1&0&0 \\ [a']_+&1&0\\ [-c]_+&0&1\end{matrix}\right]
&
\alpha_2= \left[\begin{matrix} 1&[-a]_+&0 \\ 0&-1&0\\ 0&[b']_+&1\end{matrix}\right]
&
\alpha_3= \left[\begin{matrix} 1&0&[c']_+ \\ 0&1&[-b]_+\\ 0&0&-1\end{matrix}\right]
\\ \\
\beta_1= \left[\begin{matrix} -1&0&0 \\ [-a']_+&1&0\\ [c]_+&0&1\end{matrix}\right]
&
\beta_2= \left[\begin{matrix} 1&[a]_+&0 \\ 0&-1&0\\ 0&[-b']_+&1\end{matrix}\right]
&
\beta_3= \left[\begin{matrix} 1&0&[-c']_+ \\ 0&1&[b]_+\\ 0&0&-1\end{matrix}\right]
\\ \\

\end{array}
\]

Recall that a semifield $\mathbb{P}=(\mathbb{P},\oplus,\cdot)$ is an abelian multiplicative group endowed with an auxiliary addition $\oplus:  \mathbb{P}\times\mathbb{P}\to\mathbb{P}$ which is associative, commutative, and $a(b\oplus c)=ab\oplus ac$ for every $a,b,c\in\mathbb{P}$.  Let $\mathbb{ZP}$ be the group ring of $\mathbb{P}$, and $\mathbb{QP}$ the field of fractions of $\mathbb{ZP}$.
\begin{lemma}\label{change of support}
Let ${\bf p}=(p_1,p_2,p_3)$ and $r\ge 0$.
For $i=1,2,3$ let $f_i$ be a Laurent polynomial of the form
\[ f_i= a_0 x^{{\bf p}+b_0 B_i} +a_1 x^{{\bf p}+b_1 B_i} +\cdots +a_n x^{{\bf p}+b_n B_i}
=\sum_{j=0}^n a_j x^{{\bf p}+b_jB_i},
\] 
where $a_0,\dots,a_n\in\mathbb{QP}$, $0=b_0<b_1<\cdots <b_n=r$, so that the exponents in $f_i$ are points on the line segment from ${\bf p}$ to ${\bf p}+rB_i=:{\bf q}$. 

Let $g_i$ be the rational function obtained from $f_i$ by substituting 
\[\begin{array}{c} x_1 \textup{ by } (p^-x_2^{[a']_+}x_3^{[-c]_+}+p^+ x_2^{[-a']_+}x_3^{[c]_+})/x_1 \quad \textup{if $i=1$;}\\
 x_2 \textup{ by }  (p^- x_3^{[b']_+}x_1^{[-a]_+}+p^+ x_3^{[-b']_+}x_1^{[a]_+})/x_2
  \quad \textup{if $i=2$;}\\ 
  x_3 \textup{ by }  (p^- x_1^{[c']_+}x_2^{[-b]_+}+p^+ x_1^{[-c']_+}x_2^{[b]_+})/x_3
  \quad \textup{if $i=3$.}\\
 \end{array}\]
 where $p^-, p^+\in\mathbb{P}$.

If $g_i$ is a Laurent polynomial, then
\[g_i= a'_0 x^{{\bf p}'+b'_0 B_i} +a'_1 x^{{\bf p}'+b'_1 B_i} +\cdots +a'_n x^{{\bf p}'+b'_n B_i}
=\sum_{j=0}^n a'_j x^{{\bf p}'+b'_jB_i},
\] 
where $0=b'_0<b'_1<\cdots <b'_n=r'$, so that the exponents in $g_i$ are points on the line segment from ${\bf p}'$ to ${\bf p}'+r'B_i=:{\bf q}'$.

Moreover  
$r'=r+p_i$, $a'_0=(p^+)^{p_i}a_0$, $a'_{n}=(p^-)^{p_i}a_n$, and
\[{\bf p}'=\alpha_i({\bf p}),
\quad 
{\bf q}'=\beta_i({\bf q}).\]
\end{lemma}

\begin{proof}
By symmetry, it suffices to show the case $i=1$.  We simply write $f,g$ instead of $f_1,g_1$. First assume $a\ge0$ and $ c\le 0$. Then  
\[f=\sum_{j=0}^n a_j x^{{\bf p}+b_jB_1} = x^{\bf p}\sum_{j=0}^n a_j x^{b_jB_1}\]
Note that the last sum does not depend on $x_1$ since $B_1=\left[\begin{smallmatrix} 0\\-a'\\c\end{smallmatrix}\right]$.
Therefore
$$\aligned
&g=\big[x_1^{-1}x_2^{a'}x_3^{-c}(p^-+p^+x^{B_1})\big]^{p_1}x_2^{p_2}x_3^{p_3}\sum_{j=0}^n a_j x^{b_jB_1}
\\
&=x_1^{-p_1}x_2^{p_2+ap_1}x_3^{p_3-cp_1}(p^-+p^+x^{B_1})^{p_1}\sum_{j=0}^n a_j x^{b_jB_1}
\\
&=x_1^{-p_1}x_2^{p_2+ap_1}x_3^{p_3-cp_1}\big((p^-)^{p_1}a_0+
\textup{$\cdots$ terms of intermediate degree $\cdots$} \\
&\phantom{X}+(p^+)^{p_1}a_n x^{(p_1+r)B_1}\big).\\
\endaligned
$$
So $r'=r+p_1$, $a'_0=(p^-)^{p_1}a_0$, $a'_{n'}=(p^+)^{p_1}a_n$, and ${\bf p}'=(-p_1,p_2+a'p_1,p_3-cp_1)=\alpha_1({\bf p})$. Moreover, since ${\bf q}={\bf p}+rB_1$, %(q_1,q_2,q_3)=(p_1,p_2-a'r,p_3+cr)$
 we have ${\bf q}'=(-p_1,p_2+a'p_1,p_3-cp_1)+(p_1+r)(0,-a',c)=(-p_1,p_2-a'r,p_3+cr)=(-q_1,q_2,q_3)=\beta_1({\bf q})$. This completes the proof in the case  $a\ge0$ and $ c\le 0$.

The remaining cases ``$c\ge 0, a\le 0$", ``$a,c\le0$", ``$a,c\ge 0$'' are proved similarly.
\end{proof}

\subsection{A Newton polytope change lemma}
The following lemma describes how the Newton polytopes  change under mutation. It will be used in the proof of the support condition in our main result Theorem \ref{main theorem}.

\begin{lemma}\label{lemma:support change}
Assume that $F$ is a Laurent polynomial in $x_1,x_2,x_3$, and after substituting 
\begin{equation}\label{eq:sub x1}
x_1\mapsto (p^+x_2^{[a']_+}x_3^{[-c]_+}+p^-x_2^{[-a']_+}x_3^{[c]_+})/x_1'
\end{equation} in $F$, we get a Laurent polynomial $G$ in $x_1',x_2,x_3$.
Assume the Newton polytope $R$ of a Laurent polynomial $F$ lies in a plane $S$ parallel to the plane 
$\textrm{span}(B_1,B_2,B_3)$, and $R$ satisfies the following condition:

$R$ is bounded by two (possibly length zero) line segments ${\bf q}_1{\bf q}_2$, ${\bf q}_3{\bf q}_4$, as well as  an ``inflow'' boundary $T_{in}$ joining ${\bf q}_1$ and ${\bf q}_3$ and   an ``outflow'' boundary $T_{out}$ joining ${\bf q}_2$ and ${\bf q}_4$. Each line $\ell$ parallel to $B_1$ intersects at most once with $T_{in}$ and at most once with $T_{out}$, and $\ell$ intersects $T_{in}$ if and only if it intersects $T_{out}$.  So it induces a bijection $\phi: T_{in}\to T_{out}$. Moreover, $\phi({\bf q}_1)={\bf q}_2$, $\phi({\bf q}_3)={\bf q}_4$, and for all $t\in T_{in}$, $\phi(t)\in t+\mathbb{R}_{\ge0}B_1$. 

Now define \[{\bf q}_1'=\beta_1({\bf q}_2), \quad {\bf q}_2'=\alpha_1({\bf q}_1),
\quad  {\bf q}_3'=\beta_1({\bf q}_4),
\quad  {\bf q}_4'=\alpha_1({\bf q}_3), \]
\[ T'_{out}=\alpha_1(T_{in}), 
\quad T'_{in}=\beta_1(T_{out}).\] 

Then the convex hull of the support of $G$  is the region $R'\subset S'$, where $S'$ is a plane parallel to $\textrm{span}(B_1', B_2',B_3')$, such that the following condition holds:

$R'$ is bounded by two line segments ${\bf q}_1'{\bf q}_2'$, ${\bf q}_3'{\bf q}_4'$, as well as  an ``inflow'' boundary $T'_{in}$ joining ${\bf q}_1'$ and ${\bf q}_3'$ and   an ``outflow'' boundary $T'_{out}$ joining ${\bf q}_2'$ and ${\bf q}_4'$. Each line $\ell$ parallel to $B_1'$ intersect at most once with $T'_{in}$ and at most once with $T'_{out}$, and it intersects $T'_{in}$ if and only if it intersects $T'_{out}$.  So it induces a bijection $\phi': T'_{in}\to T'_{out}$. Moreover, $\phi'({\bf q}'_1)={\bf q}'_2$, $\phi({\bf q}'_3)={\bf q}'_4$, and for all $t\in T'_{in}$, $\phi'(t)\in t+\mathbb{R}_{\ge0}B'_1$. 
\begin{figure}[ht]
\begin{tikzpicture}
\draw [->](-2,2) -- (-2,0);\node[left] at (-2,1) {$B_1$};

\draw (-1,3-.5) -- (-1,1) (1,3-.5)--(1,-1);
\draw[scale=1,ultra thick,domain=-1:1,smooth,variable=\x,blue] plot ({\x},{3-.5*\x*\x});
\draw[scale=1,ultra thick,domain=-1:1,smooth,variable=\x,red]  plot ({\x},{\x*\x-\x-1});
\node[left] at (-1,3-.5) {${\bf q}_1$};
\node[left] at (-1,1) {${\bf q}_2$};
\node[right] at (1,3-.5) {${\bf q}_3$};
\node[right] at (1,-1) {${\bf q}_4$};
\node[above] at (0,3) {$T_{in}$};
\node[below left] at (0,-.7) {$T_{out}$};
\draw[very thick](0,2) -- (0,0); \node[ right] at (0,2){${\bf p}$};  \node[right] at (0,0){${\bf q}$};
\draw[dotted, very thick] (0,2) -- (0,3) (0,0)--(0,-1); \node[below right] at (0,3){${\bf p}_0$};  \node[above right] at (0,-1.2){${\bf q}_0$};
\begin{scope}[shift={(5,0)}]
\draw (-1,3-.5-1) -- (-1,1-2-.5) (1,3-.5+1)--(1,1+2-.5);
\draw[scale=1,ultra thick,domain=-1:1,smooth,variable=\x,blue] plot ({\x},{3-.5*\x*\x+\x});
\draw[scale=1,ultra thick,domain=-1:1,smooth,variable=\x,red]  plot ({\x},{\x*\x+2*\x-.5});
\node[left] at (-1,3-.5-1) {${\bf q}_2'$};
\node[left] at (-1,1-2-.5) {${\bf q}_1'$};
\node[right] at (1,3-.5+1) {${\bf q}_4'$};
\node[right] at (1,1+2-.5) {${\bf q}_3'$};
\node[left] at (0,3.2) {$T'_{out}$};
\node[right] at (0.5,0) {$T'_{in}$};
\draw[ultra thick] (0,2) -- (0,0.5); \node[right] at (0,2){${\bf p}'$};  \node[left] at (0,.5){${\bf q}'$};
\draw[dotted, very thick] (0,2) -- (0,3) (0,0.5)--(0,-0.5); \node[right] at (0,3){${\bf p}'_0$};  \node[left] at (0,-0.5){${\bf q}'_0$};
\end{scope}
\end{tikzpicture}
\caption{$R$ and $R'$}\label{R and R'}
\end{figure}
\end{lemma}
\begin{proof}
First note that by the linearity of $\alpha_1$ and $\beta_1$, $R$ is convex if and only if $R'$ is convex.

Denote $F=\sum e(p_1,p_2,p_3)x_1^{p_1}x_2^{p_2}x_3^{p_3}$. For a fixed integer $t$, let 
$$F_t=\sum e(t,p_2,p_3)x_1^{t}x_2^{p_2}x_3^{p_3}=a_0x^{\bf p}+\cdots +a_nx^{\bf q}, \quad {\bf q}={\bf p}+rB_1$$ 
and let 
$$G_{-t}=a'_0x^{{\bf p}'}+\cdots +a_{n'}x^{{\bf q}'}, \quad {\bf q}'={\bf p}'+r'B_1$$ 
 be obtained from $F_t$ by substitution \eqref{eq:sub x1}. It suffices to show that if the support of $F_t$ (which is the segment ${\bf pq}$) is in $R$ if and only if the support of $G_{-t}$ (which is the segment ${\bf p'q'}$)  is in $R'$. See Figure \ref{R and R'}.

Assume the line through ${\bf p}$ parallel to $B_1$ intersects with $T_{in}$ and $T_{out}$ at ${\bf p}_0$ and ${\bf q}_0$, respectively.
Assume the line through ${\bf p}'$ parallel to $B_1$ intersects with $T'_{out}$ and $T'_{in}$ at ${\bf p}'_0$ and ${\bf q}'_0$, respectively. Then 
$$\alpha_1({\bf p})={\bf p'}, \alpha_1({\bf p}_0)={\bf p}'_0, \beta_1({\bf q})={\bf q'}, \beta_1({\bf q}_0)={\bf q}'_0.$$ 
 It suffices to show 

(i) ${\bf p}\in{\bf p}_0+\mathbb{R}_{\ge0}B_1$ if and only if ${\bf p}'\in{\bf p}'_0+\mathbb{R}_{\ge0}B_1$. 

(ii) ${\bf q}\in{\bf q}_0+\mathbb{R}_{\le0}B_1$ if and only if ${\bf q}'\in{\bf q}'_0+\mathbb{R}_{\le0}B_1$.

\noindent Indeed, For (i): since $\alpha_1$ is linear and fixes $B_1$, and ${\bf p}'-{\bf p}'_0\in\mathbb{R}B_1$, we see that ${\bf p}'-{\bf p}_0'=\alpha_1({\bf p}-{\bf p}_0)={\bf p}-{\bf p}_0$. This implies (i). 
And (ii) can be proved similarly. 
\end{proof}

\section{Denominator vectors of non-initial cluster variables are non-negative} \label{sect den}

If $f$ is an element of the ambient field we shall use the notation $f|_t$ for the  expansion of  $f$ in the variables in the seed $\Sigma_t$. For $t=t_0$, we  simply denote $f|_{t_0}$ by $f$.

Recall that the $d$-vector of a cluster variable $z$ is ${\bf d}\in\mathbb{Z}^n$ such that $$z=\frac{N(x_1,\dots,x_n)}{x^{\bf d}}$$ where $N(x_1,\dots,x_n)$ is a polynomial with coefficients in $\mathbb{Z}[y_i^{\pm}]$ which is not divisible by any cluster variable $x_i$ $(1\le i\le n)$. Equivalently, we can describe ${\bf d}$ as follows. Write $z$ as a sum of Laurent monomials as $z=\sum_{{\bf p}\in \mathbb{Z}^n} e({\bf p})x^{\bf p}$, and define the \emph{support} of $z$ as the set 
$$\textup{supp}(z)=\{{\bf p} \,|\, e({\bf p})\neq 0\}.$$
Let $P_1,\dots,P_m$ be the vertices of the convex hull of $\textup{supp}(z)$. Then 
\begin{equation}\label{min formular for d-vector}
{\bf d}=-\minvec \{{\bf p}\mid {{\bf p}\in \textup{supp}(z)}\}=-\minvec(P_1,\dots,P_m)
\end{equation}

It was conjectured in \cite[Conjecture 7.4 (1)]{ClusteralgebraIV} that the d-vector of any non-initial cluster variable is nonnegative, and this conjecture was proved recently in \cite{CL} 
 using positivity. Below, we give an alternative short proof by an elementary argument.
\begin{theorem}\label{d-vector positive}  
Let $\mathcal{A}$ be a skew-symmetrizable cluster algebra of arbitrary rank.
The d-vector of any non-initial cluster variable is nonnegative.
\end{theorem}

For the proof of the theorem we need the following lemma, where we assume principal coefficients. Note that the proof of the lemma does not rely on the positivity of cluster variables.

 \begin{lemma}\label{no need for positivity}
Let $\mathcal{A}$ be a skew-symmetrizable cluster algebra with $n$ mutable variables and $m-n$ frozen variables, and assume that it has principal coefficients at the initial seed (that is $m= 2n$ and the lower half of the extended exchange matrix $\tilde{B}$ at the initial seed is the $n\times n$ identity matrix).
If a cluster variable is a Laurent monomial, that is, of the form $cx^{\tilde{\bf a}}$, where $c\in\mathbb{Q}\setminus\{0\}$ and  ${\tilde{\bf a}}=(a_1,\ldots,a_{m})$,  then

(1)   $a_1,\dots,a_{n}$ are all nonnegative. 

(2)  $c=1$ and ${\tilde{\bf a}}=e_i =(0,\ldots,0,1,0,\ldots,0)$ (with $1$ at the $i$-th coordinate) for some $1\le i\le n$. 

\noindent As a conclusion, the Laurent expansion  (in $\mathbb{Z}[x_1^{\pm},\dots,x_{m}^{\pm}]$) of a non-initial cluster variable has more than one term.
\end{lemma}
\begin{proof}
(1)  If  false, we  assume without loss of generality that  $a_1<0$. Then expanding $cx^{\tilde{\bf a}}$ in the seed
 $\mu_1(\Sigma_{t_0})$, we get $x_1'^{-a_1}cx^{(0,a_2,\dots,a_{m})}/(M_1+M_2)^{-a_1}$, for some monomials $M_1\neq M_2$ (note that the inequality follows from the fact that $\tilde{B}$ does not have any zero column since its lower half is an identity matrix), and this cannot be a Laurent polynomial.

(2) 
 Apparently $c$ is a nonzero integer,  by the Laurent phenomenon \cite{fz-ClusterI}.  Since all cluster variables can be written as subtraction-free expressions, by specializing the initial variables $x_1=\cdots=x_{m}=1$, we see that $c$ is positive.  
Next,  choose any  seed $\Sigma_t$ that contains the cluster variable $cx^{\tilde{\bf a}}$; denote the cluster of this seed by $\{x_1'(=cx^{\tilde{\bf a}}),x_2',\dots,x_n'\}$. For $i=1,\dots,n$, let $f_i$ be the Laurent expansion of $x_i$ in  $\{x_1'(=cx^{\tilde{\bf a}}),x_2',\dots,x_n',x_{n+1},\dots,x_{m}\}$  (so $f_i=x_i$, only  written in the Laurent expansion form to remind us). Then 
\begin{equation}\label{x_1'=}
x_1'=cx^{\tilde{\bf a}}= cf_1^{a_1}\cdots f_n^{a_n}x_{n+1}^{a_{n+1}}\cdots x_{m}^{a_{m}}
\end{equation} This has the following two consequences.

(a) $c=1$. Indeed, substituting $x_1'=\cdots=x_{n}'=x_{n+1}=\dots=x_{m}=1$ in \eqref{x_1'=}, we get $1=c\prod_{i=1}^n f_i^{a_i}(x_1'=\cdots=x_n'=x_{n+1}=\cdots=x_{m}=1)$. Since all factors in the right hand side are positive integers, we must have $c=1$, and  $f_i(x_1'=\cdots=x_n'=x_{n+1}=\cdots=x_{m}=1)=1$ for every $1\le i\le n$ with $a_i>0$.

(b)  For every $1\le i\le n$ such that $a_i> 0$, the Laurent expansion $f_i$ must be a Laurent monomial in $\mathbb{Z}[(x'_1)^{\pm},\dots,(x'_n)^{\pm},x_{n+1}^{\pm},\dots,x_{m}^{\pm}]$ with coefficient 1. To see this, first observe that this Laurent expansion cannot have more than one term, otherwise the right hand side of \eqref{x_1'=} must have more than one term, so cannot equal to $x_1'$, a contradiction. So we can write $f_i=ux'^{{\tilde{\bf b}}^{(i)}}$ for some $u\in\mathbb{Z}$ and ${\tilde{\bf b}}^{(i)}\in\mathbb{Z}^{m}$.  Next, since $f_i(x_1'=\cdots=x_n'=x_{n+1}=\cdots=x_{m}=1)=1$ as in the proof of (a), we must have $u=1$. This proves (b).

Now combine (b) and part (1),  we see that for every $1\le i\le n$ such that $a_i>0$, we must have 
$f_i=x'^{\tilde{\bf b}^{(i)}}$, 
where 
 $\tilde{\bf b}^{(i)}\in\mathbb{Z}_{\ge0}^n\times\mathbb{Z}^{n-m}$.
  Thus 
\begin{equation}\label{x1'}
x'_1=f_1^{a_1}\cdots f_n^{a_n} x_{n+1}^{a_{n+1}}\cdots x_{m}^{a_{m}} =\Big(\prod_{\substack{1\le i\le n\\ a_i>0}} x'^{{(a_i\tilde{\bf b}}^{(i)})}\Big) x_{n+1}^{a_{n+1}}\cdots x_{m}^{a_{m}}.
\end{equation}
 Since $x'_1,\dots,x'_{n},x_{n+1},\dots,x_{m}$ are algebraically independent, the exponents on both sides of \eqref{x1'} must match, that is,
$$(1,0,\ldots,0)=(\sum_{\substack{1\le i\le n\\ a_i>0}} a_i\tilde{\bf b}^{(i)})+(0,\dots,0,a_{n+1},\dots,a_{m}).$$ Only looking at the first $n$ coordinates of the above equality, and letting ${\bf b}^{(i)}\in\mathbb{Z}_{\ge0}^n$ be the first $n$ coordinates of $\tilde{\bf b}^{(i)}$, we have
\begin{equation}\label{10...0=sum}
(1,0,\ldots,0)=(\sum_{\substack{1\le i\le n\\ a_i>0}} a_i{\bf b}^{(i)}).
\end{equation}
Next we observe that ${\bf b}^{(i)}\neq(0,\dots,0)$, since otherwise, $x_i=f_i\in \mathbb{Z}[x_{n+1}^{\pm},\dots,x_{m}^{\pm}]$, which contradicts the assumption that $x_1,\dots,x_{m}$ are algebraically independent.  
This observation together with \eqref{10...0=sum} implies that there is exactly one $1\le i\le n$ such that $a_i>0$, and for this $i$ we have $a_i=1$ and ${\bf b}^{(i)}=(1,0,\dots,0)$. Therefore 
$$x'_1=cx^{\tilde{\bf a}}=x_ix_{n+1}^{a_{n+1}}\cdots x_{m}^{a_{m}}.$$ 

Now we use the assumption that the cluster algebra has principal coefficients at the initial seed. Under this assumption, the $F$-polynomial of $x'_1$ is $x_{n+1}^{a_{n+1}}\cdots x_{m}^{a_{m}}$. On the other hand, by \cite[Proposition 5.2]{ClusteralgebraIV}, the $F$-polynomial is not divisible by any of $x_{n+1},\dots,x_{m}$. This forces $a_{n+1}=\cdots=a_{m}=0$, therefore $x_1'=x_i$ is indeed an initial cluster variable, and  ${\tilde{\bf a}}=e_i$. 
\end{proof}

\begin{remark} (a) For skew-symmetric cluster algebras, this lemma was proved in \cite[Lemma 3.7]{CKLP}. (Although the lemma in that paper is stated only for coefficient-free case, the statement extends to arbitrary coefficients by their Corollary 5.3, which states that a skew-symmetric cluster algebra with arbitrary coefficients has the proper Laurent monomial property.) 

(b) In fact, Lemma \ref{no need for positivity} holds for any skew-symmetrizable cluster algebras with arbitrary coefficients of geometric type (so $m$ can be any integer at least $n$), provided that the extended exchange matrix $\tilde{B}$ does not have any zero column. (If the $k$-th column is zero, then $\mu_k(x_k)=2/x_k$ is not an initial cluster variable but is a Laurent monomial, thus the lemma would fail.)

We can prove this generalization using positivity: 
the proof of Lemma \ref{no need for positivity} still holds except the last paragraph. (Note that to prove Lemma \ref{no need for positivity}(1) in the general setting we need the assumption that $\tilde{B}$ does not have any zero column.) Then by the separation formula \cite[Theorem 3.7]{ClusteralgebraIV}, we can write 
\begin{equation}\label{another x1'}
x_ix_{n+1}^{a_{n+1}}\cdots x_{m}^{a_{m}}=x'_1=\frac{X'_1|_{\mathcal{F}}(x_1,\dots,x_n;y_1,\dots,y_n)}{F_\mathbb{P}(y_1,\dots,y_n)}
\end{equation}
where $y_1,\dots,y_n\in\mathbb{P}$, $\mathbb{P}$ is the tropical semifield generated by $x_{n+1},\dots,x_m$,  $F_\mathbb{P}$ is the $F$-polynomial evaluated in $\mathbb{P}$,  and $X'_1$ is the cluster variable corresponding to $x'_1$, but computed using principal coefficients. Since $x'_1=x_ix_{n+1}^{a_{n+1}}\cdots x_{m}^{a_{m}}$ is a Laurent monomial with coefficient 1,  $X'_1$ must also be a monomial with coefficient 1 (indeed, substituting $x_1=\dots=x_m=1$ in \eqref{another x1'},  we get $X'_1(x_1=\cdots=x_n=y_1=\cdots=y_n=1)=1$;  by the positivity of cluster variables proved by \cite{LS,GHKK}, $X'_1$ is a monomial with coefficient 1).  Thus the $F$-polynomial of $X'_1$ is a monomial with coefficient 1. We can then apply \cite[Proposition 5.2]{ClusteralgebraIV} to conclude that this $F$-polynomial is 1. Now letting $x_1=\cdots=x_n=1$ in  \eqref{another x1'}, we get $a_{n+1}=\cdots=a_m=0$. So $x'_1=x_i$ is an initial cluster variable.
\end{remark}

\begin{proof}[Proof of Theorem \ref{d-vector positive}]
By \cite[(7.7)]{ClusteralgebraIV}, d-vectors do not depend on the coefficients. So we assume the cluster algebra $\mathcal{A}$ has principal coefficients, with rank $n$. 

Assume  that $x[{\bf d}]$ is a non-initial cluster variable of $\mathcal{A}$, and $d_k<0$ for some $1\le k\le n$. Write
\begin{equation}\label{factor x}
x[{\bf d}]=x_k^{-d_k} f.
\end{equation}

Throughout the proof we use the notation  $f|_{\mu_i(t_0)}$ for the expansion of $f$ in the cluster obtained from the initial cluster by mutation in direction $i$. In other words,  $f|_{\mu_i(t_0)}$ is obtained from $f$ by replacing $x_i$ by an expression of the form $(M_1+M_2)/x_i$, where $M_1,M_2$ are monomials.
We claim that, $f$ and $f|_{\mu_i(t_0)}$ (for every $i=1,\dots,n$) are Laurent polynomials;  that is, $f$ is in the upper bound $\mathcal{U}$ of $\mathcal{A}$ associated with the initial seed (see \cite[Definition 1.1]{ClusteralgebraIII}).
 Indeed, $f=x_k^{d_k}\cdot x[{\bf d}]$ is Laurent because it is a product of two Laurent polynomials. For the same reason, $f|_{\mu_i(t_0)}=x_k^{d_k}(x[{\bf d}]|_{\mu_i(t_0)})$ is also Laurent for $i\neq k$, and $f|_{\mu_k(t_0)}$ is Laurent because $f$ does not contain negative powers of $x_k$, that is, $f=\sum_{d\ge0} x_k^dh_d$ where $h_d$ is a Laurent polynomial in $x_1,\dots,x_{k-1},x_{k+1},\dots,x_n$, thus substituting $x_k$ by (binomial)/(monomial) still gives a Laurent polynomial. 

Since our cluster algebra has principal coefficients, the matrix $\tilde{B}$ is of full rank, and therefore 
  \cite[Corollary 1.9]{ClusteralgebraIII} implies that the upper bound $\mathcal{U}$ is equal to the upper cluster algebra $\overline{\mathcal A}$. 

Let $\Sigma_t=(x'_1,\dots,x'_n,\tilde{B}')$ be a seed that contains $x[{\bf d}]=x'_\ell$ with $1\le \ell \le n$.  Rewriting \eqref{factor x} using $\Sigma_t$ as the initial seed, we get
\begin{equation}\label{factor x 2nd}
x'_\ell = (x_k|_t)^{-d_k} (f|_t)
\end{equation}
Since $f$ is in the upper cluster algebra $\overline{\mathcal A}$, $f|_t$ is Laurent in $x'_1,\dots,x'_n,x_{n+1},\dots,x_{2n}$. 

We assert that $x_k|_t$ is equal to some $x'_i$. Otherwise, Lemma \ref{no need for positivity} implies that
 the Laurent expansion of $x_k|_t$ in the seed $\Sigma_t$ must have more than one term, then the right hand side of \eqref{factor x 2nd} must have more than one term, so cannot be equal to the left hand side, which leads to a contradiction. 

If $i\neq\ell$, then \eqref{factor x 2nd} gives $f|_t=x'_\ell/(x'_i)^{-d_k}$. But this cannot be in the upper cluster algebra. In fact, even $x'_\ell/x'_i$ is not in the upper cluster algebra, because if we rewrite it using the seed $\mu_i(\Sigma_t)$, then $x'_i$ is replaced by  $(M_1+M_2)/x_{i}''$ whose numerator is some binomial, and it is obvious that $x'_\ell x''_i/(M_1+M_2)$ is not a Laurent polynomial in $\mu_i(t)$. So we get a contradiction.

Therefore $i=\ell$ and  \eqref{factor x 2nd} gives $f|_t=(x'_\ell)^{1+d_k}$, where $d_k<0$. If $d_k\le -2$, then a similar argument as above gives a contradiction. So $d_k=-1$, and $f=1$, thus $x[{\bf d}]=x_k$ is an initial cluster variable, contradicting the assumption.
\end{proof}

\section{Main Theorem}\label{sect 3} 
In this section we state  our main result. It gives a characterization of the cluster variables of an arbitrary rank 3 cluster algebra in terms of support, normalization and divisibility conditions.

\begin{theorem}\label{main theorem} 
Let $\mathcal{A}$ be a cluster algebra of rank 3 with principal coefficients and let $x[{\bf d}]$
 be a cluster variable of $\mathcal{A}$ with d-vector ${\bf d}$. Let ${\bf P_d}$ be a weakly convex quadrilateral constructed in Lemma \ref{quadrilateral}.
 Then

$$x[{\bf d}]=\sum_{{\bf p}\in\mathbb{Z}^3} e({\bf p}) x^{\bf p}=\sum_{p_1,p_2,p_3} e(p_1,p_2,p_3) x_1^{p_1}x_2^{p_2}x_3^{p_3}$$ 
where $e({\bf p})\in\mathbb{Z}[y_1,y_2,y_3]$
is uniquely characterized by the following conditions.
\begin{itemize}
\item[(SC)] (Support condition) 
The coefficient $e({\bf p})=0$ unless ${\bf p}\in {\bf P_d}$. Equivalently, the Newton polytope of $x[{\bf d}]$ is contained in ${\bf P_d}$.  

\item[(NC)] (Normalization condition) There is precisely one $e({\bf p})$ that has a nonzero constant term, which must be 1. Moreover, the greatest common divisor of all $e({\bf p})$ is 1.

\item[(DC)] (Divisibility condition)
For each $k=1,2,3$ and $m<0$,
$$\big(\prod_{i=1}^3 x_i^{[-b_{ik}]_+}+y_k\prod_{i=1}^3 x_i^{[b_{ik}]_+}\big)^{-m}\textrm{ divides }\sum_{{\bf p}\in\mathbb{Z}^3 : p_k=m} e({\bf p}) x^{\bf p},
$$
 in the sense that the quotient is in $\mathbb{Z}[x_1^{\pm},x_2^{\pm},x_3^{\pm},y_k]$.

\end{itemize}

Moreover, (NC) can be replaced by

\begin{itemize}
\item[(NC')] 

There exists a ${\bf p}\in\mathbb{Z}^3$ such that $e({\bf p})$ has a nonzero constant term. Moreover for each vertex
${\bf p}$ of the convex hull $|{\bf P_d}|$, $e({\bf p})$ is a monomial $y_1^{\alpha_1}y_2^{\alpha_2}y_3^{\alpha_3}$ for some $\alpha_i\in \mathbb{Z}_{\ge0}$. 
\end{itemize}

And (SC) can be replaced by the following stronger condition.

\begin{itemize}
\item[(SC')]  
The Newton polytope of $x[{\bf d}]$ (which, by definition, is the convex hull of the set $\{{\bf p}\,|\, e({\bf p})\neq0\}$) is ${\bf P_d}$.
\end{itemize}

\end{theorem}

The proof of the theorem is given in the next section. For an example see Section \ref{sect 4}.
The theorem has the following consequence on the support of $F$-polynomials: 

\begin{corollary}\label{cor of main thm}
Using the same notation as in Theorem \ref{main theorem}, let ${\bf g}$ be the g-vector of $x[{\bf d}]$. The support of the $F$-polynomial is contained in the following (possibly unbounded) polyhedron:
$${\bf F}_{\bf d}:=\mathbb{R}_{\ge0}^3\cap \varphi_B^{-1}(\{{\bf p}-{\bf g}\; |\; {\bf p}\in{\bf P_d}\})$$
where $\varphi_B:\mathbb{R}^3\to\mathbb{R}^3$ is the linear map ${\bf q}\mapsto B{\bf q}$,  and $\varphi_B^{-1}$ sends a set to its preimage.
\end{corollary}
\begin{proof}
The support is in $\mathbb{R}_{\ge 0}^3$ because the $F$-polynomial is in $\mathbb{Z}[y_1,y_2,y_3]$. Next, by equation
\eqref{eq:F and p} (and using the notation therein), $e({\bf p})\neq0$ if and only if $f_{ijk}\neq0$ for some ${\bf q}=(i, j, k)$ satisfying $B{\bf q}+{\bf g}={\bf p}$, that is $B{\bf q}={\bf p}-{\bf g}$. This implies that the support of the $F$-polynomial is in $ \varphi_B^{-1}(\{{\bf p}-{\bf g}\; |\; {\bf p}\in{\bf P_d}\})$.
\end{proof}

\begin{remark}\label{remark of cor}
 (1) It has been conjectured that the support of the $F$-polynomial of a cluster variable is always saturated, which is proved  in \cite{Fei} for acyclic skew-symmetric cluster algebras. We say that a non-initial cluster variable $z$ is saturated if there are nonnegative integers $d_1,...,d_n$ and a convex polytope ${\bf T}\subset {\bf R}^n$ such that $\text{supp}(z)$ is obtained by translating ${\bf T} \cap \{(\sum_j [b_{1,j}]_+e_j + [-b_{1,j}]_+(d_j-e_j),..., \sum_j [b_{n,j}]_+e_j + [-b_{n,j}]_+(d_j-e_j) ) \ : \ 0\le e_j\le d_j\text{ for all }j\}$. When the initial exchange matrix $B$, which is the top $n\times n$ submatrix of $\tilde{B}$, is of full rank, it is easy to see that  a cluster variable is saturated if and only if the support of the corresponding $F$-polynomial is saturated. When $B$ is not of full rank,  a cluster variable is not necessarily saturated even if the support of the corresponding $F$-polynomial is saturated.  Such an example appears   in the cluster algebra associated to the following acyclic quiver 
 \[\xymatrix@R10pt{&2\ar[rd]\\1\ar[ru]\ar@<2pt>[rr]\ar@<-2pt>[rr]&&3},\]
where the cluster variable  obtained by mutating at 1,2,3,1,2,3 is not saturated. 

(2) As computed at the end of \S\ref{sect 4}, the convex polyhedron ${\bf F}_{\bf d}$ is often not equal to the Newton polytope of the $F$-polynomial. In fact, $\varphi_B^{-1}(\{{\bf p}-{\bf g}\; |\; {\bf p}\in{\bf P_d}\})$ is a union of parallel lines in the direction of $(\bar{b},\bar{a},\bar{c})$ (the vector that spans the kernel of $\varphi_B$). If the intersection of such a line with $\mathbb{R}_{\ge0}^3$ is nonempty, then the intersection is unbounded if and only if $\bar{a},\bar{b},\bar{c}$ are either all in $\mathbb{R}_{\ge0}$ or all in $\mathbb{R}_{\le0}$. As a consequence, ${\bf F}_{\bf d}$ is unbounded for non-acyclic cluster algebras, so in general it only gives a rough upper bound of the Newton polytope of the $F$-polynomial.  
\end{remark}

\section{Proof of Theorem \ref{main theorem}}\label{sect 3+}

We first show uniquenes. Thus we prove that, given ${\bf d}\in\mathbb{Z}_{\ge0}^3\setminus\{(0,0,0)\}$ which is a d-vector of a non-initial cluster variable, then there is only one Laurent polynomial satisfying the condition
 (SC)+(NC)+(DC) (respectively (SC)+(NC')+(DC)).

Indeed, let $f$ be such a Laurent polynomial. Then the Newton polytope of $f$ and $x[{\bf d}]$ are both contained in ${\bf P_d}$. After changing the initial seed to some seed $\Sigma_t$ that contains $x[{\bf d}]$ as a cluster variable (for simplicity, assume it to be $x_{1}(t)$), we see that the Newton polytope of $f|_t$ is equal to the Newton polytope of $x[{\bf d}]|_t$,  which is a one point set ${\bf P_d}=\{(1,0,0)\}$. But then there is obviously only one Laurent polynomial satisfying (SC)+(NC)+(DC) (resp. (SC)+(NC')+(DC)), namely $x_1(t)$. Thus $f|_t=x_1(t)=x[{\bf d}]|_t$, which implies $f=x[{\bf d}]$.  

\medskip 

It remains to show that a cluster variable $x[{\bf d}]$ satisfies all the conditions given in the theorem. This is obviously true for initial cluster variables. So by Lemma \ref{d-vector positive} we can assume that it is non-initial, i.e., ${\bf d}\in\mathbb{Z}_{\ge0}^3$.
We show each condition in a separate subsection.

\subsection{Proof of (DC)}
To prove (DC), we use the universal Laurent phenomenon. We only show (DC) for $k=1$ because the other cases are similar. Define $h(x_2,x_3)=x_1x_1'$. Then
$$\aligned
h(x_2,x_3)&=\prod x_i^{[-b_{ik}]_+}+y_k\prod x_i^{[b_{ik}]_+}=x_2^{[a']_+}x_3^{[-c]_+}+x_2^{[-a']_+}x_3^{[c]_+}y_1\\
&=x_2^{[-a']_+}x_3^{[-c]_+}(x_2^{a'}+x_3^cy_1),
\endaligned$$
where the last identity holds because $[m]_+=m+[-m]_+$.

Denote the Laurent expansion of $x[{\bf d}]|_{\mu_1(t_0)}$ by 
$$x[{\bf d}]|_{\mu_1(t_0)}=\sum_{p_1',p_2',p_3'} e'(p_1',p_2',p_3') (x_1')^{p_1'}x_2^{p_2'}x_3^{p_3'}, \quad \textrm{where $e'(p_1',p_2',p_3')\in\mathbb{Z}[y'_1,y'_2,y'_3]$}$$
and for each $j=1,2,3$, $y'_j=\prod_{i=4}^6x_i^{b_{ij}'}$ is a Laurent polynomial in   $x_4=y_1$, $x_5=y_2$, $x_6=y_3$,  where we denote by $B'=[b_{ij}']$ the B-matrix of the seed $\mu_1(\Sigma_{t_0})$. 
Then we have
$$\aligned
&\sum_{p_1,p_2,p_3} e(p_1,p_2,p_3) x_1^{p_1}x_2^{p_2}x_3^{p_3}
=\sum_{p_1',p_2',p_3'} {e'}({p_1',p_2',p_3'}) ({x_1'})^{p_1'}x_2^{p_2'}x_3^{p_3'}\\
&=\sum_{p_1',p_2',p_3'} {e'}({p_1',p_2',p_3'}) \Big(\frac{h(x_2,x_3)}{x_1}\Big)^{p_1'}x_2^{p_2'}x_3^{p_3'}\\
&=\sum_{p_1',p_2',p_3'} {e'}({p_1',p_2',p_3'}) x_1^{-p_1'}\Big(h(x_2,x_3)^{p_1'}x_2^{p_2'}x_3^{p_3'}\Big)
\endaligned$$
Regard the above as a Laurent polynomial in $\mathbb{Z}[x_2^{\pm},x_3^{\pm}][x_1^{\pm}]$, that is, as a one-variable Laurent polynomial in $x_1$; then, for a fixed $p_1<0$, take the coefficient of $x_1^{p_1}$ on both ends of the above equalities (so we should take $p_1'=-p_1$ on the right hand side). We then get an equality
$$\sum_{p_2,p_3} e(p_1,p_2,p_3) x_2^{p_2}x_3^{p_3} = h(x_2,x_3)^{-p_1}\sum_{p_2',p_3'} e'(-p_1,p_2',p_3')x_2^{p_2'}x_3^{p_3'}$$
Thus the left hand side is divisible by $h(x_2,x_3)^{-p_1}$, which is equivalent to the condition  (DC) for $k=1$.

From the proof we can also conclude that, if a Laurent polynomial satisfies (DC), then it is in the upper bound/upper cluster algebra $\mathcal{U}=\overline{\mathcal{A}}$. 
Indeed, (DC) implies that, $x[{\bf d}]$ is in the upper bound $U(\Sigma_{t_0})$, where
$\Sigma_{t_0}$ is the initial seed. Since we assume $B$
is skew-symmetrizable, $\Sigma_{t_0}$ is totally mutable. Moreover, $\Sigma_{t_0}$
is coprime because $\tilde{B}$ is full rank \cite[Proposition 1.8]{ClusteralgebraIII}. Therefore 
\cite[Corollary 1.7]{ClusteralgebraIII} implies that $U(\Sigma_{t_0})$ is equal to the upper cluster algebra
$\bar{\mathcal{A}}(\Sigma_{t_0}).$
\medskip

\subsection{Proof of (NC)}  It is shown in \cite{ClusteralgebraIV} that $x[{\bf d}]$ can be expressed by its F-polynomial $F(y_1,y_2,y_3)=\sum_{i,j,k\ge0} f_{ijk}y_1^iy_2^jy_3^k$ as follows (where ${\bf g}=(g_1,g_2,g_3)$ is its g-vector)
$$\aligned
x[{\bf d}]&=x^{\bf g} F(y_1x_2^{-a'}x_3^c,y_2x_1^ax_3^{-b'}\!\!,y_3x_1^{-c'}x_2^b)\\
&=
x_1^{g_1}x_2^{g_2}x_3^{g_3} \!\!\sum_{i,j,k\ge0} \!\!f_{ijk}(y_1x_2^{-a'}x_3^c)^i(y_2x_1^ax_3^{-b'})^j(y_3x_1^{-c'}x_2^b)^k.
\endaligned$$
So
\begin{equation}\label{eq:F and p}
e({\bf p})=\sum f_{ijk}y_1^iy_2^jy_3^k,
\quad \textrm{ where $i,j,k\ge0$ satisfy }
B\begin{bmatrix}i\\j\\k\end{bmatrix}+{\bf g}={\bf p}
\end{equation}
Since the constant term of the F-polynomial of any cluster variable is 1 (Lemma \ref{F-polynomial 1}), there is only one $e({\bf p})$ which has a nonzero constant term, which must be 1. 

Now assume the greatest common divisor of all $e({\bf p})$, which exists uniquely up to sign, is $h\in\mathbb{Z}[y_1,y_2,y_3]$ and $h\neq \pm1$. Since one of $e({\bf p})$ has constant term 1, we can choose $h$ to have constant term 1. Thus $h$ has at least two terms. Define $$X=x[{\bf d}]/h.$$ 

We observe that  $X$ still satisfies (DC), thus is in the upper cluster algebra $\overline{\mathcal{A}}$. Indeed, for each  $m<0$, denote 
$$Y=
\Big({\sum_{{\bf p}\in\mathbb{Z}^3 : p_k=m} e({\bf p}) x^{\bf p}}\Big)
\Big/
{\big(\prod_{i=1}^3 x_i^{[-b_{ik}]_+}+y_k\prod_{i=1}^3 x_i^{[b_{ik}]_+}\big)^{-m}}. 
$$
Then $Y\in \mathbb{Z}[x_1^{\pm},x_2^{\pm},x_3^{\pm},y_k]$ since $x[{\bf d}]$ satisfies (DC). We need to show that $Y/h$ is also in $\mathbb{Z}[x_1^{\pm},x_2^{\pm},x_3^{\pm},y_k]$. Since $\mathbb{Z}[x_1^{\pm},x_2^{\pm},x_3^{\pm},y_k]$ is a UFD and $h$ divides the numerator of $Y$, it suffices to show that $h$ is relatively prime to the denominator of $Y$, or equivalently, show that $h$ is relatively prime to $\prod_{i=1}^3 x_i^{[-b_{ik}]_+}+y_k\prod_{i=1}^3 x_i^{[b_{ik}]_+}$. This is false only if $h=1+y_k$ and $b_{1k}=b_{2k}=b_{3k}=0$, which will not happen because we assume the $B$-matrix is non-degenerate (see Remark \ref{remark:nondegenerate}).

 Similar to the proof of Lemma \ref{d-vector positive}, 
let $\Sigma_t=(x_1',\dots,x_n', y_1',\dots,y_n', B')$ be a seed that contains $x[{\bf d}]=x'_\ell$.  Then
\begin{equation}\label{x'=Xh}
x'_\ell=(X|_{t})(h|_{t})
\end{equation}

We claim that $h|_{t}$, written as a Laurent polynomial in $y_1',y_2',y_3'$, has the same number of terms as $h$ written as a Laurent polynomial in $y_1,y_2,y_3$. Indeed, $y_i'=y^{{\bf c}_i}$ where ${\bf c}_i$ are the $c$-vectors. By Lemma \ref{C-matrix det}, the $c$-vectors ${\bf c}_1, {\bf c}_2, {\bf c}_3\in\mathbb{Z}^3$ are linearly independent, so distinct monomials in $y_1',y_2',y_3'$ convert to distinct monomials in $y_1,y_2,y_3$.

By the above claim, $h|_{t}$ is a Laurent polynomial with at least two terms. So the right hand side of \eqref{x'=Xh} has at least two terms, but the left hand side has only one term,  a contradiction. Therefore the greatest common divisor of all $e({\bf p})$ is 1.

\subsection{Proof of (SC')}\label{subsection:Proof of (SC')} Proving (SC') also proves (SC).
 Let $x=x[{\bf d}]$ be a cluster variable, expressed as a Laurent
polynomial in the initial seed $\Sigma_{t_0}.$
Assume that $\Sigma_{t_0},\Sigma_1,\Sigma_2,...,\Sigma_m$ is a sequence of
mutations of seeds, and that  $x$ is a cluster variable in $\Sigma_m$. We want to show that the condition (SC') holds for $x$ over the seed $\Sigma_{t_0}$. We use induction on $m$. 

If $m=0$ then $x$ is an initial cluster variable and (SC') holds by definition of ${\bf P_d}$. (Recall that as exceptional cases, ${\bf P_d}$ is defined for ${\bf d}=(-1,0,0), (0,-1,0), (0,0,1)$ at the beginning of \S\ref{Weakly convex quadrilaterals in rank 3 cluster algebras}).
For the induction step, we need to show that the quadrilateral  ${\bf P_d} $ is compatible with the mutation. 
This is the longest part of the proof, consisting of a case-by-case computation of the boundary of the quadrilaterals. Without loss of generality, we only need to discuss the cases described in  Lemma \ref{quadrilateral}.

\subsubsection{Proof of (SC') Case 1} Assume $a,b>0$ and $c<0$.  
Using Lemma \ref{quadrilateral} we obtain the quadrilateral ${\bf P_d}$ having the following vertices:

 $$\aligned
 &P_1=(-d_1,-d_2+{a'}d_1,-d_3-cd_1+{b'}d_2),\\
 &P_2=(-d_1,-d_2,-d_3+{b'}d_2),\\
 &P_3=(-d_1+ad_2,-d_2,-d_3),\\
 &P_4=(-d_1+ad_2-{c'}d_3,-d_2+bd_3,-d_3).
 \endaligned
 $$

\medskip

\noindent\underline{We show that the quadrilateral changes as expected under the mutation $\mu_1$:}
More precisely, by changing the initial seed from $t_0$ to $\mu_1(t_0)$, we substitute $x_1$ by $(p_1^+x_2^{a'}x_3^{-c}+p_1^-)/x_1'$ in $x[{\bf d}]$, and get a cluster variable $x'[{\bf d}']=\sum e'({\bf p})x'^{\bf p}=\sum e'(p_1,p_2,p_3)(x_1')^{p_1}x_2^{p_2}x_3^{p_3}$ with $d$-vector ${\bf d}'$. Then we need to show that the convex hull of  the set $\{{\bf p}| e'({\bf p})\neq 0\}$ is $|{\bf P}^{B'}_{{\bf d}'}|$, where $B'$ is as follows (note that $a,a'>0$ and $c,c'<0$ by assumption):
$$ 
B':=\mu_1(B)
=\begin{bmatrix}0&-a&c'\\ a'&0&b+{\rm sgn}(a')[(-a')(-c')]_+\\-c&-b'+{\rm sgn}(c)[ca]_+&0\end{bmatrix}
=\begin{bmatrix}0&-a&c'\\a'&0&b\\-c&-b'&0\end{bmatrix}.$$
\smallskip

First, we use Lemma \ref{lemma:support change} to determine the convex hull of $\{{\bf p}| e'({\bf p})\neq 0\}$. 
Let $T_{in}$ be the segment $P_1P_4$, $T_{out}$ be the  polygonal chain\footnote{A  polygonal chain $P_1P_2\cdots P_n$ is a curve consisting of line segments connecting the consecutive vertices $P_i$ and $P_{i+1}$ for $i=1,\dots,n-1$.} $P_2P_3P_4$, and define points $P'_1,\dots,P'_4$ to satisfy  $\alpha_1(P_1P_4)=P'_1P'_4$ and $\beta_1(P_2P_3P_4)=P'_1P'_2P'_3$, that is,
$$\aligned
&{P'_1}=\alpha_1(P_1)=\beta_1(P_2)=(d_1,-d_2,-d_3+{b'}d_2),\\
&{P'_2}=\beta_1(P_3)=(d_1-ad_2,-d_2,-d_3),\\
&{P'_3=\beta_1(P_4)=(d_1-ad_2+c'd_3,-d_2+bd_3,-d_3),}\\
&{P'_4=\alpha_1(P_4)=(d_1-ad_2+c'd_3,-a'd_1+(aa'-1)d_2+(b-a'c')d_3},\\
&\quad\quad{cd_1-acd_2+(cc'-1)d_3)}.
\endaligned
$$ 
Then Lemma \ref{lemma:support change} guarantees that convex hull of the set $\{{\bf p}| e'({\bf p})\neq 0\}$  is $|P'_1P'_2P'_3P'_4|$. (See Figure \ref{fig:SCcase1}.)
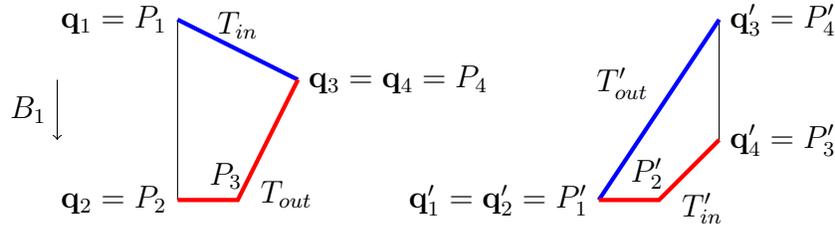
\begin{figure}[ht]
\begin{tikzpicture}[scale=.8]
\draw [->](-3,2) -- (-3,1);\node[left] at (-3,1.5) {$B_1$};
\draw (-1,3) -- (-1,0);
\draw[scale=1,ultra thick,blue] (-1,3)--(1,2);
\draw[scale=1,ultra thick,red] (-1,0)--(0,0)--(1,2);
\node[left] at (-1,3) {${\bf q}_1=P_1$};
\node[left] at (-1,0) {${\bf q}_2=P_2$};
\node[right] at (1,2) {${\bf q}_3={\bf q}_4=P_4$};
\node[above] at (-.2,0) {$P_3$};
\node[above] at (0,2.5) {$T_{in}$};
\node[below right] at (0.2,0.5) {$T_{out}$};
\begin{scope}[shift={(7,0)}]
\draw (1,3) -- (1,1);
\draw[scale=1,ultra thick,blue] (-1,0)--(1,3);
\draw[scale=1,ultra thick,red] (-1,0)--(0,0)--(1,1);
\node[left] at (-1,0) {${\bf q}'_1={\bf q}'_2=P'_1$};
\node[right] at (1,3) {${\bf q}'_3=P'_4$};
\node[right] at (1,1) {${\bf q}'_4=P'_3$};
\node[above] at (-.2,0) {$P'_2$};
\node[above left] at (0,1.5) {$T'_{out}$};
\node[below right] at (0.2,0.3) {$T'_{in}$};
\end{scope}
\end{tikzpicture}
\caption{(SC'), Case 1, $\mu_1$. (Left: projection to $xy$-plane. Right: projection to $xy$-plane followed by reflection about $y$-axis.)}\label{fig:SCcase1}
\end{figure}
\smallskip

Next, we explicitly determine ${\bf d}'$. By \eqref{min formular for d-vector}, ${\bf d}'$ is equal to the $-\minvec$ of the vertices of the convex hull  $|P'_1P'_2P'_3P'_4|$, therefore 
\begin{equation}\label{minvec for P'_1,...}
{\bf d}'=-\minvec(P'_1,\dots,P'_4)
\end{equation}
 Thus
\[\begin{array}{rcl}
d'_1&=&-\min(d_1,d_1-ad_2,d_1-ad_2+c'd_3)=ad_2-c'd_3-d_1,\\
d'_2&=&-\min(-d_2,-d_2+bd_3,-a'd_1+(aa'-1)d_2+(b-a'c')d_3)\\
&=&-\min(-d_2,-d_2+bd_3,-d_2+a'd'_1+bd_3) \ = \ d_2,
\\
d'_3&=&-\min(-d_3+b'd_2,-d_3,cd_1-acd_2+(cc'-1)d_3)\\
&=&-\min(-d_3+b'd_2,-d_3,-d_3-cd'_1) \ =\ d_3\\
\end{array}\]

Lastly, we show that $x'[{\bf d}']$ satisfies (SC'), that is, the convex hull of the set $\{{\bf p}| e'({\bf p})\neq 0\}$ is equal to $|{\bf P}^{B'}_{{\bf d}'}|$, or equivalently,
$$|P'_1P'_2P'_3P'_4|=|{\bf P}^{B'}_{{\bf d}'}|$$
If $d'_1<0$, then by Lemma \ref{d-vector positive}, ${\bf d}'=(-1,0,0)$, and (SC') is trivially true. So in the following we assume $d'_1\ge0$.

We shall show that we can actually take ${\bf P}^{B'}_{{\bf d}'}=P'_1P'_2P'_3P'_4$, that is, $P'_1,\dots,P'_4$ satisfy the two conditions in Lemma \ref{quadrilateral} (recall that ${\bf P}^{B'}_{{\bf d}'}$  may not be unique but its convex hull $|{\bf P}^{B'}_{{\bf d}'}|$ is). 

The condition (2) follows from \eqref{minvec for P'_1,...}. 

The condition  (1) holds for $(i,j,k)=(2,3,1)$. Indeed:

For (1a), we have $b'_{23}=b\ge0$ and $b'_{31}=-c\ge0$.

For (1b), we have  $\vv'_i=d'_iB'_i$ for $i=1,2,3$, $\vv'_4=-\vv'_1-\vv'_2-\vv'_3=(ad_2-c'd_3, a'd_1-aa'd_2+(a'c'-b)d_3,-cd_1+(b'+ac)d_2-cc'd_3)$. It is straightforward to check
$P'_1P'_2=\vv'_2$, $P'_2P'_3=\vv'_3$, $P'_3P'_4=\vv'_1$.

For (1c), we have  $B'_2, B'_3, B'_1, \vv'_4$ are in circular order because $B'_2, B'_3, B'_1$ are strictly in the same half plane, and
$$DB'\begin{bmatrix}-\bar{b}\\\bar{c}\\\bar{a}\end{bmatrix}
=\begin{bmatrix}0&-\bar{a}&\bar{c}\\ \bar{a}&0&\bar{b}\\-\bar{c}&-\bar{b}&0\end{bmatrix}\begin{bmatrix}-\bar{b}\\\bar{c}\\\bar{a}\end{bmatrix}={\bf 0}
\Rightarrow
B'\begin{bmatrix}-\bar{b}\\\bar{c}\\\bar{a}\end{bmatrix}={\bf 0}
\Rightarrow -\bar{b}B'_1+\bar{c}B'_2+\bar{a}B'_3={\bf 0}
$$
implies that  $B'_3=(-\bar{c}/ \bar{a})B'_2+(\bar{b}/\bar{a})B'_1$ where both coefficients are positive. So (1c) follows from Lemma \ref{criterion of circular order}.

This completes the proof that  the quadrilateral changes as expected under the mutation $\mu_1$.

\medskip

The rest of the proof is similar to the above discussion of the quadrilateral change after $\mu_1$. For this reason we simply point out the difference. 
\medskip

\noindent\underline{To show that the quadrilateral changes as expected under the mutation $\mu_2$:}

Substitute $x_2$ by $(p_2^+x_3^{b}+p_2^-x_1^{a'})/x_2$ in $x[{\bf d}]$, and get $x'[{\bf d}']$. Define
$$ 
B':=\mu_2(B)=\begin{bmatrix}0&-a&ab-c'\\a'&0&-b\\c-a'b'&b'&0\end{bmatrix}
$$
Note that
$$DB'\begin{bmatrix}\bar{a}\bar{b}/\delta_2-\bar{c}\\\bar{a}\end{bmatrix}
=\begin{bmatrix}0&-\bar{a}&\bar{a}\bar{b}/\delta_2-\bar{c}\\ \bar{a}&0&-\bar{b}\\\bar{c}-\bar{a}\bar{b}/\delta_2&\bar{b}&0\end{bmatrix}\begin{bmatrix}\bar{b}\\ 
\bar{a}\bar{b}/\delta_2-\bar{c}\\\bar{a}\end{bmatrix}={\bf 0}
$$
implies $ \bar{b}B'_1+(\bar{a}\bar{b}/\delta_2-\bar{c})B'_2+\bar{a}B'_3={\bf 0}$.

The $y$-coordinate of $P_1$ is $-d_2+ad_1$ and the $y$-coordinate of $P_4$ is $-d_2+bd_3$; thus their difference is  $-a'd_1+bd_3$. We will distinguish three cases according to the sign of this difference. 

(i) Suppose $-a'd_1+bd_3<0$. Geometrically, it means that $P_4$ is strictly lower than $P_1$ after projection to $xy$-plane; see Figure \ref{fig:SCcase1mu2i}. Since $T_{in}$ and $T_{out}$ are taken relative to $B_2$ in this case, we have $T_{in}=P_1P_2$ and $T_{out}=P_1P_4P_3$.
Thus in the notation of Lemma \ref{lemma:support change}, we have 
${\bf q}_1=P_1$,
${\bf q}_2=P_1$,
${\bf q}_3=P_2$,
${\bf q}_4=P_3$,
$T_{out}'=\alpha_2(T_{in})$,
$T_{in}'=\beta_2(T_{out})$. 
Therefore Lemma \ref{lemma:support change} implies
$$\aligned
&{P'_1=\beta_2(P_1)=((aa'-1)d_1-ad_2,d_2-a'd_1,-d_3-cd_1+b'd_2),}\\
&{P'_2=\alpha_2(P_1)=(-d_1,d_2-a'd_1,-d_3+(a'b'-c)d_1),}\\
&{P'_3=\alpha_2(P_2)=\beta_2(P_3)=(-d_1,d_2,-d_3),}\\
&{P'_4=\beta_2(P_4)=(-d_1+(ab-c')d_3,d_2-bd_3,-d_3).}
\endaligned
$$ 

\begin{figure}[ht]
\begin{tikzpicture}
\draw [->](-3,1) -- (-2,1);\node[above] at (-2.5,1) {$B_2$};
\draw[scale=1,ultra thick,blue]  (-1,3) -- (-1,0);
\draw(-1,0)--(0,0);
\draw[scale=1,ultra thick,red] (0,0)--(1,2)--(-1,3);
\node[left] at (-1,3) {$P_1$};
\node[left] at (-1,0) {$P_2$};
\node[right] at (1,2) {$P_4$};
\node[right] at (0,0) {$P_3$};
\node[left] at (-1,1.5) {$T_{in}$};
\node[below right] at (0.8,1.3) {$T_{out}$};
\begin{scope}[shift={(7,0)}]
\draw[scale=1,ultra thick,blue]  (-1,3) -- (-1,0);
\draw(-1,3)--(1.3,3);
\draw[scale=1,ultra thick,red] (-1,0)--(1,2)--(1.3,3);
\node[left] at (-1,3) {$P'_2$};
\node[left] at (-1,0) {$P'_3$};
\node[right] at (1,2) {$P'_4$};
\node[right] at (1.3,3) {$P'_1$};
\node[left] at (-1,1.5) {$T'_{out}$};
\node[below right] at (0.5,1.3) {$T'_{in}$};
\end{scope}
\end{tikzpicture}
\caption{(SC'), Case 1, $\mu_2$, (i). (Left: projection to $xy$-plane. Right:  projection to $xy$-plane followed by reflection about $x$-axis.)}\label{fig:SCcase1mu2i}
\end{figure}
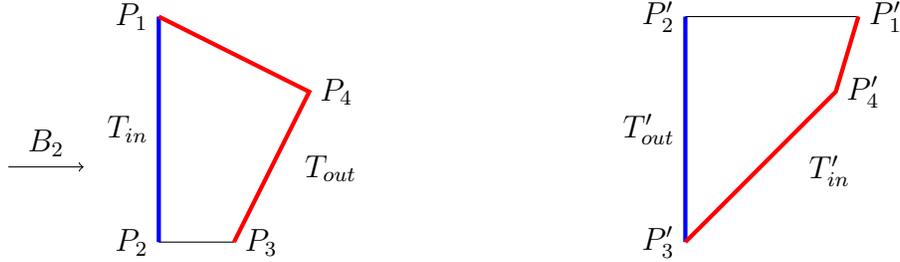

We claim that $x'[{\bf d}']$ also satisfies (SC'). If $d'_2<0$, then ${\bf d}'=(0,-1,0)$, and (SC') is trivially true. So we assume $d'_2\ge0$. The condition (2) of  Lemma \ref{quadrilateral} determines the vector ${\bf d}'=(d'_1,d'_2,d'_3)$:
\[
\begin{array}
 {rcl}
d_2'&=&-\min(d_2-a'd_1,d_2,d_2-bd_3)\\&=&a'd_1-d_2, \textup{(because of the assumption $-a'd_1+bd_3<0$)}
\\
d_3'&=&-\min(-d_3-cd_1+b'd_2,-d_3+(a'b'-c)d_1,-d_3)=d_3,
\\
d_1'&=&-\min((aa'-1)d_1-ad_2,-d_1,-d_1+(ab-c')d_3)
\\&=&-\min(-d_1+ad'_2,-d_1,-d_1+(ab-c')d_3)=d_1.
\end{array}\]
We show that the three conditions in  Lemma \ref{quadrilateral} (1) hold for $(i,j,k)=(2,1,3)$:

(1a)  $b'_{21}=a'\ge0$ and $b'_{13}=ab-c'\ge0$.

(1b) Use $\vv'_4=(aa'd_1-ad_2-(ab-c')d_3,-a'd_1+bd_3,-cd_1+b'd_2)$.

(1c)  $B'_2, B'_1, B'_3$ are not strictly in the same half plane, and $\vv_4'=\lambda_2 B_2'+\lambda_3 B_3'$ with $\lambda_2=(-cd_1+b'd_2)/b'\ge0$, $\lambda_3=(-a'd_1+bd_3)/(-b)>0$. So $B'_2, B'_1, B'_3, \vv'_4$ are in circular order by Lemma \ref{criterion of circular order}.
\medskip

(ii) Suppose $-a'd_1+bd_3>0$. 
 Geometrically, it means that $P_4$ is strictly higher than $P_1$ after projection to $xy$-plane; see Figure \ref{fig:SCcase1mu2ii}. So $T_{in}=P_4P_1P_2$ and $T_{out}=P_3P_4$. Therefore Lemma~\ref{lemma:support change} implies

$$\aligned
&{P'_1=\alpha_2(P_1)=(-d_1,d_2-a'd_1,-d_3+(a'b'-c)d_1),}\\
&{P'_2=\alpha_2(P_2)=\beta_2(P_3)=(-d_1,d_2,-d_3),}\\
&{P'_3=\beta_2(P_4)=(-d_1+(ab-c')d_3,d_2-bd_3,-d_3),}\\
&{P'_4=\alpha_2(P_4)=(-d_1+ad_2-c'd_3,d_2-bd_3,-b'd_2+(bb'-1)d_3).}
\endaligned
$$ 

\begin{figure}[ht]
\begin{tikzpicture}[scale=.8]
\draw [->](-3,1) -- (-2,1);\node[above] at (-2.5,1) {$B_2$};
\draw[scale=1,ultra thick,blue]  (1,3)--(-1,2) -- (-1,0);
\draw(-1,0)--(0,0);
\draw[scale=1,ultra thick,red] (0,0)--(1,3);
\node[left] at (-1,2) {$P_1$};
\node[left] at (-1,0) {$P_2$};
\node[right] at (1,3) {$P_4$};
\node[right] at (0,0) {$P_3$};
\node[left] at (-1,1) {$T_{in}$};
\node[right] at (0.5,1.5) {$T_{out}$};
\begin{scope}[shift={(7,0)}]
\draw[scale=1,ultra thick,blue]  (-1,0)--(-1,2)--(0,3);
\draw(0,3)--(2.5,3);
\draw[scale=1,ultra thick,red] (-1,0)--(2.5,3);
\node[left] at (-.2,3) {$P'_4$};
\node[left] at (-1,0) {$P'_2$};
\node[left] at (-1,2) {$P'_1$};
\node[right] at (2.5,3) {$P'_3$};
\node[left] at (-1,1) {$T'_{out}$};
\node[below right] at (0.5,1.3) {$T'_{in}$};
\end{scope}
\end{tikzpicture}
\caption{(SC'), Case 1, $\mu_2$, (ii). (Left: projection to $xy$-plane. Right:  projection to $xy$-plane followed by reflection about $x$-axis.)}\label{fig:SCcase1mu2ii}
\end{figure}
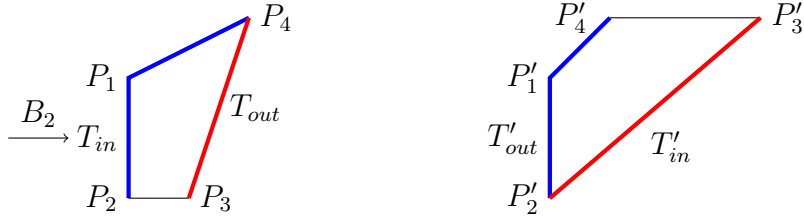

We claim that $x'[{\bf d}']$ also satisfies (SC'). Like before, assume $d'_2\ge0$. The condition (2) of  Lemma \ref{quadrilateral} determines the vector ${\bf d}'=(d'_1,d'_2,d'_3)$:

$d_2'=-\min(d_2-a'd_1,d_2,d_2-bd_3)=bd_3-d_2$, (because of the assumption $-a'd_1+bd_3>0$)

$d_1'=-\min(-d_1,-d_1+(ab-c')d_3,-d_1+ad_2-c'd_3)=d_1$,

$d_3'=-\min(-d_3+(a'b'-c)d_1,-d_3,-b'd_2+(bb'-1)d_3)=-\min(-d_3+(a'b'-c)d_1,-d_3,-d_3+b'd'_2=d_3$.

We show that the three conditions in  Lemma \ref{quadrilateral} (1) hold for $(i,j,k)=(1,3,2)$:

(1a)  $b'_{13}=ab-c'\ge0$ and $b'_{32}=b'\ge0$.

(1b) Use  $\vv'_4=(-ad_2+c'd_3,bd_3-a'd_1,(a'b'-c)d_1+b'd_2-bb'd_3)$.

(1c)  $B'_1, B'_3, B'_2$ are not in the same half plane, and $\vv_4'=\lambda_1 B_1'+\lambda_2 B_2'$ with $\lambda_1=(bd_3-a'd_1)/a'>0$, $\lambda_2=(-ad_2+c'd_3)/(-a)\ge0$. So $B'_1, B'_3, B'_2, \vv'_4$ are in circular order by Lemma \ref{criterion of circular order}.
\medskip

(iii) Suppose $-a'd_1+bd_3=0$.  Geometrically, it means that $P_4$ is at the same height at $P_1$ after projection to $xy$-plane; see Figure \ref{fig:SCcase1mu2iii}. So $T_{in}=P_1P_2$ and $T_{out}=P_3P_4$.

The Newton polytope of $x'[{\bf d}']$ is in a triangle $P'_1P'_2P'_3$, determined by $\alpha_2(P_1P_2)=P'_1P'_2$ and  $\beta_2(P_3P_4)=P'_2P'_3$. We can view this as a degenerate case of either (i) or (ii), and the proof of (SC')  still works. Note that these two ${\bf P_d}$ gives the same triangle convex hull $|{\bf P_d}|$.

\begin{figure}[ht]
\begin{tikzpicture}[scale=.8]
\draw [->](-3,1) -- (-2,1);\node[above] at (-2.5,1) {$B_2$};
\draw[scale=1,ultra thick,blue]  (-1,2) -- (-1,0);
\draw(-1,0)--(0,0) (1,2)--(-1,2);
\draw[scale=1,ultra thick,red] (0,0)--(1,2);
\node[left] at (-1,2) {$P_1$};
\node[left] at (-1,0) {$P_2$};
\node[right] at (1,2) {$P_4$};
\node[right] at (0,0) {$P_3$};
\node[left] at (-1,1) {$T_{in}$};
\node[right] at (0.5,1) {$T_{out}$};
\begin{scope}[shift={(4,0)}]
\draw[scale=1,ultra thick,blue]  (-1,0)--(-1,2);
\draw(-1,2)--(2,2);
\draw[scale=1,ultra thick,red] (-1,0)--(2,2);
\node[above] at (0,2) {$P'_1$};\draw[black,fill] (0,2) circle (.5ex);
\node[left] at (-1,0) {$P'_3$};
\node[left] at (-1,2) {$P'_2$};
\node[right] at (2,2) {$P'_4$};
\node[left] at (-1,1) {$T'_{out}$};
\node[below right] at (.8,1.3) {$T'_{in}$};
\node[below] at (0.5,-.5) {As a degenerate of (i)};
\end{scope}
\begin{scope}[shift={(9,0)}]
\draw[scale=1,ultra thick,blue]  (-1,0)--(-1,2);
\draw(-1,2)--(2,2);
\draw[scale=1,ultra thick,red] (-1,0)--(2,2);
\node[above] at (0.5,2) {$P'_4$};\draw[black,fill] (0.5,2) circle (.5ex);
\node[left] at (-1,0) {$P'_2$};
\node[left] at (-1,2) {$P'_1$};
\node[right] at (2,2) {$P'_3$};
\node[left] at (-1,1) {$T'_{out}$};
\node[below right] at (.8,1.3) {$T'_{in}$};
\node[below] at (0.5,-.5) {As a degenerate of (ii)};
\end{scope}
\end{tikzpicture}
\caption{(SC'), Case 1, $\mu_2$, (iii). (Left: projection to $xy$-plane. Middle and Right: projection to $xy$-plane followed by reflection about $x$-axis.)}\label{fig:SCcase1mu2iii}
\end{figure}

\medskip

\noindent\underline{To show that the quadrilateral changes as expected under the mutation $\mu_3$:}

Substitute $x_3$ by $(p_3^++p_1^-x_2^bx_1^{-c'})/x_3$ in $x[{\bf d}]$, and get $x'[{\bf d}']$. 
$$ 
B':=\mu_3(B)=[b'_{ij}]=\begin{bmatrix}0&a&c'\\-a'&0&-b\\-c&b'&0\end{bmatrix}
$$
Applying  Lemma \ref{change of support} to $T_{in}=P_1P_2P_3$, $T_{out}=P_1P_4$ , we obtain the quadrilateral $P'_1P'_2P'_3P'_4$, determined by $\alpha_3(P_1P_2P_3)=P'_2P'_3P'_4$ and $\beta_3(P_1P_4)=P'_1P'_4$. See Figure \ref{fig:SCcase1mu3}. Thus
$$\aligned
&{P'_1=\beta_3(P_1)=(-d_1-c'(-cd_1+b'd_2-d_3),-d_2+a'd_1+b(-d_3-cd_1+b'd_2)},\\
&\quad\quad {d_3+cd_1-b'd_2),}\\
&{P'_2=\alpha_3(P_1)=(-d_1,-d_2+a'd_1,d_3+cd_1-b'd_2),}\\
&{P'_3=\alpha_3(P_2)=(-d_1,-d_2,d_3-b'd_2),}\\
&{P'_4=\alpha_3(P_3)=\beta_3(P_4)=(-d_1+ad_2,-d_2,d_3).}
\endaligned
$$

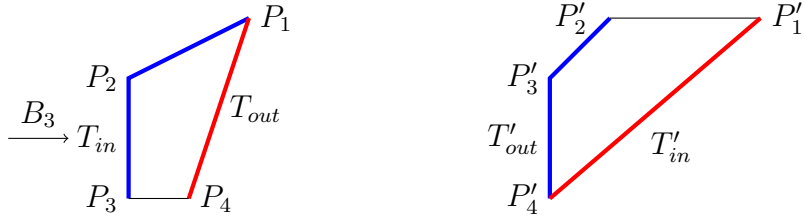
\begin{figure}[ht]
\begin{tikzpicture}[scale=.8]
\draw [->](-3,1) -- (-2,1);\node[above] at (-2.5,1) {$B_3$};
\draw[scale=1,ultra thick,blue]  (1,3)--(-1,2) -- (-1,0);
\draw(-1,0)--(0,0);
\draw[scale=1,ultra thick,red] (0,0)--(1,3);
\node[left] at (-1,2) {$P_2$};
\node[left] at (-1,0) {$P_3$};
\node[right] at (1,3) {$P_1$};
\node[right] at (0,0) {$P_4$};
\node[left] at (-1,1) {$T_{in}$};
\node[right] at (0.5,1.5) {$T_{out}$};
\begin{scope}[shift={(7,0)}]
\draw[scale=1,ultra thick,blue]  (-1,0)--(-1,2)--(0,3);
\draw(0,3)--(2.5,3);
\draw[scale=1,ultra thick,red] (-1,0)--(2.5,3);
\node[left] at (-.2,3) {$P'_2$};
\node[left] at (-1,0) {$P'_4$};
\node[left] at (-1,2) {$P'_3$};
\node[right] at (2.5,3) {$P'_1$};
\node[left] at (-1,1) {$T'_{out}$};
\node[below right] at (0.5,1.3) {$T'_{in}$};
\end{scope}
\end{tikzpicture}
\caption{(SC'), Case 1, $\mu_3$. (Left: projection to $yz$-plane. Right:  projection to $yz$-plane followed by reflection about $y$-axis.)}\label{fig:SCcase1mu3}
\end{figure}

To show $x'[{\bf d}']$ satisfies (SC'). We can assume $d'_3\ge0$ as before. Then
 $d'_1=d_1$, $d'_2=d_2$, $d'_3=-cd_1+b'd_2-d_3$, $\vv'_4=(-c'(-cd_1+b'd_2-d_3)-ad_2,a'd_1+b(-d_3-cd_1+b'd_2),cd_1-b'd_2)$.
We claim the conditions in  Lemma \ref{quadrilateral} (1) hold for $(i,j,k)=(3,1,2)$. Indeed:

For (1a): $b'_{31}=-c\ge0$, $b'_{12}=a>0$.

For (1b): straightforward check.

For (1c): $B'_3, B'_1, B'_2$ are strictly in the same half-plane, and $B'_1=(a'/b)B'_3+(-c/b')B'_2$ where both coefficients are nonnegative. 

\subsubsection{Proof of (SC') Case 2}  Assume $Q$ is of the form $1\to2\to3$, that is, $a,b>0$ and $c=0$.

This is a degenerated case of Case 1. We shall only explain the difference in the argument. 

For $\mu_1$: assume $d'_1\ge0$. The vectors $B_1$ and $B_3$ are in opposite direction, so $P_1P_2$ is parallel to $P_3P_4$. The point $P'_3$ is on the line segment $P_2'P_4'$, so $|{\bf P_d}|$ is the triangle $P_1'P_2'P_4'$. The proof is same as Case 1; the circular order condition (1c) trivially holds (where $(i,j,k)=(2,3,1)$) because $B'_1$ and $B'_3$ are in the same direction. See Figure \ref{fig:SCcase2mu1}.

\begin{figure}[ht]
\begin{tikzpicture}
\draw [->](-3,2) -- (-3,1);\node[left] at (-3,1.5) {$B_1$};
\draw (-1,3) -- (-1,0) (1,0)--(1,2);
\draw[scale=1,ultra thick,blue] (-1,3)--(1,2);
\draw[scale=1,ultra thick,red] (-1,0)--(1,0);
\node[left] at (-1,3) {$P_1$};
\node[left] at (-1,0) {$P_2$};
\node[right] at (1,2) {$P_4$};
\node[right] at (1,0) {$P_3$};
\node[above] at (0,2.5) {$T_{in}$};
\node[above] at (0,0) {$T_{out}$};
\begin{scope}[shift={(7,0)}]
\draw (1,3) -- (1,0);
\draw[scale=1,ultra thick,blue] (-1,0)--(1,3);
\draw[scale=1,ultra thick,red] (-1,0)--(1,0);
\node[left] at (-1,0) {$P'_1$};
\node[right] at (1,3) {$P'_4$};
\node[right] at (1,1) {$P'_3$}; \draw[black,fill] (1,1) circle (.5ex);
\node[right] at (1,0) {$P'_2$};
\node[above left] at (0,1.5) {$T'_{out}$};
\node[above] at (0,0) {$T'_{in}$};
\end{scope}
\end{tikzpicture}
\caption{(SC'), Case 2, $\mu_1$. (Left: projection to $xy$-plane. Right: projection to $xy$-plane followed by reflection about $y$-axis.)}\label{fig:SCcase2mu1}
\end{figure}
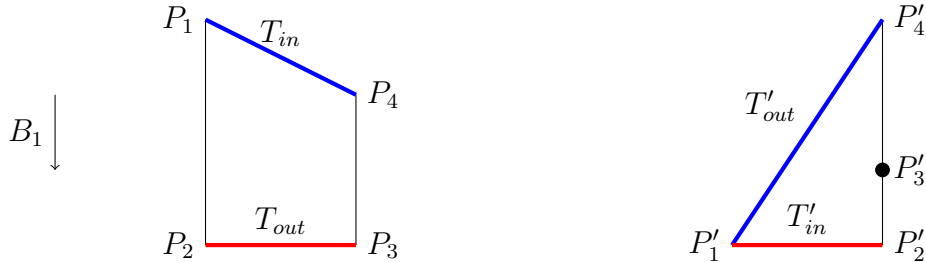

For $\mu_2$: same argument as in Case 1.

For $\mu_3$: assume $d'_3\ge0$. The point $P'_2$ is on the line segment $P_1'P_3'$, so $|{\bf P_d}|$ is the triangle $P_1'P_3'P_4'$. The proof is same as Case 1; the circular order condition trivially holds (where $(i,j,k)=(3,1,2)$) because $B'_1$ and $B'_3$ are in the same direction. See Figure \ref{fig:SCcase2mu3}.

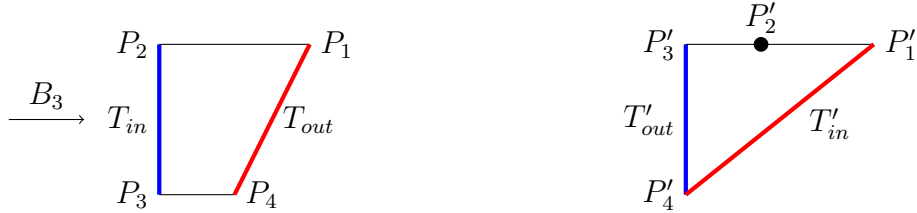
\begin{figure}[ht]
\begin{tikzpicture}
\draw [->](-3,1) -- (-2,1);\node[above] at (-2.5,1) {$B_3$};
\draw[scale=1,ultra thick,blue] (-1,2) -- (-1,0);
\draw(-1,0)--(0,0) (1,2)--(-1,2);
\draw[scale=1,ultra thick,red] (0,0)--(1,2);
\node[left] at (-1,2) {$P_2$};
\node[left] at (-1,0) {$P_3$};
\node[right] at (1,2) {$P_1$};
\node[right] at (0,0) {$P_4$};
\node[left] at (-1,1) {$T_{in}$};
\node[right] at (0.5,1) {$T_{out}$};
\begin{scope}[shift={(7,0)}]
\draw[scale=1,ultra thick,blue]  (-1,0)--(-1,2);
\draw(-1,2)--(1.5,2);
\draw[scale=1,ultra thick,red] (-1,0)--(1.5,2);
\node[above] at (0,2) {$P'_2$};\draw[black,fill] (0,2) circle (.5ex);
\node[left] at (-1,0) {$P'_4$};
\node[left] at (-1,2) {$P'_3$};
\node[right] at (1.5,2) {$P'_1$};
\node[left] at (-1,1) {$T'_{out}$};
\node[below right] at (0.5,1.3) {$T'_{in}$};
\end{scope}
\end{tikzpicture}
\caption{(SC'), Case 2, $\mu_3$. (Left: projection to $yz$-plane. Right:  projection to $yz$-plane followed by reflection about $y$-axis.)}\label{fig:SCcase2mu3}
\end{figure}

\subsubsection{Proof of (SC') Case 3} (a) Assume $Q$ is of the form $1\to 3\leftarrow 2$, that is, $a=0$, $b>0>c$. Thus $B_1$ and $B_2$ are in the same direction. 

This is degenerated from Case 1. We shall explain the difference. 

For $\mu_1$: Note $B'_1=(0,0,-c)$, $B'_2=(0,0,-b')$, $B'_3=(c',b,0)$, $(i,j,k)=(2,3,1)$. To show that $B'_2, B'_3, B'_1, \vv'_4$ are in circular order, we use the fact that $B'_2$ and $B'_1$ are in opposite directions. See Figure \ref{fig:SCcase3mu1}.

\begin{figure}[ht]
\begin{tikzpicture}
\draw [->](-3,2) -- (-3,1);\node[left] at (-3,1.5) {$B_1$};
\draw (-1,3) -- (-1,0) ;
\draw[scale=1,ultra thick,blue] (-1,3)--(1,0);
\draw[scale=1,ultra thick,red] (-1,0)--(1,0);
\node[left] at (-1,3) {$P_1$};
\node[left] at (-1,1) {$P_2$}; \draw[black,fill] (-1,1) circle (.5ex);
\node[right] at (1,0) {$P_4$};
\node[left] at (-1,0) {$P_3$};
\node[above] at (0,2) {$T_{in}$};
\node[above] at (0,0) {$T_{out}$};
\begin{scope}[shift={(7,0)}]
\draw (1,3) -- (1,0) (-1,0)--(-1,1);
\draw[scale=1,ultra thick,blue] (-1,1)--(1,3);
\draw[scale=1,ultra thick,red] (-1,0)--(1,0);
\node[left] at (-1,0) {$P'_2$};
\node[right] at (1,3) {$P'_4$};
\node[left] at (-1,1) {$P'_1$}; 
\node[right] at (1,0) {$P'_3$};
\node[above left] at (0,1.8) {$T'_{out}$};
\node[above] at (0,0) {$T'_{in}$};
\end{scope}
\end{tikzpicture}
\caption{(SC'), Case 3, $\mu_1$ or $\mu_2$. (Left: projection to $xz$-plane. Right: projection to $xz$-plane followed by reflection about $z$-axis.)}\label{fig:SCcase3mu1}
\end{figure}
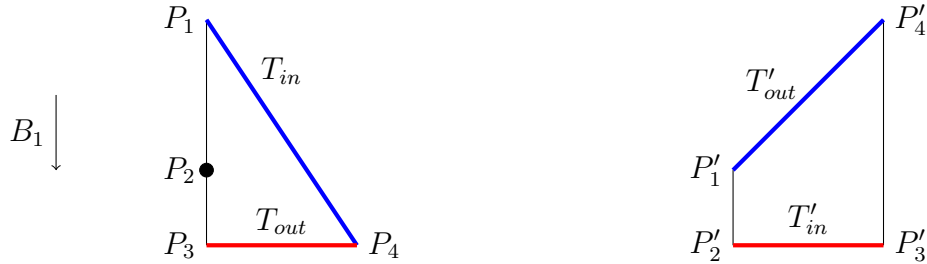

For $\mu_2$: Note $B'_1=(0,0,c)$, $B'_2=(0,0,b')$, $B'_3=(-c',-b,0)$. The three cases described in (Case 1, $\mu_2$) degenerate to:

(i): If $bd_3<0$. Impossible since we assume $d_3\ge0$.

(ii) and (iii): If $bd_3\ge 0$. Take $(i,j,k)=(1,3,2)$. Vectors $B'_1,B'_3,B'_2,\vv'_4$ are in circular order because $B'_1, B'_2$ are in opposite directions. Also see Figure \ref{fig:SCcase3mu1}. 

\medskip

For $\mu_3$: This is the same argument as Case 1.

\bigskip

(b)  Assume $Q$ is of the form $2\leftarrow 1\to 3$, that is,  $a>0$, $b=0$, $c<0$. 
 Then $B_2$ and $B_3$ are in the same direction.

This case is also degenerated from Case 1. We shall explain the difference. 

For $\mu_1$: This is the same argument as in Case 1.

For $\mu_2$: As before, assume $d'_2=bd_3-d_2=-d_2\ge0$. Since we assume $d_2\ge0$, we must have $d'_2=d_2=0$.  Note $B'_1=(0,a',c)$, $B'_2=(-a,0,0)$, $B'_3=(-c',0,0)$. The three cases described in (Case 1, $\mu_2$) degenerate to: 

(i) and (iii): If $-a'd_1\le 0$. Then take $(i,j,k)=(2,1,3)$. Vectors $B'_2,B'_1,B'_3,\vv'_4$ are in circular order because $B'_2, B'_3$ are in opposite directions. 

(ii): If $-a'd_1>0$. Then $d_1<0$, ${\bf d}=(-1,0,0)$ and $x[{\bf d}]=x_1$, which is a trivial case.

\begin{figure}[ht]
\begin{tikzpicture}[scale=.8]
\draw [->](-4,1) -- (-3,1);\node[left] at (-3,1.5) {$B_2$};
\draw (-1,0) -- (2,0);
\draw[scale=1,ultra thick,blue] (-1,2)--(-1,0);
\draw[scale=1,ultra thick,red] (-1,2)--(2,0);
\node[left] at (-1,2) {$P_1$};
\node[left] at (-1,0) {$P_2$}; 
\node[below] at (0,0) {$P_3$}; \draw[black,fill] (0,0) circle (.5ex);
\node[right] at (2,0) {$P_4$};
\node[left] at (-1,1) {$T_{in}$};
\node[above right] at (.6,0.8) {$T_{out}$};
\begin{scope}[shift={(7,0)}]
\draw (-1,0) -- (0,0) (-1,2)--(2,2);
\draw[scale=1,ultra thick,blue] (-1,2)--(-1,0);
\draw[scale=1,ultra thick,red] (0,0)--(2,2);
\node[left] at (-1,2) {$P'_2$};
\node[left] at (-1,0) {$P'_3$}; 
\node[below] at (0,0) {$P'_4$}; \draw[black,fill] (0,0) circle (.5ex);
\node[right] at (2,2) {$P'_1$};
\node[left] at (-1,1) {$T'_{out}$};
\node[right] at (1,0.8) {$T'_{in}$};
\end{scope}
\end{tikzpicture}
\caption{(SC'), Case 4, $\mu_2$ or $\mu_3$. (Left: projection to $xz$-plane. Right: projection to $xz$-plane followed by reflection about $z$-axis.)}\label{fig:SCcase4mu2}
\end{figure}
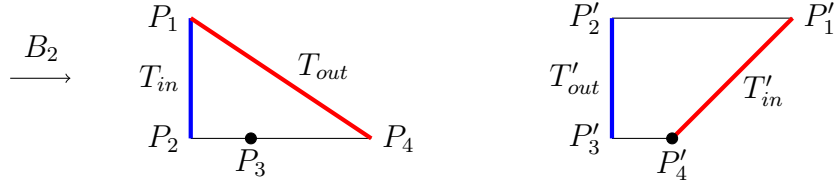

For $\mu_3$: Take $(i,j,k)=(3,1,2)$. Vectors $B'_3, B'_1, B'_2,\vv'_4$ are in circular order because $B'_3=(c',0,0)$ and $B'_2=(a,0,0)$ are in opposite directions.

\subsubsection{Proof of (SC') Case 4} Suppose  $a,b,c>0$. Renumbering vertices $(1,2,3)$ as $(2,3,1)$ or $(3,1,2)$ if necessary, we can assume that $B_1, B_2, B_3,\vv_4$ are in circular order (and still satisfy $b_{12},b_{23},b_{31}>0$); thus
\begin{equation}\label{case 5 circular implies}
(-ad_2+c'd_3,a'd_1-bd_3,-cd_1+b'd_2)=\vv_4=\lambda_1 B_1+\lambda_3 B_3=(-c'\lambda_3,-a'\lambda_1+b\lambda_3, c\lambda_1)
\end{equation}
for some real numbers  $\lambda_1,\lambda_3\ge 0$.
So by Lemma \ref{quadrilateral}, we get the same expression for $P_1,\dots,P_4$ as in Case 1:
 $$\aligned
 &{P_1=(-d_1,-d_2+a'd_1,-d_3-cd_1+b'd_2),}\\
 &{P_2=(-d_1,-d_2,-d_3+b'd_2),}\\
 &{P_3=(-d_1+ad_2,-d_2,-d_3),}\\
 &{P_4=(-d_1+ad_2-c'd_3,-d_2+bd_3,-d_3).}
 \endaligned
 $$
It is easy to check that ${\bf d}=-\minvec(P_1,\dots,P_4)$ by observing $ad_2-c'd_3=c'\lambda_3\ge0$ and $-cd_1+b'd_2=c\lambda_1\ge0$.

\medskip
\noindent\underline{To show that the quadrilateral changes as expected under the mutation $\mu_1$:}

If $\mu_1(Q)$ is acyclic, then we can apply the previous argument for $\mu_1(Q)$ to conclude that the quadrilateral is compatible with the mutation. So  in below we assume that $\mu_1(Q)$ is still cyclic, i.e., $ac-b'>0$, or equivalently, $a'c'-b>0$. 

By changing the initial seed from $\Sigma_{t_0}$ to $\mu_1(\Sigma_{t_0})$, we substitute $x_1$ by $(p_1^+x_2^{a'}+p_1^-x_3^c)/x_1$ in $x[{\bf d}]$, and get $x'[{\bf d}']$. Using Lemma \ref{change of support}, we obtain that the support of $x'[{\bf d}']$ lies in the quadrilateral $P'_1P'_2P'_3P'_4$, determined by $\alpha_1(P_1P_4P_3)=P'_1P'_4P'_3$ and $\beta_1(P_2P_3)=P'_1P'_2$. Thus 
$P'_1=\alpha_1(P_1)=\beta_1(P_2), P'_2=\beta_1(P_3), P'_3=\alpha_1(P_3),P'_4=\alpha_1(P_4)$, and more explicitly
$$\begin{array}{rclcr}
P'_1&=&(d_1,&-d_2,&-d_3+b'd_2-cd_1),\\
P'_2&=&(d_1-ad_2,&-d_2,&-d_3-cd_1+acd_2),\\
P'_3&=&(d_1-ad_2,&(aa'-1)d_2-a'd_1,&-d_3),\\
P'_4&=&(d_1-ad_2+c'd_3,&-a'd_1+(aa'-1)d_2+(b-a'c')d_3,&-d_3).
\end{array}
$$ 
See Figure \ref{fig:SCcase5mu1}. 

\begin{figure}[ht]
\begin{tikzpicture}
\draw [->](-3,2) -- (-3,1);\node[left] at (-3,1.5) {$B_1$};
\draw (-1,3) -- (-1,0);
\draw[scale=1,ultra thick,blue] (-1,3)--(0,2)--(1,0);
\draw[scale=1,ultra thick,red] (-1,0)--(1,0);
\node[left] at (-1,3) {$P_1$};
\node[left] at (-1,0) {$P_2$};
\node[right] at (0,2) {$P_4$};
\node[right] at (1,0) {$P_3$};
\node[above] at (0,2.5) {$T_{in}$};
\node[above] at (0,0) {$T_{out}$};
\begin{scope}[shift={(7,0)}]
\draw (1,3) -- (1,0);
\draw[scale=1,ultra thick,blue] (-1,0)--(0,2)--(1,3);
\draw[scale=1,ultra thick,red] (-1,0)--(1,0);
\node[left] at (-1,0) {$P'_1$};
\node[right] at (1,3) {$P'_3$};
\node[left] at (0,2) {$P'_4$};
\node[right] at (1,0) {$P'_2$};
\node[above left] at (-.5,1) {$T'_{out}$};
\node[above] at (0,0) {$T'_{in}$};
\end{scope}
\end{tikzpicture}
\caption{(SC'), Case 4, $\mu_1$. (Left: projection to $xy$-plane. Right: projection to $xy$-plane followed by reflection about $y$-axis.)}\label{fig:SCcase5mu1}
\end{figure}
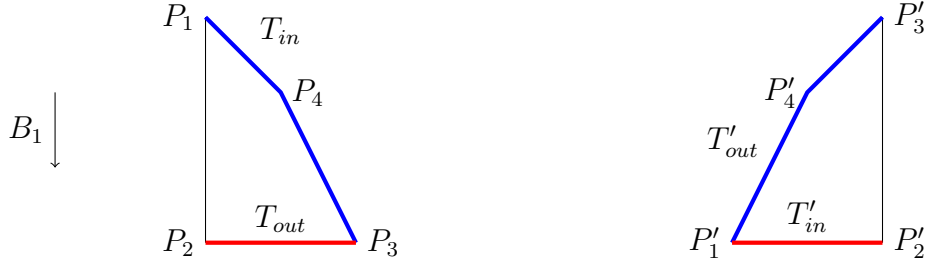

We claim that this $x'[{\bf d}']$ also satisfies (SC). If $d'_1<0$, it is trivially true. So we assume $d'_1\ge0$.
Denote
$$ 
B'=\mu_1(B)=\begin{bmatrix}0&-a&c'\\a'&0&b-a'c'\\-c&ac-b'&0\end{bmatrix}
$$
By \eqref{min formular for d-vector}, we have ${\bf d}=\minvec(P_1',P_2',P_3',P_4')$, thus
 $d'_1=ad_2-d_1$, $d'_2=d_2$, $d'_3=d_3$, because $d_1'=ad_2=d_1\ge 0$. The vector $\vv'_4$ is equal to $\vv'_4=-\vv'_1-\vv'_2-\vv'_3=(ad_2-c'd_3, a'd_1-aa'd_2+(a'c'-b)d_3,b'd_2-cd_1)$.

We show that the conditions in  Lemma \ref{quadrilateral} (1) are all satisfied for $(i,j,k)=(2,1,3)$:

For (1a): We have $b'_{21}=a'\ge0$ and $b'_{13}=c'\ge0$.

For (1b): This is straightforward.

For (1c): We have $B'_2, B'_1, B'_3, \vv'_4$ are in circular order because $B'_2, B'_1, B'_3$ are not in the same half plane, and $\vv'_4=(b'd_2-cd_1)/(ac-b')B'_2-(a'd_1-aa'd_2+(a'c'-b)d_3)/(a'c'-b)B'_3$ where the first coefficient $=c\lambda_1/(ac-b')\ge0$, the second coefficient $=(a'\lambda_1+(a'c'-b)\lambda_3)/(a'c'-b)\ge0$. 
 
\medskip

\noindent\underline{To show that the quadrilateral changes as expected under the mutation $\mu_2$:}

Like above, we can assume $\mu_2(Q)$ is cyclic, i.e. $ab-c'>0$, $a'b'-c>0$.
Define
$$ 
B'=[b'_{ij}]=\begin{bmatrix}0&-a&ab-c'\\a'&0&-b\\c-a'b'&b'&0\end{bmatrix}
$$
 There are three cases to consider:

(i) Suppose $-a'd_1+bd_3<0$. See Figure \ref{fig:SCcase5mu2i}.

\begin{figure}[ht]
\begin{tikzpicture}
\draw [->](-3,1) -- (-2,1);\node[above] at (-2.5,1) {$B_2$};
\draw[scale=1,ultra thick,blue]  (-1,3) -- (-1,0);
\draw(-1,0)--(2,0);
\draw[scale=1,ultra thick,red] (2,0)--(1,2)--(-1,3);
\node[left] at (-1,3) {$P_1$};
\node[left] at (-1,0) {$P_2$};
\node[right] at (1,2) {$P_4$};
\node[right] at (2,0) {$P_3$};
\node[left] at (-1,1.5) {$T_{in}$};
\node[right] at (1.4,1.1) {$T_{out}$};
\begin{scope}[shift={(7,0)}]
\draw[scale=1,ultra thick,blue]  (-1,3) -- (-1,0);
\draw(-1,3)--(1.3,3);
\draw[scale=1,ultra thick,red] (-1,0)--(1,2)--(1.3,3);
\node[left] at (-1,3) {$P'_2$};
\node[left] at (-1,0) {$P'_3$};
\node[right] at (1,2) {$P'_4$};
\node[right] at (1.3,3) {$P'_1$};
\node[left] at (-1,1.5) {$T'_{out}$};
\node[below right] at (0.5,1.3) {$T'_{in}$};
\end{scope}
\end{tikzpicture}
\caption{(SC'), Case 4, $\mu_2$, (i). (Left: projection to $xy$-plane. Right:  projection to $xy$-plane followed by reflection about $x$-axis.)}\label{fig:SCcase5mu2i}
\end{figure}
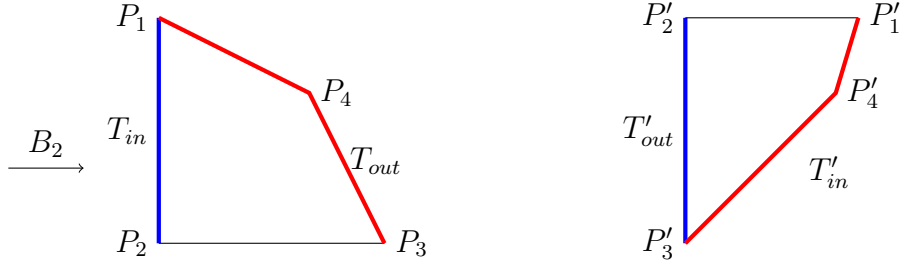

The quadrilateral ${\bf P_d}$ is determined by $\alpha_2(P_1P_2)=P'_2P'_3$ and  $\beta_2(P_1P_4P_3)=P'_1P'_4P'_3$. Thus
$$\aligned
&{P'_1=\beta_2(P_1)=((aa'-1)d_1-ad_2,d_2-a'd_1,-d_3-cd_1+b'd_2),}\\
&{P'_2=\alpha_2(P_1)=(-d_1,d_2-a'd_1,-d_3+(a'b'-c)d_1),}\\
&{P'_3=\alpha_2(P_2)=\beta_2(P_3)=(-d_1,d_2,-d_3),}\\
&{P'_4=\beta_2(P_4)=(-d_1+(ab-c')d_3,d_2-bd_3,-d_3).}
\endaligned
$$ 
We claim that $x'[{\bf d}']$ also satisfies (SC'). Like before, assume $d'_2\ge0$.
First compute:

$d'_1=d_1$, 
 
$d'_2=a'd_1-d_2$ (because of the assumption $-a'd_1+bd_3<0$), 

$d'_3=d_3$ (recall $-cd_1+b'd_2=c\lambda_1\ge0$).

$\vv'_4=(aa'd_1-ad_2-(ab-c')d_3,-a'd_1+bd_3,-cd_1+b'd_2)$.

\noindent Then show that the conditions in  Lemma \ref{quadrilateral} (1) hold for $(i,j,k)=(2,1,3)$:

(1a) We have $b'_{21}=a\ge0$ and $b'_{13}=ab-c\ge0$.

(1b) This is straightforward.

(1c) We have $B'_2, B'_1, B'_3$ are not strictly in the same half plane, and $\vv_4'=\lambda'_2 B_2'+\lambda'_2 B_3'$ with $\lambda'_2=(-cd_1+b'd_2)/b'=c\lambda_1/b\ge0$, $\lambda'_3=(a'd_1-bd_3)/b>0$, so $B'_2, B'_1, B'_3, \vv'_4$ are in circular order. 
\medskip

(ii) Suppose $-a'd_1+bd_3>0$.  See Figure \ref{fig:SCcase5mu2ii}.

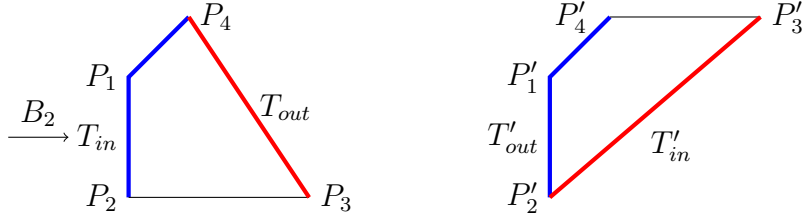
\begin{figure}[ht]
\begin{tikzpicture}[scale=.8]
\draw [->](-3,1) -- (-2,1);\node[above] at (-2.5,1) {$B_2$};
\draw[scale=1,ultra thick,blue]  (0,3)--(-1,2) -- (-1,0);
\draw(-1,0)--(2,0);
\draw[scale=1,ultra thick,red] (2,0)--(0,3);
\node[left] at (-1,2) {$P_1$};
\node[left] at (-1,0) {$P_2$};
\node[right] at (0,3) {$P_4$};
\node[right] at (2,0) {$P_3$};
\node[left] at (-1,1) {$T_{in}$};
\node[right] at (1,1.5) {$T_{out}$};
\begin{scope}[shift={(7,0)}]
\draw[scale=1,ultra thick,blue]  (-1,0)--(-1,2)--(0,3);
\draw(0,3)--(2.5,3);
\draw[scale=1,ultra thick,red] (-1,0)--(2.5,3);
\node[left] at (-.2,3) {$P'_4$};
\node[left] at (-1,0) {$P'_2$};
\node[left] at (-1,2) {$P'_1$};
\node[right] at (2.5,3) {$P'_3$};
\node[left] at (-1,1) {$T'_{out}$};
\node[below right] at (0.5,1.3) {$T'_{in}$};
\end{scope}
\end{tikzpicture}
\caption{(SC'), Case 4, $\mu_2$, (ii). (Left: projection to $xy$-plane. Right:  projection to $xy$-plane followed by reflection about $x$-axis.)}\label{fig:SCcase5mu2ii}
\end{figure}

The quadrilateral ${\bf P_d}$ is determined by $\alpha_2(P_4P_1P_2)=P'_4P'_1P'_2$ and  $\beta_2(P_3P_4)=P'_2P'_3$. Thus
$$\aligned
&{P'_1=\alpha_2(P_1)=(-d_1,d_2-a'd_1,-d_3+(a'b'-c)d_1),}\\
&{P'_2=\alpha_2(P_2)=\beta_2(P_3)=(-d_1,d_2,-d_3),}\\
&{P'_3=\beta_2(P_4)=(-d_1+(ab-c')d_3,d_2-bd_3,-d_3),}\\
&{P'_4=\alpha_2(P_4)=(-d_1+ad_2-c'd_3,d_2-bd_3,-b'd_2+(bb'-1)d_3).}
\endaligned
$$ 
We claim that this $x'[{\bf d}']$ also satisfies (SC). Like before, assume $d'_2\ge0$. We have
\[\begin{array} {rcl}
d'_1&=&-\min(-d_1,-d_1+(ab-c')d_3,-d_1+ad_2-c'd_3)\\&=&d_1 \qquad\textup{ because $ab-c'>0$ and $ad_2-c'd_3=c'\lambda_3\ge0$}, 
 \\
d'_2&=&-\min(d_2-a'd_1,d_2,d_2-bd_3)=bd_3-d_2 \qquad\textup{ because $a'd_1<bd_3$}, 
\\
d'_3&=&-\min(-d_3+(a'b'-c)d_1,-d_3,-b'd_2+(bb'-1)d_3)\\
&=&-\min(-d_3+(a'b'-c)d_1,-d_3,b'd'_2-d_3)
\\&=&d_3 \qquad \textup{ because $a'b'-c>0$ and $d'_2\ge0$},\\ 
\vv'_4&=&(-ad_2+c'd_3,bd_3-a'd_1,(a'b'-c)d_1+b'd_2-bb'd_3).
\end{array}\]
\noindent To show that the conditions in  Lemma \ref{quadrilateral} (1) hold for $(i,j,k)=(1,3,2)$, the only nontrivial condition is (1c) $B'_1, B'_3, B'_2, \vv'_4$ are in circular order. To see this, note that $B'_1, B'_3, B'_2$ are not in the same half plane, and $\vv_4'=\lambda'_1 B_1'+\lambda'_2 B_2'$ with $\lambda'_1=(bd_3-a'd_1)/a'>0$, $\lambda'_2=(ad_2-c'd_3)/a'=c'\lambda_3/a'\ge0$.

\medskip
(iii) Suppose $-a'd_1+bd_3=0$. See Figure \ref{fig:SCcase5mu2iii}.  The Newton polytope of $x'[{\bf d}']$ is a triangle. We can view this as a degenerate case of either (i) or (ii), and the proof of (SC') therein still holds.

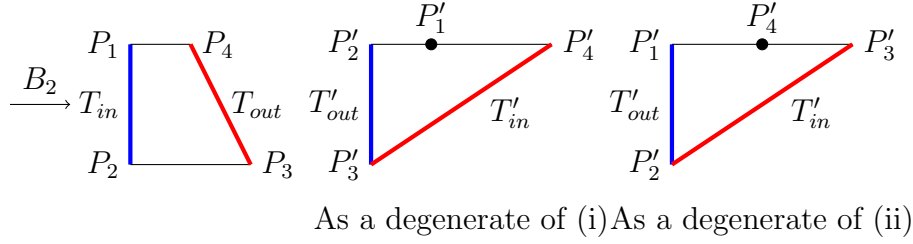
\begin{figure}[ht]
\begin{tikzpicture}[scale=.8]
\draw [->](-3,1) -- (-2,1);\node[above] at (-2.5,1) {$B_2$};
\draw[scale=1,ultra thick,blue]  (-1,2) -- (-1,0);
\draw(-1,0)--(1,0) (0,2)--(-1,2);
\draw[scale=1,ultra thick,red] (1,0)--(0,2);
\node[left] at (-1,2) {$P_1$};
\node[left] at (-1,0) {$P_2$};
\node[right] at (0,2) {$P_4$};
\node[right] at (1,0) {$P_3$};
\node[left] at (-1,1) {$T_{in}$};
\node[right] at (0.5,1) {$T_{out}$};
\begin{scope}[shift={(4,0)}]
\draw[scale=1,ultra thick,blue]  (-1,0)--(-1,2);
\draw(-1,2)--(2,2);
\draw[scale=1,ultra thick,red] (-1,0)--(2,2);
\node[above] at (0,2) {$P'_1$};\draw[black,fill] (0,2) circle (.5ex);
\node[left] at (-1,0) {$P'_3$};
\node[left] at (-1,2) {$P'_2$};
\node[right] at (2,2) {$P'_4$};
\node[left] at (-1,1) {$T'_{out}$};
\node[below right] at (.8,1.3) {$T'_{in}$};
\node[below] at (0.5,-.5) {As a degenerate of (i)};
\end{scope}
\begin{scope}[shift={(9,0)}]
\draw[scale=1,ultra thick,blue]  (-1,0)--(-1,2);
\draw(-1,2)--(2,2);
\draw[scale=1,ultra thick,red] (-1,0)--(2,2);
\node[above] at (0.5,2) {$P'_4$};\draw[black,fill] (0.5,2) circle (.5ex);
\node[left] at (-1,0) {$P'_2$};
\node[left] at (-1,2) {$P'_1$};
\node[right] at (2,2) {$P'_3$};
\node[left] at (-1,1) {$T'_{out}$};
\node[below right] at (.8,1.3) {$T'_{in}$};
\node[below] at (0.5,-.5) {As a degenerate of (ii)};
\end{scope}
\end{tikzpicture}
\caption{(SC'), Case 4, $\mu_2$, (iii). (Left: projection to $xy$-plane. Middle and Right: projection to $xy$-plane followed by reflection about $x$-axis.)}\label{fig:SCcase5mu2iii}
\end{figure}

\medskip

\noindent\underline{To show that the quadrilateral changes as expected under the mutation $\mu_3$:}

Assume $\mu_3(Q)$ is cyclic, i.e. $bc-a'>0$, $b'c'-a>0$. See Figure  \ref{fig:SCcase5mu3}. We have
$$ 
B':=\mu_3(B)=\begin{bmatrix}0&a-b'c'&c'\\bc-a'&0&-b\\-c&b'&0\end{bmatrix}
$$
The quadrilateral ${\bf P_d}$ is determined by $\alpha_3(P_2P_3)=P'_2P'_3$ and $\beta_3(P_2P_1P_4)=P'_1P'_4P'_3$. See Figure \ref{fig:SCcase5mu3}. Thus
$$\aligned
&{P'_1=\beta_3(P_1)=(-d_1,(a'-bc)d_1+(bb'-1)d_2-bd_3,cd_1-b'd_2+d_3),}\\
&{P'_2=\beta_3(P_2)=(-d_1,(bb'-1)d_2-bd_3,-b'd_2+d_3),}\\
&{P'_3=\alpha_3(P_2)=(-d_1+b'c'd_2-c'd_3,-d_2,-b'd_2+d_3),}\\
&{P'_4=\alpha_3(P_3)=\beta_3(P_4)=(-d_1+ad_2-c'd_3,-d_2,d_3).}\\
\endaligned
$$

\begin{figure}[ht]
\begin{tikzpicture}
\draw [->](-3,1) -- (-2,1);\node[above] at (-2.5,1) {$B_3$};
\draw[scale=1,ultra thick,blue]  (-1,2) -- (-1,0);
\draw(-1,0)--(0,0);
\draw[scale=1,ultra thick,red] (0,0)--(1,1)--(-1,2);
\node[left] at (-1,2) {$P_2$};
\node[left] at (-1,0) {$P_3$};
\node[right] at (1,1) {$P_1$};
\node[right] at (0,0) {$P_4$};
\node[left] at (-1,1) {$T_{in}$};
\node[right] at (0.5,1.7) {$T_{out}$};
\begin{scope}[shift={(7,0)}]
\draw[scale=1,ultra thick,blue]  (-1,0)--(-1,2);
\draw(-1,2)--(1.5,2);
\draw[scale=1,ultra thick,red] (-1,0)--(1,1)--(1.5,2);
\node[below right] at (1,1) {$P'_1$};
\node[left] at (-1,0) {$P'_4$};
\node[left] at (-1,2) {$P'_2$};
\node[right] at (1.5,2) {$P'_2$};
\node[left] at (-1,1) {$T'_{out}$};
\node[left] at (0.7,1) {$T'_{in}$};
\end{scope}
\end{tikzpicture}
\caption{(SC'), Case 4, $\mu_3$. (Left: projection to $yz$-plane. Right:  projection to $yz$-plane followed by reflection about $y$-axis.)}\label{fig:SCcase5mu3}
\end{figure}
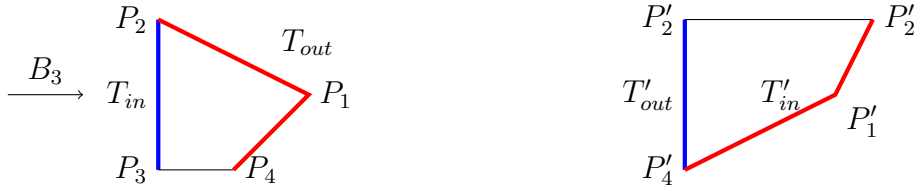

To show $x'[{\bf d}']$ satisfies (SC'). We can assume $d'_3\ge0$ as before. Compute
 
\[\begin{array}
 {rcl} 
 d'_3&=&-\min(cd_1-b'd_2+d_3, -b'd_2+d_3,d_3)=b'd_2-d_3,\\ 
 d'_2&=&-\min((a'-bc)d_1+(bb'-1)d_2-bd_3, (bb'-1)d_2-bd_3,-d_2)\\
 &=&-\min(-d_2+(bc-a')\lambda_1+b\lambda_3, bd'_3-d_2, -d_2)=d_2, 
 \\
 d'_1&=&-\min(-d_1,-d_1+b'c'd_2-c'd_3,-d_1+ad_2-c'd_3)\\
 &=&-\min(-d_1,-d_1+c'd'_3,-d_1+c'\lambda_3)=d_1, 
 \\
\vv'_4&=&(-ad_2+c'd_3,(a'-bc)d_1+bb'd_2-bd_3,cd_1-b'd_2).  
\end{array}
\]

We claim the conditions in  Lemma \ref{quadrilateral} (1) holds for $(i,j,k)=(1,3,2)$. Indeed:

For (1a): We have $b'_{13}=c'\ge0$, $b'_{32}=b'>0$.

For (1b): This is a straightforward check.

For (1c): We have $B'_1, B'_3, B'_2$ are not in the same half-plane, and $\vv_4'=\lambda'_1B'_1+\lambda'_2B'_2$ with coefficients 
$\lambda'_1=((a'-bc)d_1+bb'd_2-bd_3)/(bc-a')=((bc-a')\lambda_1+b\lambda_3)/(bc-a')\ge0$, $\lambda'_2=(-ad_2+c'd_3)/(a-b'c')=c'\lambda_3/(b'c'-a)\ge0$. 
 
This completes the proof of (SC') for all four cases.

\subsection{Proof of (NC')} 
The existence of $e({\bf p})$ with nonzero constant term follows from (NC). To show the second part of (NC'), it suffices to show that the property that $e({\bf p})$ is a monomial for each vertex of ${\bf P_d}$ is invariant under mutation. Suppose therefore that ${\bf p}\in\mathbb{Z}^3$ is a vertex of the weakly convex quadrilateral ${\bf P_d}$  of the cluster variable $x[\mathbf{d}]$ with respect to the initial seed $\Sigma_{t_0}$, and suppose that $e({\bf p})=y_1^{r_1}y_2^{r_2}y_3^{r_3}$. Let $\mathbf{P}'$ be the weakly convex quadrilateral of the same cluster variable but with respect to the seed $\Sigma_{\mu_1(t_0)}$. Thus $\mathbf{P'}$ is
 obtained from $\mathbf{P}$ by substituting $x_1$ by $(M_1+M_2)/x_1$ (substituting $x_2, x_3$ can be argued similarly). By Lemma \ref{lemma:support change}, the vertices of $|{\bf P'}|$ are obtained as either 

(a) $\alpha_1({\bf p})$, where ${\bf p}$ is a vertex of $|{\bf P_d}|$,  the intersection of the line ${\bf p}+\mathbb{R}B_1$ with the quadrilateral ${\bf P_d}$ is a line segment $|{\bf pq}|$ with ${\bf q}={\bf p}+rB_1$ ($r\ge0$), or 

(b) $\beta_1({\bf q})$, where ${\bf q}$ is a vertex of $|{\bf P_d}|$, the intersection of the line ${\bf q}+\mathbb{R}B_1$ with the quadrilateral ${\bf P_d}$ is a line segment $|{\bf pq}|$ with ${\bf p}={\bf q}-rB_1$ ($r\ge0$). 

We only need to consider (a) because (b) can be argued similarly.
We use Lemma \ref{change of support} with $f=\sum e(\mathbf{p}) x^{\mathbf{p}}$ the Laurent expansion of $x[\mathbf{d}]$ in the seed $\Sigma_{t_0}$ and $g$ its Laurent expansion in the seed $\Sigma_{\mu_1(t_0)}$.  The vertex $\mathbf{p}$ corresponds to the term $a_0x^{\mathbf{p}}$ (with $b_0=0$) in the lemma and it transforms to the new vertex $\mathbf{p}'$ yielding the term $a_0'x^{\mathbf{p}'}$ (with $b_0'=0$) in $g$. The lemma implies $a_0'=(p^+)^{p_1}a_0$ and thus
$$e'({\bf p}')=(p^+)^{p_1}e({\bf p})=\Big(\prod_i y_i^{[e_i]_+}\Big)^{p_1} e(\mathbf{p}) =\prod_i y_i^{r_i+p_1[e_i]_+}$$
which is a Laurent monomial in $y_1,y_2,y_3$, so it is also a Laurent monomial in $y_1',y_2',y_3'$. By induction on the number of mutations it follows that $e({\bf p})$ is a monomial for all vertices ${\bf p}$ of all $|{\bf P_d}|$.
This completes the proof of condition (NC') and of Theorem \ref{main theorem}.
	
\section{Example}
\label{sect 4}	
\begin{example}\label{eg:x621}
Consider $a=b=-c=2$: 
$$Q=\xymatrix{ &2\ar[dr]|2&\\1\ar[ru]|2\ar[rr]|2&&3} \qquad B=\begin{bmatrix} 0&2&2\\-2&0&2\\-2&-2&0 \end{bmatrix}
$$
Consider the cluster variable $x[6,2,1]$, obtained by the mutation sequence $\{1,2,3\}$.
The quadrilateral ${\bf P_d}$ is computed as in Lemma \ref{quadrilateral} as follows. Start with vertices
\[\begin{array}{cccccccccccccc}
\tilde P_1&=&(0,0,0)&&&\tilde P_2&=&d_1B_1&=&(0,-12,-12)\\
\tilde P_3&=&\tilde P_1+d_2B_2&=&(4,-12,-16)\qquad&\tilde P_4&=& \tilde P_3+d_3B_3&=&(6,-10,-16)
\end{array}\]
and then shift by the vector 
\[-\minvec(\tilde P_1,\tilde P_2,\tilde P_3,\tilde P_4)-{\bf d}=-(0,-12,-16)-(6,2,1)=(-6,10,15)\]
to obtain the vertices 
of $P_{\bf d}$ as follows
\[P_1=(-6,10,15), P_2=(-6,-2,3), P_3=(-2,-2,-1), P_4=(0,0,-1). \]

On the other hand, the cluster has the following Laurent expansion.
$$\begin{array}{rcl}x[6,2,1]&=& x_1^{-6}x_2^{-2}x_3^{-1}\big(x_{2}^{12} x_{3}^{16} + 6 x_{2}^{10} x_{3}^{14} y_{1} + 2 x_{1}^{2}
x_{2}^{8} x_{3}^{10} y_{1}^{2} y_{2} + 15 x_{2}^{8} x_{3}^{12} y_{1}^{2}
\\[5pt]
&&+ 8 x_{1}^{2} x_{2}^{6} x_{3}^{8} y_{1}^{3} y_{2}+ 20 x_{2}^{6}
x_{3}^{10} y_{1}^{3} + x_{1}^{4} x_{2}^{4} x_{3}^{4} y_{1}^{4} y_{2}^{2}
+ 12 x_{1}^{2} x_{2}^{4} x_{3}^{6} y_{1}^{4} y_{2} \\[5pt]
&& 
+ x_{1}^{6} x_{2}^{2}
y_{1}^{6} y_{2}^{2} y_{3}  + 15 x_{2}^{4} x_{3}^{8} y_{1}^{4}
+2x_{1}^{4} x_{2}^{2} x_{3}^{2} y_{1}^{5} y_{2}^{2} 
+ 8 x_{1}^{2}
x_{2}^{2} x_{3}^{4} y_{1}^{5} y_{2} + 6 x_{2}^{2} x_{3}^{6} y_{1}^{5}\\[5pt]
&& +
x_{1}^{4} y_{1}^{6} y_{2}^{2} + 2 x_{1}^{2} x_{3}^{2} y_{1}^{6} y_{2} +
x_{3}^{4} y_{1}^{6}\big).\end{array}$$

We project the support to 2nd and 3rd exponents of $x$ (that is, draw a point of coordinate $(i,j)$ if $x_2^ix_3^j$ appears in $x[6,2,1]$). We obtain the  picture in Figure \ref{example1}. 
\begin{figure}[ht]
  \includegraphics[width=.4\textwidth]{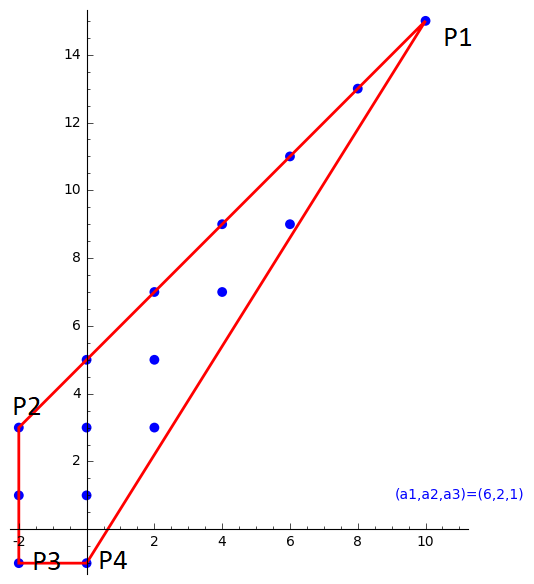}
  \caption{Support of $x[6,2,1]$. The blue dots are the support, the red polytope is the Newton polytope.
}\label{example1} 
\end{figure}

The monomials in corresponding positions are:
{\tiny
$$\begin{bmatrix}
&&&&&&x_1^{-6}x_{2}^{10} x_{3}^{15}\\
&&&&&6x_1^{-6}x_{2}^{8} x_{3}^{13}y_1\\
&&&&15x_1^{-6}x_{2}^{6} x_{3}^{11}y_1^2\\
&&&20x_1^{-6}x_2^4x_3^9y_1^3&2x_1^{-4}x_{2}^{6} x_{3}^{9}y_1^2y_1\\
&&15x_1^{-6}x_2^2x_3^7y_1^4&8x_1^{-4}x_2^4x_3^7y_1^3y_2\\
&6x_1^{-6}x_3^5y_1^5&
12x_1^{-4}x_2^2x_3^5y_1^4y_2
\\
x_1^{-6}x_2^{-2}x_3^3y_1^6&
8x_1^{-4}x_3^3y_1^5y_2&x_1^{-2}x_2^2x_3^3y_1^4y_2\\
2x_1^{-4}x_2^{-2}x_3y_1^6y_2&2x_1^{-2}x_3y_1^5y_2^2
\\
x_1^{-2}x_2^{-2}x_3^{-1}y_1^6y_2^2&x_3^{-1}y_1^6y_2^2y_3
\end{bmatrix}$$
}

Let us consider what happens if we substitute $x_1$ by its mutation $(p_1^+x_2^{a}x_3^{-c}+p_1^-)/x_1=(x_2^2x_3^2+y_1)/x_1$ in $x[6,2,1]$ (because $p_1^+=1$, $p_1^-=y_1$). The 7 terms with $x_1^{-6}$ (that is, the line segment $|P_1P_2|$) add up to be
$$f=x_1^{-6}x_2^{10}x_3^{15}
+6x_1^{-6}x_2^8x_3^{13}y_1
+\cdots
+x_1^{-6}x_2^{-2}x_3^3y_1^6=x_2^{-2}x_3^3(\frac{x_2^2x_3^2+y_1}{x_1})^6$$
So after the substitution, we get $x_1^6x_2^{-2}x_3^3$, 
which also follows in general using Lemma \ref{change of support} (1), with ${\bf p}=(-6,10,15)$, ${\bf q}=(-6,-2,3)={\bf p}+6(0,-2,-2)$: since the first and last term of $g$ in that lemma will have exponents
${\bf p}'=\alpha_1({\bf p})={\bf q}'=\beta_1({\bf q})=(6,-2,3)$, which is $P_1'$. 

In general, let $f$ be the sum of the terms containing $x_1^{-constant}$. Then after substitution the two endpoints are mapped by $\alpha_1$ and $\beta_1$, both being linear maps. Consider the line segments parallel to line $P_1P_2$ and in the quadrilateral $P_1P_2P_3P_4$. The two ends of each line segment are mapped by maps $\alpha_1$ and $\beta_1$. So the segment $P_1P_4$ (i.e. the set of right endpoints) is mapped by $\alpha_1$, $P_2P_3P_4$ (i.e. the set of left endpoints) is mapped by $\beta_1$. Their image encloses a new quadrilateral ${\bf P_d}'=P_1'P_2'P_3'P_4'$ (which actually degenerates to a triangle) with
$$
P'_1=(6,-2,3),\quad
P'_2=(2,-2,-1),\quad
P'_3=(0,0,-1),\quad
P'_4=(0,0,-1).
$$ 
\end{example}

\begin{remark}\label{remark:F-poly}
Note that the $F$-polynomial of $x[6,2,1]$ is
$1 + 6 y_{1} + 2  y_{1}^{2} y_{2} + 15  y_{1}^{2}
+ 8  y_{1}^{3} y_{2} 
+ 20 y_{1}^{3} +  y_{1}^{4} y_{2}^{2}
+ 12 y_{1}^{4} y_{2} + 
y_{1}^{6} y_{2}^{2} y_{3} + 15  y_{1}^{4}
 +2 y_{1}^{5} y_{2}^{2}
+ 8 y_{1}^{5} y_{2} + 6 y_{1}^{5} +
 y_{1}^{6} y_{2}^{2} + 2  y_{1}^{6} y_{2} +
 y_{1}^{6}$. Its Newton polytope has vertices
 $(0, 0, 0), (6, 2, 1), (6, 2, 0), (6, 0, 0), (4, 2, 0)$. In contrast, the region ${\bf F_d}$ as defined in Corollary \ref{cor of main thm} is a convex polyhedron with vertices $(0, 0, 0)$,  $(6, 0, 0)$,   $(8, 0, 2)$,   $(6,2,0)$, $(5, 3, 0)$,  $(8, 0, 3)$, which contains the Newton polytope of the $F$-polynomial as a proper subset.
\end{remark}

\section{Quantum analogue}\label{sect 5} 
In this section, we prove that Theorem \ref{main theorem} generalizes to the quantum cluster algebras introduced in \cite{BZ}.
We consider here only principal coefficients. The statement with non-principal coefficients should follow easily from this.

First we fix some notation. For a nonzero integer $\delta$, define
$[n]_\delta=(v^{\delta n}-v^{-\delta n})/(v^\delta-v^{-\delta})$ (note that $[n]_\delta=[n]_{-\delta}$)
and define the quantum binomial coefficient (where $k,n\in\mathbb{Z}$, $k\ge0$) 
$${n\brack k}_\delta=\frac{[n]_\delta [n-1]_\delta\cdots[n-k+1]_\delta}{[k]_\delta[k-1]_\delta\cdots[1]_\delta}$$
Define
$$(x+y)_\delta^n=\sum_{k\ge0} {n\brack k}_\delta x^ky^{n-k}$$
For example, $$(x+y)_5^3=y^3+(v^5+1+v^{-5})xy^2+(v^5+1+v^{-5})x^2y+y^3.$$

\begin{remark}
To see the motivation of the above definition: 
consider two quasi-commuting variables $X$, $Y$ with $YX=v^{2\delta}XY$. Denote 
$X^{(i,j)}:=v^{ij\delta}X^iY^j$. Then the above quantum binomial coefficients satisfy
$$(X+Y)^n=\sum_{k\ge0} {n\brack k}_{\delta}v^{k(n-k)\delta}X^kY^{n-k}=\sum_{k\ge0} {n\brack k}_{\delta}X^{(k,n-k)}.$$
\end{remark}
\medskip

Let $B$ be a skew-symmetrizable matrix, $D$ a positive diagonal matrix such that $DB$ is skew-symmetric. Let
$$\Lambda=\begin{bmatrix}0&-D\\D&-DB\\\end{bmatrix}, \quad \tilde{B}=\begin{bmatrix}B\\I\end{bmatrix}.$$
Recall that the based quantum torus $\mathcal{T}(\Lambda)$ is the $\mathbb{Z}[v^{\pm}]$-algebra with a distinguished $\mathbb{Z}[v^{\pm}]$-basis $\{X^e: e\in\mathbb{Z}^{2n}\}$ and the multiplication is given by
$$X^eX^f=v^{\Lambda(e,f)}X^{e+f}\quad (e,f\in\mathbb{Z}^{2n})$$
where $\Lambda(e,f)=e^T\Lambda f$.

We introduce the following convention to represent a quantum Laurent polynomial using a commutative Laurent polynomial: namely, we define a function 
$$\varphi: \mathbb{Z}[v^{\pm}][x_1,\dots,x_6]\to\mathcal{T}(\Lambda)$$
$$\sum a_{i_1,\dots,i_6}x_1^{i_1}\cdots x_6^{i_6}
\mapsto
\sum a_{i_1,\dots,i_6}X^{(i_1,\dots,i_6)}$$
Recall that we denote $y_1=x_4, y_2=x_5, y_3=x_6$. 

\medskip 

We have the following generalization of our main result.

\begin{theorem}\label{main quantum theorem}

A quantum cluster variable $x[{\bf d}]$ with d-vector ${\bf d}$ can be written as

$$x[{\bf d}]=\varphi\sum_{{\bf p}\in\mathbb{Z}^3} e({\bf p}) x^{\bf p}=\varphi\sum_{p_1,p_2,p_3} e(p_1,p_2,p_3) x_1^{p_1}x_2^{p_2}x_3^{p_3}$$ 
where $e({\bf p})\in\mathbb{Z}[v^{\pm}][y_1,y_2,y_3]$
is uniquely characterized by the following conditions:
\begin{itemize}
\item[(SC)] (Support condition) 
The coefficient $e({\bf p})=0$ unless ${\bf p}\in {\bf P_d}$. Equivalently, the Newton polytope of $x[{\bf d}]$ is contained in ${\bf P_d}$.

\item[(NC)] (Normalization condition) There is only one $e({\bf p})$ which has a nonzero constant term, which must be 1. Moreover, the greatest common divisor of all $e({\bf p})$ is 1.

\item[(DC)] (Divisibility condition)
For each $k=1,2,3$, if $p_k<0$, then 
$$\big(\prod_{i=1}^3 x_i^{[-b_{ik}]_+}+y_k\prod_{i=1}^3 x_i^{[b_{ik}]_+}\big)^{-p_k}_{\delta_k}\textrm{ divides }\displaystyle\sum_{p_1,\dots,\widehat{p}_k,\dots, p_3} e(p_1,p_2,p_3)x_1^{p_1}x_2^{p_2}x_3^{p_3}$$
where the notation $\widehat{p}_k$ under the sum means that we have $p_k$ fixed and the other two $p_i$ run over all integers.
\end{itemize}

Moreover, (NC) can be replaced by:

\begin{itemize}
\item[(NC')] There is a coefficient $e({\bf p})$ with nonzero constant term, and for each vertex ${\bf p}$ of the convex hull $|{\bf P_d}|$, $e({\bf p})$ is a monomial in $y_1,y_2,y_3$.
\end{itemize}

And (SC) can be replaced by a stronger condition:

\begin{itemize}
\item[(SC')]  
The Newton polytope of $x[{\bf d}]$ is equal to ${\bf P_d}$.
\end{itemize}
\end{theorem}
\begin{proof}
The proof is similar to Theorem \ref{main theorem}. The main difference is the quantum version of the divisibility condition (DC), which follows easily from Lemma \ref{lem:quantumDC} below.
\end{proof}
\begin{lemma}\label{lem:quantumDC}
Fix $1\le k\le 3$ and fix $p_k\in\mathbb{Z}$.
Let $f$ be a Laurent polynomial 
$$f=\sum_{p_1,\dots,\hat{p}_k,\dots, p_3} e(p_1,p_2,p_3)x_1^{p_1}x_2^{p_2}x_3^{p_3}$$
where $e(p_1,p_2,p_3)\in \mathbb{Z}[v^{\pm}][y_1,y_2,y_3]$.  Then $\varphi(f)\in \mathcal{T}(\Lambda)$ is a Laurent polynomial in 
$$\mathbb{Z}[v^{\pm}][X_1^{\pm},\dots,(X'_k)^{\pm},\dots,X_3^{\pm},X_4^{\pm},X_5^{\pm},X_6^{\pm}]$$ if and only if 
$$\big(\prod x_i^{[-b_{ik}]_+}+y_k\prod x_i^{[b_{ik}]_+}\big)^{[-p_k]_+}_{\delta_k}\textrm{ divides } f.$$

\end{lemma}
\begin{proof}
Without loss of generality assume $k=1$. So $p_1$ is fixed throughout the proof.

Let $\tilde{B}'=\mu_1(\tilde B)$. By definition of mutation of quantum cluster variables, 
$$X_1={X'}^{-e_1+[\tilde B_1']_+}+{X'}^{-e_1+[-\tilde B_1']_+}={X'}^{-e_1+[-\tilde B_1]_+}+{X'}^{-e_1+[\tilde B_1]_+},$$
where the second equality holds because $\tilde B_1'=-\tilde B_1$. 
We introduce the following notation: for $p,q\in\mathbb{Z}^3$, let
$$X^{{\bf p}\brack{\bf q}}=X^{(p_1,p_2,p_3,q_1,q_2,q_3)}, 
\quad
{X'}^{{\bf p}\brack{\bf q}}={X'}^{(p_1,p_2,p_3,q_1,q_2,q_3)}$$
We denote the Laurent expansion of $\varphi(f)$ in the original cluster (respectively in the cluster $\{X_1',X_2,X_3,X_4,X_5,X_6\}$) as follows
 $$\varphi(f)
=\sum_{p_2, p_3,{\bf q}} e_{\bf pq}
X^{{\bf p}\brack{\bf q}}
\qquad(\textup{respectively }\varphi(f)
=\sum_{p'_2, p'_3,{\bf q'}} e'_{\bf p'q'}
{X'}^{{\bf p'}\brack{\bf q'}}),
$$
 where ${\bf p}=(p_1,p_2,p_3)$ and ${\bf p'}=(p_1',p'_2,p'_3)$ and $p'_1=-p_1$.
So $e(p_1,p_2,p_3)=\sum_{\bf q} e_{\bf pq}X^{{\bf p}\brack{\bf q}}$.
For convenience of notation, we let $\bar{\bf p}=(0,p_2,p_3)$, and similar for $\bar{\bf p'}$. 
Then for fixed $p_1$, 
$$\sum_{\bf p,q} e_{\bf pq}v^{-\Lambda({{\bar{\bf p}}\brack{\bf q}},\, p_1e_1)}X^{{\bar{\bf p}}\brack{\bf q}}X_1^{p_1}=
 \sum_{\bf p',q'} e'_{\bf p'q'}{X'}^{{\bf p'}\brack{\bf q'}}$$

\medskip
Now we prove the lemma in two cases: $p_1<0$ and $p_1\ge0$.
\medskip

\noindent (Case 1) $p_1<0$. We shall show that
\begin{equation}\label{epq case1}
e_{\bf pq}=\sum_k
e'_{\bf p'q'}{-p_1\brack k}_{\delta_1}
\end{equation}
Indeed,
$$\aligned
&{\sum_{\bf p,q} e_{\bf pq}v^{-\Lambda({{\bar{\bf p}}\brack{\bf q}},\, p_1e_1)}X^{{\bar{\bf p}}\brack{\bf q}}=
 \sum_{\bf p',q'} e'_{\bf p'q'}{X'}^{{\bf p'}\brack{\bf q'}}X_1^{-p_1}}\\
&{=
 \sum_{\bf p',q'} e'_{\bf p'q'}{X'}^{{\bf p'}\brack{\bf q'}}
({X'}^{-e_1+[-\tilde B_1]_+}+{X'}^{-e_1+[\tilde B_1]_+})^{-p_1}}\\
&{\stackrel{*}{=}
 \sum_{\bf p',q'} e'_{\bf p'q'}{X'}^{{\bf p'}\brack{\bf q'}}
\sum_k {-p_1\brack k}_{\delta_1}
{X'}^{p_1e_1+k[-\tilde B_1]_+ + (-p_1-k)[\tilde B_1]_+}}\\
&{=\hspace{-5pt}
 \sum_{{\bf p',q'},k} e'_{\bf p'q'}{-p_1\brack k}_{\delta_1}
\hspace{-5pt}v^{\Lambda'({\bar{\bf p}\brack {\bf q}},p_1e_1+k[-\tilde B_1]_+ + (-p_1-k)[\tilde B_1]_+)}
{X'}^{{{\bf p'}\brack{\bf q'}}+p_1e_1+k[-\tilde B_1]_+ + (-p_1-k)[\tilde B_1]_+}}\\
&{\stackrel{**}{=}
 \sum_{{\bf p',q'},k} e'_{\bf p'q'}{-p_1\brack k}_{\delta_1}
v^{\Lambda'({\bar{\bf p}\brack {\bf q}},p_1e_1+k[-\tilde B_1]_+ + (-p_1-k)[\tilde B_1]_+)}
{X}^{{{\bar{\bf p'}}\brack{\bf q'}}+k[-\tilde B_1]_+ + (-p_1-k)[\tilde B_1]_+}}\\
\endaligned$$
(The equality  ``$\stackrel{*}{=}$'' is because
$${X'}^{-e_1+[-\tilde B_1]_+}{X'}^{-e_1+[\tilde B_1]_+}=v^{2\Lambda'(-e_1+[-\tilde B_1]_+,-e_1+[\tilde B_1]_+)}{X'}^{-e_1+[\tilde B_1]_+}{X'}^{-e_1+[-\tilde B_1]_+}$$
where $\Lambda'(-e_1+[-\tilde B_1]_+,-e_1+[\tilde B_1]_+)=\Lambda'(-e_1+[-\tilde B_1]_+,\tilde B_1)=\Lambda(e_1,\tilde B_1)=-\delta_1$. The equality ``$\stackrel{**}{=}$'' is because the exponent of $X'$ has zero in the first coordinate, so we can replace $X'$ by $X$). 
Now comparing the coefficients of $X^{{\bf p}\brack{\bf q}}$ on both sides, we get
\begin{equation}\label{eq:epq}
e_{\bf pq}v^{-\Lambda({{\bar{\bf p}}\brack{\bf q}},\, p_1e_1)}
=\sum_k
e'_{\bf p'q'}{-p_1\brack k}_{\delta_1}
v^{\Lambda'({\bar{\bf p}\brack {\bf q}},p_1e_1+k[-\tilde B_1]_+ + (-p_1-k)[\tilde B_1]_+)}
\end{equation}
where ${\bf p'}$ and ${\bf q'}$ are determined by
$${{\bar{\bf p'}}\brack{\bf q'}}+k[-\tilde B_1]_+ + (-p_1-k)[\tilde B_1]_+={{\bar{\bf p}}\brack{\bf q}}$$
which can be rewritten as
\begin{equation}\label{eq:relation pq p'q'}
{{\bar{\bf p'}}\brack{\bf q'}}+ (-p_1)[-\tilde B_1]_+(-p_1-k)\tilde{B}_1={{\bar{\bf p}}\brack{\bf q}}
\end{equation} 
We claim that the exponents of $v$ on both sides of \eqref{eq:epq} are equal. Indeed,
$$
\aligned
&{\Lambda'({\bar{\bf p}\brack {\bf q}},p_1e_1+k[-\tilde B_1]_+ + (-p_1-k)[\tilde B_1]_+)}\\
&={\Lambda({\bar{\bf p}\brack {\bf q}},-p_1e_1+p_1[\tilde B_1]_+ +k[-\tilde B_1]_+ + (-p_1-k)[\tilde B_1]_+)}\\
&={\Lambda({\bar{\bf p}\brack {\bf q}},-p_1e_1 +k[-\tilde B_1]_+  -k[\tilde B_1]_+)}
=-{\Lambda({\bar{\bf p}\brack {\bf q}},p_1e_1)}-k{\Lambda({\bar{\bf p}\brack {\bf q}},\tilde B_1)}
\\
\endaligned
$$
Moreover, $\Lambda({\bar{\bf p}\brack {\bf q}},\tilde B_1)$ is the 1st coordinate of the following row vector, so is equal to 0:
$$\aligned
{\bar{\bf p}\brack {\bf q}}^T\Lambda \tilde B
&={\bar{\bf p}\brack {\bf q}}^T\begin{bmatrix}0&-D\\D&-DB\end{bmatrix} \begin{bmatrix}B\\I\end{bmatrix}
={\bar{\bf p}\brack {\bf q}}^T\begin{bmatrix}-D\\0\end{bmatrix}=-\bar{\bf p}^TD\\
&=-\begin{bmatrix}0&p_2&p_3\end{bmatrix}\begin{bmatrix}\delta_1&0&0\\0&\delta_2&0\\0&0&\delta_3\end{bmatrix}=-\begin{bmatrix}0&p_2\delta_2&p_3\delta_3\end{bmatrix}
\endaligned
$$
Thus we can cancel out the exponents of $v$ on both sides of \eqref{eq:epq}, and obtain \eqref{epq case1}.

The lemma then follows easily from  \eqref{epq case1}:
$$
\aligned
f&=\sum_{\bf p,q} e_{\bf pq}x^{{\bf p}\brack{\bf q}}
=\sum_{{\bf p},{\bf q},k}
e'_{\bf p'q'}{-p_1\brack k}_{\delta_1}x^{{\bf p}\brack{\bf q}} \quad (\textrm{which satisfies \eqref{eq:relation pq p'q'}})\\
&=\sum_{{\bf p}',{\bf q}',k}
e'_{\bf p'q'}{-p_1\brack k}_{\delta_1}x^{{{\bar{\bf p'}}\brack{\bf q'}}+p_1e_1+ (-p_1)[-\tilde B_1]_+(-p_1-k)\tilde{B}_1}\\
&
=\sum_{{\bf p}',{\bf q}'}
e'_{\bf p'q'}x^{{{\bar{\bf p'}}\brack{\bf q'}}+p_1e_1+ (-p_1)[-\tilde B_1]_+}
\sum_k{-p_1\brack k}_{\delta_1}x^{(-p_1-k)\tilde{B}_1}\\
&=\sum_{{\bf p}',{\bf q}'}
e'_{\bf p'q'}x^{{{\bar{\bf p'}}\brack{\bf q'}}+p_1e_1+ (-p_1)[\tilde B_1]_+}
\big(1+x^{\tilde{B}_1}\big)^{-p_1}_{\delta_1}\\
&=\sum_{{\bf p}',{\bf q}'}
e'_{\bf p'q'}x^{{{\bar{\bf p'}}\brack{\bf q'}}+p_1e_1+ (-p_1)[\tilde B_1]_+}
\prod x_i^{-[-b_{i1}]_+}\big(\prod x_i^{[-b_{i1}]_+}+y_1\prod x_i^{[b_{i1}]_+}\big)^{-p_1}_{\delta_1}\\
\endaligned
$$ 
therefore $f$ is divisible by $\big(\prod x_i^{[-b_{i1}]_+}+y_1\prod x_i^{[b_{i1}]_+}\big)^{-p_1}_{\delta_1}$ if and only if finitely many $e'_{\bf p'q'}$ are nonzero.
\medskip

\noindent (Case 2) $p_1\ge0$. Similar to above, we have
$$e'_{\bf p'q'}=\sum_k
e_{\bf pq}{p_1\brack k}_{\delta_1}
$$
So it is always true that only finitely many $e'_{\bf p'q'}$ are nonzero. Meanwhile, the divisibility condition becomes ``1 divides $f$'' which is also always true.
\end{proof}

\begin{example}
The coefficients of the quantum cluster variable $x[6,2,1]$ corresponding to Example \ref{eg:x621} are shown in the following matrix, where $[n]:=[n]_1$:
$$\begin{bmatrix}
&&&&&&1\\
&&&&&[6]\\
&&&&\frac{[6][5]}{[2]}\\
&&&\frac{[6][5][4]}{[3][2]}&[2]\\
&&\frac{[6][5]}{[2]}&[2][4]\\
&[6]&
[3][4]
\\
1&
[2][4]&1\\
[2]&[2]
\\
1&1
\end{bmatrix}$$
\end{example}

\medskip
\noindent\textbf{Acknowledgment.} 
The authors wish to thank the referee for providing many valuable suggestions, particularly
for suggesting of a reference of Lemma \ref{C-matrix det}.

\nocite{*}
\bibliographystyle{amsplain-ac}
\bibliography{LLR}
\end{document}